\documentclass[12pt,a4paper,intlimits,sumlimits]{amsart}

\usepackage{amsmath, amsthm, mathtools}
\usepackage{amsfonts}
\usepackage[alphabetic]{amsrefs}
\usepackage{graphicx}
\usepackage{framed}
\usepackage{amssymb}
\usepackage{esvect}
\usepackage{xcolor}
\usepackage{eucal}

\usepackage{hyperref}
\usepackage[english]{babel}

\def \N {\mathbb{N}}
\def \R {\mathbb{R}}

\theoremstyle{definition}
\newtheorem{definition}{Definition}[section]
\newtheorem{example}[definition]{Example}
\newtheorem{remark}[definition]{Remark}

\theoremstyle{plain}
\newtheorem{theorem}[definition]{Theorem}
\newtheorem{proposition}[definition]{Proposition}
\newtheorem{lemma}[definition]{Lemma}
\newtheorem{corollary}[definition]{Corollary}

\numberwithin{equation}{section}

\usepackage{geometry}
\renewcommand{\epsilon}{\varepsilon}
\newcommand{\e}{\varepsilon}

\renewcommand{\leq}{\leqslant}
\renewcommand{\le}{\leqslant}

\renewcommand{\ge}{\geqslant}

\allowdisplaybreaks

 \title[Homoclinic solutions for nonlocal equations]{Homoclinic solutions for nonlocal equations
and applications to the theory of atom dislocation}
 
 \author{Serena Dipierro, Caterina Sportelli, Enrico Valdinoci}

\address{Serena Dipierro: Department of Mathematics and Statistics, The University of Western Australia, 35 Stirling Highway, Crawley, Perth, WA 6009, Australia}
\email{serena.dipierro@uwa.edu.au}

\address{Caterina Sportelli: Departamento de Análisis Matemático,  Universidad de Granada, 18071 Granada, Spain  \newline\indent \&  \newline\indent Department of Mathematics and Statistics, The University of Western Australia, 35 Stirling Highway, Crawley, Perth, WA 6009, Australia}
\email{caterina.sportelli@uwa.edu.au,  caterina.sp@ugr.es}

\address{Enrico Valdinoci: Department of Mathematics and Statistics, The University of Western Australia, 35 Stirling Highway, Crawley, Perth, WA 6009, Australia}
\email{enrico.valdinoci@uwa.edu.au}

 \date{}

\makeindex

\begin{document}

 \maketitle

\begin{abstract}
We establish the existence of homoclinic solutions
for suitable systems of nonlocal equations whose forcing term is of gradient type.

The elliptic operator under consideration is the fractional Laplacian and
the potentials that we take into account are of two types:
the first one is a spatially homogeneous function with a strict local maximum at the origin,
the second one is a spatially inhomogeneous potential satisfying the Ambrosetti--Rabinowitz condition coupled to a quadratic term with spatially dependent growth at infinity.

The existence of these special solutions has interesting consequences for the
theory of atomic edge dislocations in crystals according to the Peierls--Nabarro model and its generalization to fractional equations.

Specifically, for the first type of potentials, the results obtained
give the existence of a crystal configuration with atoms located at both extrema
in an unstable rest position, up to an arbitrarily small modification of the structural potential and a ``pinch'' of a particle at any given position.

For the second type of potentials, the results obtained also entail
the existence of a crystal configuration reaching an equilibrium at infinity,
up to an arbitrarily small superquadratic perturbation of the classical Peierls--Nabarro potential.
\end{abstract}

\tableofcontents
  
\section{Introduction} 
In this paper, we construct homoclinic solutions for systems of equations
under suitable structural conditions on the potential. The results
obtained are confronted with the Peierls--Nabarro theory of crystal dislocations, providing new information on cases of interest.\medskip

{F}rom the technical point of view, the results obtained can be seen as a natural fractional counterpart of some classical discoveries by
Paul H. Rabinowitz and Kazunaga Tanaka~\cite{Rab1, Rab2}.
The proofs, however, deviate substantially from the classical techniques, which
hinged mostly on variational methods and analysis of ordinary differential equations.
Indeed, due to the lack of a convenient Hamiltonian formalism in the nonlocal setting,
an important ingredient in our analysis consists in suitable versions of the elliptic regularity theory, combined with appropriate energy estimates to locate the critical levels, thus
obtaining regularity bootstrap.\medskip

Two types of potentials will be considered. The first scenario, which will be detailed in Section~\ref{POSE1}, focuses on the case of a potential with a strict global maximum.
The second, presented in Section~\ref{POSE2}, allows the potential to have spatial dependency
and accounts for the Ambrosetti--Rabinowitz condition in combination with a quadratic term with spatially dependent growth at infinity.

\subsection{Spatially homogeneous potentials with a global maximum}\label{POSE1}

The setting that we consider here is as follows.
We take a fractional parameter~$s\in (0, 1)$
and, as customary, we define the fractional Laplacian of a function~$u:\R\to\R$ as
\begin{equation}\label{deflap58947049}
(-\Delta)^s u(x):= 2c_s\int_{\R}\frac{u(x)-u(y)}{|x-y|^{1+2s}}\,dy,
\end{equation}
where the integral is intended in the principal value sense.
The positive normalizing constant~$c_s$ is chosen in such a way that
provides consistent limits as~$s\nearrow1$ and as~$s\searrow0$, namely
\[
\lim_{s\nearrow1}(-\Delta)^s u=(-\Delta)^1 u=-\Delta u \qquad{\mbox{and}}\qquad \lim_{s\searrow0}(-\Delta)^s u=(-\Delta)^0 u=u.
\]
The factor~$2$ in~\eqref{deflap58947049} is due to a future renormalization in the energy setting, see the forthcoming formula~\eqref{4836fewfgetuie}.

Let also~$a\in [-\infty, +\infty)$
and~$b\in (-\infty, +\infty]$. With a slight abuse of terminology, we adopt the notation
\begin{equation}\label{ABUNO}
[a,b]:=\begin{dcases}
[a,b] & {\mbox{ if $a\ne-\infty$ and $b\ne+\infty$,}}\\
[a,+\infty) & {\mbox{ if $a\ne-\infty$ and $b=+\infty$,}}\\
(-\infty,b] & {\mbox{ if $a=-\infty$ and $b\ne+\infty$,}}\\
(-\infty,+\infty) & {\mbox{ if $a=-\infty$ and $b=+\infty$.}}
\end{dcases}
\end{equation}
We assume\footnote{For instance, in light of~\eqref{ABUNO},
when~$s\in (0, 1/2]$, the following cases for~$[a,b]$ are allowed: $[a,b]=[0,1]$, $[a,b]=[-1,1]$, $[a,b]=(-\infty,0]$, $[a,b]=(-\infty,1]$, $[a,b]=[0,+\infty)$, $[a,b]=[1,+\infty)$.

When~$s\in (1/2, 1)$, all these cases are allowed, and also
$[a,b]=\{0\}$, $[a,b]=\{1\}$.

The distinction between the fractional exponent ranges~$s\in (0, 1/2]$ and~$s\in (1/2, 1)$
is somewhat unavoidable, as stressed in the forthcoming Proposition~\ref{NASIK},
and it is due to the different capacity theories of fractional Sobolev spaces in dependence
of the fractional exponent~$s$.}
that~$a\le b$ if $s\in (1/2, 1)$ and that $a<b$ if $s\in (0, 1/2]$. 

Given
\begin{equation}\label{ALPHAQ0}
q_0\in C^{\alpha}([a,b], \R^n)\cap\mathcal{D}^{s, 2}([a,b], \R^n)
\quad\mbox{ with } \alpha\in (0, 1),
\end{equation}
we look for homoclinic solutions~$q:\R\to\R^n$ of the fractional Laplacian problem
\begin{equation}\label{prob1}
\begin{dcases}
(-\Delta)^s q = \nabla V(q) &\mbox{ in } \R\setminus [a, b],\\
q=q_0 &\mbox{ in } [a, b].
\end{dcases}
\end{equation}
In this notation, we have that~$n\in\N$ with~$ n\ge 1$:
in particular, problem~\eqref{prob1} corresponds to a system of equations for $n\ge 2$ and to a single equation for $n=1$. 

Moreover, we suppose that the potential $V\in C^1(\R^n, \R)$ satisfies
\begin{equation} \label{V1}
V(q) < V(0) \quad\mbox{ for all } q\in\R^n\setminus\{0\}.
\end{equation}
Without loss of generality, we assume throughout the paper that $V(0)=0$. 

Furthermore, given~$R>0$ sufficiently large, we let $T_R:\R\to\R$ denote the cutoff function defined as
\begin{equation}\label{cut}
T_R (r) :=\begin{cases}
r &\mbox{ if } r\in (-R, R),\\
R &\mbox{ if } r\ge R,\\
-R &\mbox{ if } r\le - R.
\end{cases}
\end{equation}
With this notation, in addition to~\eqref{V1}, we assume also that\footnote{For example,
a function which validates~\eqref{V1} and~\eqref{V2} is~$V(q):= -|q|^2$.}
\begin{equation}\label{V2}
V(q_1, \dots, q_n) \le V(T_R(q_1), \dots, T_R(q_n))\quad \mbox{ for all } q=(q_1,\cdots,q_n)\in\R^n\setminus (-R, R)^n.
\end{equation}

Our goal is to prove existence of homoclinic solutions to problem~\eqref{prob1}, according to the following result:

\begin{theorem}\label{TH1}
There exists a solution $q$ to problem~\eqref{prob1} which is continuous and satisfies
\[
\lim_{x\to\pm\infty} q(x)=0.
\]
\end{theorem}

The solution claimed in Theorem~\ref{TH1} is obtained by performing a variational argument. More precisely, we introduce the set
\begin{equation}\label{gamma}
\Gamma: = \left\lbrace q\in\mathcal{D}^{s, 2}(\R, \R^n): q =q_0 \mbox{ in } [a, b]\right\rbrace
\end{equation}
and we address the minimization problem of the energy functional $I:\Gamma\to\R$ defined as
\begin{equation}\label{functional}
I(q):= \frac12 [q]^2_s -\int_{\R} V(q(x))\, dx.
\end{equation}
Here above and in the rest of this paper we denote the Gagliardo seminorm of~$q$ by
\begin{equation}\label{4836fewfgetuie}
[q]_s:= \left(c_s\iint_{\R^2} \frac{|q(x)-q(y)|^2}{|x-y|^{1+2s}} \, dx\, dy \right)^{1/2}.
\end{equation}

We emphasize that Theorem~\ref{TH1} is a consequence of the next results
(namely Theorems~\ref{main1} and~\ref{THAGG}), in which we also illustrate the basic properties of the minimizer:

\begin{theorem}\label{main1} Let~$s\in(0,1)$.
Let $a\in [-\infty, +\infty)$ and $b\in (-\infty, +\infty]$ and assume that $a<b$.

Let~$\alpha\in (0, 1)$ and $q_0\in C^{\alpha}([a, b], \R^n)\cap\mathcal{D}^{s, 2}([a,b], \R^n)$.  Let $\overline{\beta}\in (0, s)$ and
\begin{equation}\label{beta}
\beta:=\begin{cases}
\min\{\alpha, s\} &\mbox{ if } \alpha\neq s,\\
\overline{\beta} &\mbox{ if } \alpha = s.
\end{cases}
\end{equation}

Then, there exists a solution $q$ to problem~\eqref{prob1} such that~$q\in C^{\beta}(\R, \R^n)$ and
\begin{equation}\label{QCBETA2}
\|q\|_{C^{\beta}(\R, \R^n)}\le C,
\end{equation}
for some $C>0$ depending only on $n, s, \alpha, V$, $a$, $b$ and $q_0$.

Moreover, $I(q)<+\infty$ and
\begin{equation}\label{minimo}
I(q) \le I(\widetilde q) \quad\mbox{ for any } \widetilde q\in\Gamma.
\end{equation}

Furthermore, 
\begin{equation}\label{limite}
\lim_{x\to\pm\infty} q(x)=0.
\end{equation}

Also, if $V\in C^{1+\gamma}(\R^n, \R)$ for some $\gamma\in (\max\{0,1-2s\}, 1)$, then 
\begin{equation}\label{dot_lim}
\lim_{x\to\pm\infty} \dot{q}(x)=0.
\end{equation}

Additionally, if $a=-b$ and $q_0$ is even, then there exists an even solution~$q$.  
\end{theorem}

\begin{theorem}\label{THAGG}
Let~$s\in (1/2, 1)$ and~$q_0\in \R$.  

Then, there exists a solution $q$ to problem~\eqref{prob1} with~$a=b$ such that~$q\in C^{s-\frac12}(\R, \R^n)$
and
\begin{equation}\label{BOH2}
\|q\|_{C^{s-\frac12}(\R, \R^n)}\le C,
\end{equation}
for a positive constant $C$ depending only on $n$, $s$, $V$ and~$q_0$.

Moreover, $I(q)<+\infty$ and~\eqref{minimo} and~\eqref{limite} hold true.

Also,
if $V\in C^{1+\gamma}(\R, \R^n)$ for some~$\gamma\in (\max\{0,1-2s\}, 1)$, then~\eqref{dot_lim} is satisfied as well.

Furthermore, $q$ is symmetric with respect to~$a$, namely~$q(a-x)=q(a+x)$ for all~$x\in\R$.
\end{theorem}

A concrete application of Theorem~\ref{TH1} (or of Theorems~\ref{main1}
and~\ref{THAGG}
in their more precise formulations) arises in the theory of atom dislocation in crystals according to the Peierls--Nabarro model and its generalization to fractional equations.

In this setting, $q:\R\to\R$ represents the atom dislocation along the slip line of an edged-deformed crystal (see~\cite{MR4531940} and Section~2 in~\cite{MR3296170} for a physical explanation from basic principles). Typically, the equation~$\sqrt{-\Delta} \,q = \nabla V (q)$ (or, more generally, $(-\Delta)^s q = \nabla V (q)$ with~$s\in(0,1)$) accounts for the balance between atom mutual attractive interactions (described by the fractional Laplace operator) and the structure of the crystal in the large (expressed by the potential~$V$) which favors special configurations (corresponding to the minima, or more generally to the critical points, of~$V$).

In this context, equation~\eqref{prob1} looks at crystal configurations in which the atoms at both sides of infinity lie in a rest position ($q=0$, corresponding, in light of~\eqref{V1}, to a maximum of the potential~$V$). While in general one cannot expect nontrivial configurations of this type (see the forthcoming Example~\ref{EXAMPLE1}), Theorems~\ref{TH1},
\ref{main1} and~\ref{THAGG}
provide sufficient conditions on the physical modifications of the crystal which produce these kinds of structures. Namely, first of all, the potential~$V$ has to select the equilibrium at~$q=0$ somewhat as a special configuration (namely, a strict global maximum, in view of assumption~\eqref{V1}, with the additional decay\footnote{For instance, an arbitrarily small and low-frequency perturbation of classical periodic potentials used for atom dislocation in crystals leads to potentials \label{FOOKSHDJ} satisfying~\eqref{V1} and~\eqref{V2} (with the choice~$R:=2\pi$), e.g.
$$ V(q)=\cos q-1+\varepsilon\left( e^{-\delta |q|^2}-1\right),$$
for every~$\varepsilon$, $\delta>0$.

{F}rom the physical point of view, it is reasonable that enhancing the maximality properties of the potential value~$q=0$ enhances the chances of having orbits emanating from it (or approaching it), making it an unstable equilibrium for the system.}
property in~\eqref{V2}) and the atom needs to be ``pinched'' away from this
equilibrium at least at a given point~$a$ if~$s\in(1/2,1)$, or at a given interval~$[a,b]$ if~$s\in(0,1/2]$ (interestingly, the point~$a$ and the interval~$[a,b]$ can be arbitrarily prescribed, and the interval can be also arbitrarily small). 

We also recall that constructing solutions with specific behaviors at infinity is typically useful as a first step to construct rather
complicated, and even chaotic, patterns, see~\cite{MR3594365}.

\subsection{Spatially dependent potentials satisfying
the Ambrosetti--Rabinowitz condition coupled to a strong quadratic term}\label{POSE2}

We are now concerned with the existence of homoclinic solutions for some spatial-dependent nonlocal systems. To be more precise, for $s\in (0, 1)$, we look for homoclinic solutions~$q:\R\to\R^n$ of the system of equations given by
\begin{equation}\label{eqn2}
(-\Delta)^s q(x) + L(x)q(x)= \nabla_q W(x, q(x))  \qquad {\mbox{for all }} x\in\R.
\end{equation}
Here, we assume that $L\in C(\R, \R^{n\times n})$ is a symmetric and uniformly positive definite matrix for all $x\in\R$,  that is
\begin{equation}\label{LMAT}
\mbox{there exists $\alpha>0$ such that } L(x) q\cdot q\ge\alpha |q|^2 \mbox{ for any } q\in\R^n.
\end{equation}
Additionally, we require that
\begin{equation}\label{AUTOVAL}\lim_{x\to \pm\infty}
\inf_{|\xi|=1} L(x) \xi\cdot\xi =+\infty.
\end{equation}
Moreover, we consider a potential $W\in C^1(\R\times\R^n, \R)$ such that
\begin{eqnarray}\label{W0IN0}
&&W(x, 0) =0\\
\label{LITTLEO}
{\mbox{and }}&&\lim_{|q|\to 0}\sup_{x\in\R} \frac{|\nabla_q W(x, q)|}{|q|} = 0  .
\end{eqnarray}
We note that, by~\eqref{LITTLEO}, the identically null function is a solution of~\eqref{eqn2},
hence our main focus from now on will be on nontrivial solutions.

Besides, we assume that the potential $W$ satisfies the Ambrosetti--Rabinowitz condition, namely
\begin{equation}\label{AR}
\begin{split}
&\mbox{there exists $\mu>2$ such that }\\
&0<\mu W(x, q)\le \nabla_q W(x, q)\cdot q \quad\mbox{ for any $x\in\R$ and any $ q\in\R^n\setminus\{0\}$}.
\end{split}
\end{equation}

Depending on whether the fractional power $s$ falls into the range $(1/2, 1)$ or $(0, 1/2]$, we provide two distinct results. If $s\in (1/2, 1)$, we have the following:

\begin{theorem}\label{maintheorem2}
Let $s\in (1/2, 1)$. Let~$L\in C(\R, \R^{n\times n})$ satisfy~\eqref{LMAT}
and~\eqref{AUTOVAL} and let~$W\in C^1(\R\times\R^n, \R)$ satisfy~\eqref{W0IN0}, \eqref{LITTLEO} and~\eqref{AR}. 

In addition, assume that\footnote{For example, a function satisfying~\eqref{W0IN0}, \eqref{LITTLEO}, \eqref{AR} and~\eqref{WSEGNATO}
is~$W(x,q)=a(x)|q|^\mu$ with~$a\in L^\infty(\R)$.}
\begin{equation}\label{WSEGNATO}
\begin{split}
&\mbox{for any~$M>0$ there exists~$\kappa_M>0$ such that}\\
&\left|\frac12\nabla_q W(x, q)\cdot q-W(x, q)\right|\le \kappa_M |q|^2\quad\mbox{ for any $x\in\R$ and any $q\in\R^n$ with~$|q|\le M$}.
\end{split}
\end{equation}

Then,  there exists a nontrivial solution $q$ to problem~\eqref{eqn2} such that $q\in C^{s-1/2}(\R, \R^n)$ and
\begin{equation}\label{QREGC}
\|q\|_{C^{s-1/2}(\R, \R^n)}\le C^\star
\end{equation}
for some $C^\star>0$ depending only on $n, s$, $W$ and $L$.

Also, $q\in C^{2s}(B_{1/2}, \R^n)$ and
\begin{equation}\label{REGLOC}
\|q\|_{C^{2s}(B_{1/2}, \R^n)} \le C^{\star\star}
\end{equation}
for some $C^{\star\star}>0$ depending only on $n, s$, $W$ and $L$.

Moreover, 
\begin{equation}\label{LIM}
\lim_{x\to\pm\infty} q(x)=0.
\end{equation}
\end{theorem}

Here below and in the rest of the paper, we use the notation
\begin{equation}\label{2STARESSE}
2^*_s:=\begin{cases}
\dfrac{2}{1-2s} &\mbox{ if } s\in \left(0, \frac12\right),\\
+\infty &\mbox{ if } s = \frac12.
\end{cases}
\end{equation}

In the case~$s\in (0, 1/2]$,
in addition to~\eqref{W0IN0},~\eqref{LITTLEO} and~\eqref{AR}, we also require the following condition to hold true\footnote{For example, a function which satisfies~\eqref{W0IN0}, \eqref{LITTLEO}, \eqref{AR}
and~\eqref{INPIUNUOVA} is~$W(x, q)= \frac{|q|^p}{p}$, with~$p=\mu\in (2, 2^*_s)$.}
\begin{equation}\label{INPIUNUOVA}
\begin{split}
&\mbox{there exist~$a_0>0$ and~$p\in (2, 2^*_s)$ such that}\\
&|\partial_{q_j} W(x, q)|\le a_0\,|q_j|(1+ |q|^{p-2}) \quad\mbox{ for any $j\in\{1,\dots,n\}$, any $x\in\R$ and any~$q\in\R^n$}.
\end{split}
\end{equation}

Then, in this scenario, we have the following result:

\begin{theorem}\label{maintheorem3}
Let~$s\in (0, 1/2]$.  Let~$L\in C(\R, \R^{n\times n})$ satisfy~\eqref{LMAT} and~\eqref{AUTOVAL} and let~$W\in C^1(\R\times\R^n, \R)$ satisfy~\eqref{W0IN0}, \eqref{LITTLEO}, \eqref{AR} and~\eqref{INPIUNUOVA}.

Then, there exists a nontrivial solution~$q$ to problem~\eqref{eqn2}.

Moreover, if~$s\in (0, 1/2)$ and~$L(x)$ has nonnegative entries, then~$q\in L^\infty(\R,\R^n)\cap C^{2s}(B_{1/2}, \R^n)$ and
\begin{equation}\label{cbakbd}\|q\|_{L^\infty(\R,\R^n)}+
\|q\|_{C^{2s}(B_{1/2}, \R^n)} \le C^{\star},
\end{equation}
for a suitable~$C^\star>0$ depending only on $n$, $s$, $W$ and~$L$.

If instead~$s=1/2$ and~$L(x)$ has nonnegative entries, then~$q\in L^\infty(\R,\R^n)\cap C^{1-\varepsilon}(B_{1/2}, \R^n)$ for any~$\varepsilon>0$ and
\begin{equation}\label{cbnklaj}\|q\|_{L^\infty(\R,\R^n)}+
\|q\|_{C^{1-\varepsilon}(B_{1/2}, \R^n)} \le C^{\star\star},
\end{equation}
for a suitable~$C^{\star\star}>0$ depending only on $n$, $s$, $W$, $L$ and~$\varepsilon$.

If, in addition to the previous assumptions, 
\begin{equation}\label{AGGIMPPO}\begin{split}
&{\mbox{there exists~$D>0$ and continuous functions~$d_1,\dots, d_n:\R\to\R$}}\\&
{\mbox{such that, for all~$j\in\{1,\dots,n\}$,}}\quad \lim_{x\to\pm\infty} d_j(x)=+\infty\\
&{\mbox{and, for all~$x\in\R\setminus(-D,D)$,}}\\
&L(x)=\left(\begin{matrix}
d_1(x) & 0&0 &\dots &0\\
0 & d_2(x)& 0&\dots& 0\\ \vdots&
&\ddots\\ \vdots&
& & \ddots\\ 0&
\dots &\dots&0& d_{n}(x) 
\end{matrix}\right)
\end{split}
\end{equation}
then, \eqref{LIM} holds true.
\end{theorem}

We point out that the additional assumption in~\eqref{AGGIMPPO} is needed in
Section~\ref{PKSJLDNPHFOIKBFUIJBN} to construct an ad-hoc barrier in order
to control the behavior of the solution~$q$ at infinity and thus\footnote{This is an interesting conceptual difference
between the cases~$s\in(0,1/2]$ and~$s\in(1/2,1)$.
Indeed, when~$s\in(1/2,1)$ any bound on the energy functional
ensures uniform continuity and, as a byproduct, the decay
at infinity in~\eqref{LIM}. Instead, when~$s\in(0,1/2]$
one needs to rely on nonlocal elliptic regularity theory.
This is however nontrivial, because the coercivity
of~$L$ makes it impossible, in principle, to have global estimates,
due to the divergence of a term in the main equation
(from the standard theory, one would only obtain local,
rather than global, estimates). To overcome this
difficulty, we construct a suitable barrier at infinity.} establish~\eqref{LIM} when~$s\in(0,1/2]$. When~$n=1$, that is when the system boils down to single equation, this additional assumption
is redundant, as it is clearly implied by~\eqref{AUTOVAL}. It would be interesting to
investigate to what extent condition~\eqref{AGGIMPPO} can be relaxed.

The problem in~\eqref{eqn2} can also be interpreted in the light of atom dislocation in crystals.
In this setting, the linear term modulated by~$L$ plays the role of a confinement potential
(roughly speaking, for solutions of controlled energy, one expects a decay of the solution at infinity, to compensate the coefficient growth in~\eqref{AUTOVAL}). In this framework, homoclinic configurations for
the dislocation functions are shown in Theorems~\ref{maintheorem2} and~\ref{maintheorem3} to be possible
under suitable structural assumptions\footnote{In the spirit of footnote~\ref{FOOKSHDJ}, a small perturbation of a classical periodic potential fulfills all the assumptions of Theorems~\ref{maintheorem2} and~\ref{maintheorem3}. For example, one can take
$$L(x):=\big(1+\varepsilon |x|^2\big) {\rm{Id}}$$ to satisfy~\eqref{LMAT} and~\eqref{AUTOVAL}, and~$$W(x, q):=-\left(1-\cos q\right)^{1+\varepsilon}+ \varepsilon |q|^p$$ to satisfy~\eqref{W0IN0}, \eqref{LITTLEO}, \eqref{AR} and~\eqref{WSEGNATO},
with~$\varepsilon>0$ and~$p\in(2,2+2\varepsilon)$.}
on the potential (which is here denoted by~$W$, not to be confused
with the previous, constructionally different, potential~$V$).
\medskip

The rest of this article is organized as follows. 
Section~\ref{opt:S} presents some careful discussions about the main
assumptions of Theorems~\ref{main1} and~\ref{THAGG}. In particular, it shows that
equation~\eqref{prob1} does not possess a global solution in the
whole of~$\R$ (hence, the assumption that a datum needs to be prescribed
in some possibly degenerate interval~$[a, b]$ with~$a\le b$ cannot be removed).
Moreover, we show that in the range of fractional parameter~$s\in(0,1/2]$,
the minimization problem trivializes, since the infimum is attained at the zero energy level.

The proofs of Theorems~\ref{main1} and~\ref{THAGG} occupy Sections~\ref{OJNDDohnerfweP},
\ref{sec-regularity} and~\ref{section-proof}. Specifically,
Section~\ref{OJNDDohnerfweP} contains some auxiliary observations about
cutoff methods, asymptotics and continuous embeddings and
Section~\ref{sec-regularity} develops some regularity results.
With this preliminary work, the proofs of
Theorems~\ref{main1} and~\ref{THAGG} are thus completed in Section~\ref{section-proof}.

The proofs of Theorems~\ref{maintheorem2} and~\ref{maintheorem3}
are contained in Sections~\ref{ojsnkdD}, \ref{SIPDJOLNDFUOJFOJLN}, \ref{KSPldMFSw}, \ref{PKSJLDNPHFOIKBFUIJBN}
and~\ref{SPKJODLN}.
More precisely, Section~\ref{ojsnkdD} is devoted to some preliminary estimates
on the potential function~$W$.
Section~\ref{SIPDJOLNDFUOJFOJLN} presents the functional setting
which comes in handy to study problem~\eqref{eqn2}, Section~\ref{KSPldMFSw}
contains the estimate in class~$L^\infty(\R,\R^n)$,
and Section~\ref{PKSJLDNPHFOIKBFUIJBN} constructs a useful barrier
to control the behavior of solutions at infinity also when~$s\in (0, 1/2]$
(hence playing an important role in the proof of Theorem~\ref{maintheorem3}).
The proofs of Theorems~\ref{maintheorem2} and~\ref{maintheorem3}
are then completed in Section~\ref{SPKJODLN}.

It is interesting to remark that the uniform bounds
obtained in Section~\ref{KSPldMFSw} are
not completely standard, both because they deal with systems of equations and because the system under consideration presents some unbounded coefficients which require a special treatment.

The paper ends with Appendices~\ref{AGGNORMPROOFSEC}--\ref{interpappe}, where we establish some
useful facts needed in the proofs of the main results.

\section{Optimality of the assumptions in Theorems~\ref{main1} and~\ref{THAGG}}\label{opt:S}

In this section we discuss the optimality of the results stated in Theorems~\ref{main1} and~\ref{THAGG}. 

It is worth pointing out that, in general, the solution provided by Theorem~\ref{main1} is not a solution in the whole of~$\R$.
On this matter, we present the following example:

\begin{example}\label{EXAMPLE1}
Let~$n=1$ and consider the potential~$V(q):=-|q|^2$,
which clearly satisfies assumptions~\eqref{V1} and~\eqref{V2}.

We assume, for the sake of contradiction, that~$q$  is a continuous and bounded solution of the problem
\begin{equation}\label{EQR}
(-\Delta)^s q = V'(q) \quad\mbox{ in } \R
\end{equation}
which satisfies~\eqref{limite}.

By~\eqref{limite}, we have that either~$q=0$ or, without loss of generality,  we can suppose that there exists~$x\in\R$ such that~$q(x)>0$. Let~$\overline x$ be the maximum of~$q$ over~$\R$. In particular, $q(\overline x)>0$.  
Thus, by evaluating~\eqref{EQR} at~$q(\overline x)$, we have that
\[
0 \le (-\Delta)^s \, q(\overline x) = -2 \, q(\overline x) <0,
\]
which is a contradiction. This proves that~$q$ cannot be a solution in the whole of~$\R$.
\end{example}
We emphasize that the non-existence of homoclinic solutions is not an uncommon phenomenon. On this topic, we refer the interested reader, e.g., to the non-existence result obtained in~\cite[Proposition~C.1]{MR4108219} for the case of a multi-well potential.

An interesting feature which distinguishes Theorem~\ref{main1} from Theorem~\ref{THAGG} is that for~$s\in (1/2, 1)$ one is also allowed to take~$a=b$.
One might wonder if the same choice can be made in the complementary case~$s\in (0, 1/2]$. The answer is negative. Indeed, if~$a=b$ and~$s\in (0, 1/2]$, the minimum is not achieved in~$\Gamma$, according to the next result. In particular, this means that Theorem~\ref{main1} is optimal with respect to this feature.


\begin{proposition}\label{NASIK}
Let~$s\in (0, 1/2]$. Then, for any~$\varepsilon>0$ and~$M\in\R$ there exists an even function~$q_\varepsilon\in C(\R, \R)$ such that~$q_\varepsilon(0)=M$ and 
\begin{equation}\label{SPKq}
\lim_{\varepsilon\to 0} I(q_\varepsilon)=0.
\end{equation}
\end{proposition}

\begin{proof}
We first suppose that~$M\ge0$. In this case, let
\[
q_*(x):=\begin{cases}
\log\big(1-\log|x|\big) &\mbox{ if } x\in (-1, 0)\cup (0, 1),\\
{+\infty}&\mbox{ if } x=0,\\
0 &\mbox{ otherwise}.
\end{cases}
\]
By~\cite[Appendix~A]{m3as} we know that~$q_*\in \mathcal{D}^{\frac12, 2}(\R)$. We define
\[
\theta:=\begin{cases}
1 &\mbox{ if } s\in \left(0,\frac12\right),\\
\frac{1}{\varepsilon} &\mbox{ if } s=\frac12
\end{cases}
\]
and
\begin{equation}\label{qsharp}
q_\sharp(x):= \min\left\{\frac{q_*(x)}{\theta}, M \right\}.
\end{equation}
We observe that
\begin{equation}\label{ABC}
|q_\sharp(x)-q_\sharp(y)|\le\frac{|q_*(x)-q_*(y)|}{\theta}
,\end{equation}
whence we infer that~$q_\sharp\in \mathcal{D}^{\frac12, 2}(\R)$.  Moreover,  since~$q_\sharp$ is bounded and compactly supported, we also have that~$q_\sharp\in L^2(\R)$. 

As a matter of fact, we claim that 
\begin{equation}\label{inDS}
q_\sharp\in \mathcal{D}^{s, 2}(\R)\quad\mbox{ for any } s\in (0, 1/2].
\end{equation}
Indeed, when~$s\in(0,1/2)$,
\[
\begin{split}&
[q_\sharp]^2_s = \int_{\R\setminus (-1,1)} |2\pi \xi|^{2s} |\widehat{q_\sharp}(\xi)|^2 \,d\xi + \int_{(-1,1 )} |2\pi \xi|^{2s} |\widehat{q_\sharp}(\xi)|^2 \,d\xi\\
&\qquad \quad\le \int_\R |2\pi\xi| |\widehat{q_\sharp}(\xi)|^2 \,d\xi + 2\pi\int_\R |\widehat{q_\sharp}(\xi)|^2 \,d\xi= [q_{\sharp}]^2_{1/2} + \|q_\sharp\|^2_{L^2(\R)} <+\infty,
\end{split}
\]
which establishes~\eqref{inDS}.

Now, we set
\[
q_\varepsilon (x):= q_\sharp\left(\frac{x}{\varepsilon}\right).
\]
By definition, $q_\varepsilon\in C(\R,\R)$,
$q_\varepsilon$ is even and~$q_\varepsilon(0)=M$. 

\begin{figure}[h]
\begin{center}
\includegraphics[scale=.75]{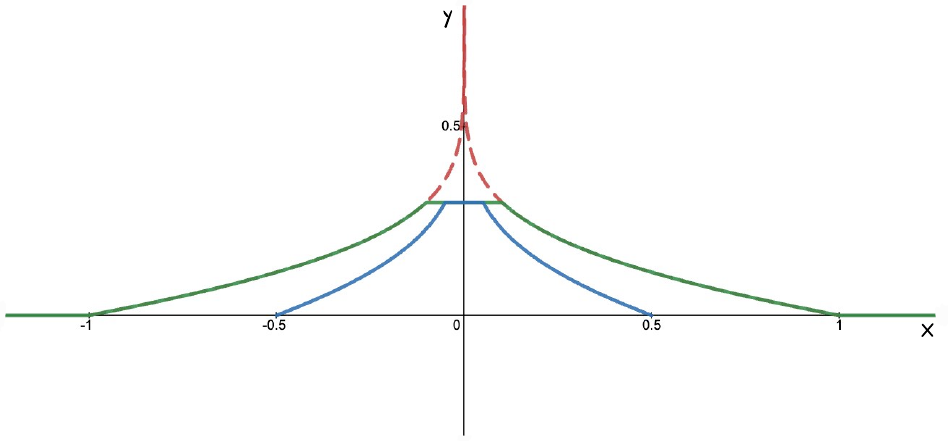}
\end{center}
\caption{The functions~$q_\sharp$ (in green) and~$q_\varepsilon$ (in blue).}
\end{figure}

We now check that~\eqref{SPKq} holds true.
For this, we notice that~$q_\varepsilon(x)=0$ for any~$x\in\R\setminus (-\varepsilon, \varepsilon)$. {F}rom this and recalling that~$V(0)=0$, 
\begin{equation}\label{Vinfty}
\begin{split}&
\int_\R |V(q_\varepsilon(x))| \,dx = \int_{\R\setminus (-\varepsilon, \varepsilon)} |V(q_\varepsilon(x))| \,dx +\int_{(-\varepsilon, \varepsilon)} |V(q_\varepsilon(x))| \,dx\\
&\qquad\qquad= \int_{(-\varepsilon, \varepsilon)} |V(q_\varepsilon(x))| \,dx
 \le 2\varepsilon\, \|V\|_{L^\infty((0, M))}.
\end{split}
\end{equation}
Furthermore,
\begin{equation}\label{26BIS}\begin{split}
\iint_{\R^2} \frac{|q_\varepsilon(x)-q_\varepsilon(y)|^2}{|x-y|^{1+2s}} \,dx\,dy &= \iint_{\R^2} \frac{\left|q_\sharp\left(\frac{x}{\varepsilon}\right)-q_\sharp\left(\frac{y}{\varepsilon}\right)\right|^2}{|x-y|^{1+2s}} \,dx\,dy\\&= \varepsilon^{1-2s}\iint_{\R^2} \frac{|q_\sharp(x)-q_\sharp(y)|^2}{|x-y|^{1+2s}} \,dx\,dy.\end{split}
\end{equation}

Let now distinguish two cases: if~$s\in (0, 1/2)$, recalling~\eqref{functional}
and exploiting~\eqref{26BIS},
\eqref{Vinfty} and~\eqref{inDS}, we find that
\[
I(q_\varepsilon) = \frac{c_s}2
\iint_{\R^2} \frac{|q_\varepsilon(x)-q_\varepsilon(y)|^2}{|x-y|^{1+2s}} \,dx\,dy - \int_\R V(q_\varepsilon(x)) \,dx \le \frac{c_s \varepsilon^{1-2s} [q_\sharp]^2_s}2 +2\varepsilon \|V\|_{L^\infty((0, M))}
\]
and the desired result in~\eqref{SPKq} follows by taking the limit as~$\varepsilon\searrow 0$. 

If instead~$s=\frac12$, by~\eqref{functional}, \eqref{ABC}, \eqref{Vinfty} and~\eqref{qsharp}, and recalling that~$q_*\in \mathcal{D}^{\frac12, 2}(\R)$,
\[
\begin{split}
I(q_\varepsilon) &\le
\frac{c_s}{2\theta^2} \iint_{\R^2} \frac{|q_*(x)-q_*(y)|^2}{|x-y|^2} \,dx\,dy - \int_\R V(q_\varepsilon(x)) \,dx\\
&\le\frac{c_s \varepsilon^2}2 [q_*]^2_{1/2} + 2\varepsilon \|V\|_{L^\infty((0, M))}.
\end{split}
\]
Once again, the desired result in~\eqref{SPKq} follows by taking the limit as~$\varepsilon\searrow 0$.

In the case~$M<0$, one takes
\begin{eqnarray*} q_*(x)&:=&\begin{cases}
-\log\big(1-\log|x|\big) &\mbox{ if } x\in (-1, 0)\cup (0, 1),\\
0 &\mbox{ otherwise},
\end{cases}\\
{\mbox{and }}\quad 
q_\sharp(x)&:=& \max\left\{\frac{q_*(x)}{\theta}, M \right\}.
\end{eqnarray*}
The desired result then follows as in the previous case.
\end{proof}

\section{Preliminary results towards the proofs of Theorems~\ref{main1} and~\ref{THAGG}}\label{OJNDDohnerfweP}

The proof of Theorem~\ref{main1} comes through the study of the functional~$I$ in~\eqref{functional}.  To this aim, we gather in this section some auxiliary observations.

We start by proving that, given a measurable function~$q:\R\to\R^n$, we can decrease the seminorm of~$q$ by performing a cutoff procedure (in the sense of~\eqref{cut}).
This fact, together with assumption~\eqref{V2}, will allow us to ``lower" the total energy.
\begin{lemma}\label{lemma_taglio}
Let~$R>0$.
For any measurable function~$q:\R\to\R^n$, let
\[
\mathcal T_R (q):= (T_R(q_1), \dots, T_R(q_n)),
\]
where, for all~$i\in\{1,\dots, n\}$, $T_R(q_i)$ is defined as in~\eqref{cut}.

Then,
\begin{equation*}
[\mathcal T_R(q)]^2_s \le [q]^2_s.
\end{equation*}
\end{lemma}

\begin{proof}
We realize the cutoff procedure into two steps. 
For any~$i\in\{1, \dots, n\}$, let 
\[
\underline{T}_R (q_i):= \min\{q_i, R\} \quad\mbox{ and }\quad \overline{T}_R (q_i):= \max\{q_i, -R\}.
\]
By~\eqref{cut} we infer that
\begin{equation*}
T_R(q_i) = \overline{T}_R(\underline{T}_R(q_i)) =\underline{T}_R(\overline{T}_R(q_i)).
\end{equation*}
Moreover, for any~$x$, $y\in\R$, we have that
\begin{eqnarray*}
&&\big|(\underline{T}_R (q_i))(x) - (\underline{T}_R (q_i))(y)\big| \le |q_i(x) -q_i(y)|
\\ {\mbox{and }} &&
\big|(\overline{T}_R (q_i))(x) - (\overline{T}_R (q_i))(y)\big|^2 \le |q_i(x) -q_i(y)|^2.
\end{eqnarray*}
{F}rom these observations we deduce that
\[
\begin{split}
\left[\mathcal T_R(q)\right]^2_s & = \iint_{\R^2} \sum_{i=1}^n \frac{|T_R(q_i(x))-T_R(q_i(y))|^2}{|x-y|^{1+2s}} \, dx\, dy\\
&= \iint_{\R^2} \sum_{i=1}^n \frac{|\overline{T}_R(\underline{T}_R(q_i(x)))-\overline{T}_R(\underline{T}_R(q_i(y)))|^2}{|x-y|^{1+2s}} \, dx\, dy\\
&\le \iint_{\R^2} \sum_{i=1}^n \frac{|\overline{T}_R(q_i(x))-\overline{T}_R(q_i(y))|^2}{|x-y|^{1+2s}}\, dx\, dy\\
&\le \iint_{\R^2} \sum_{i=1}^n \frac{|q_i(x) - q_i(y)|^2}{|x-y|^{1+2s}}\, dx\, dy\\
& = [q]^2_s,
\end{split}
\]
which gives the desired result.
\end{proof}

Now, we point out that if the solution~$q$ of~\eqref{prob1} is sufficiently regular and~$V$ satisfies~\eqref{V1}, then~$q$ tends to~$0$ as~$x$ goes to infinity. This is a standard observation,
but since we need it to prove the claim in~\eqref{limite}, we make it explicit via the following statement.

\begin{lemma}\label{lemma-limite}
Let~$V\in C(\R^n, \R)$ satisfy~\eqref{V1}. Assume that~$q\in L^\infty(\R, \R^n)$ is uniformly continuous and~$I(q)<+\infty$. 

Then,
\begin{equation*}
\lim_{x\to \pm\infty} q(x)=0.
\end{equation*}
\end{lemma}

\begin{proof}
We establish the limit as~$x\to +\infty$, the other one being similar.
For this, assume by contradiction that there exist~$\varepsilon>0$ and a sequence~$x^j$ such that~$x^{j}\to +\infty$ as~$j\to +\infty$ and
\[
|q(x^{j})| >\varepsilon.
\]
Without loss of generality, we can suppose that~$x^{j+1}>x^j+1$.

By the uniform continuity of~$q$, one can find~$\delta\in(0,1/2)$ small enough such that, for all~$j\in\N$ and all~$x\in [x^j-\delta, x^j +\delta]$,
\[
|q(x)| >\frac{\varepsilon}{2}.
\]

We set
$$ S:=\varepsilon+\displaystyle\sup_{\R} |q| \qquad {\mbox{and}}\qquad
\sigma_{\varepsilon, S}:= \inf_{q\in B_S\setminus B_{\varepsilon/2}} \, (-V(q)).
$$
By~\eqref{V1}, we have that~$\sigma_{\varepsilon, S}>0$. Thus, recalling~\eqref{functional}, we get that
\[
I(q)=\frac12 [q]^2_s -\int_\R V(q(x)) \,dx\ge - \sum_{j=1}^{+\infty} \int_{x^j-\delta}^{x^j+\delta} V(q(x)) \,dx \ge  \sigma_{\varepsilon, S}  \sum_{j=1}^{+\infty} (2\delta) = +\infty,
\]
which contradicts the assumption~$I(q)<+\infty$. Hence the desired result is proved.
\end{proof}

We now prove that if~$q$ is regular enough, then also~\eqref{dot_lim} is verified. We point out that property~\eqref{dot_lim} holds also in the local setting considered in~\cite{Rab2}. However,  the arguments which will lead to its proof in the nonlocal scenario
are entirely different from the one in~\cite{Rab2} and rely upon a regularity argument (exploiting the forthcoming Lemma~\ref{reg4} and Corollary~\ref{C1}, while~\cite{Rab2} can use elementary methods from ordinary differential equations). To this end, we recall the following elementary observation:

\begin{lemma}[Decay of the derivative]\label{qpunto}
Assume that~$q\in C^{\eta}(\R\setminus [a-1/2, b+1/2], \R^n)$, for some~$\eta>1$, and that
\begin{equation}\label{lim1}
\lim_{x\to\pm\infty} q(x) =c\in\R. 
\end{equation}

Then,
\[
\lim_{x\to \pm\infty} \dot{q}(x)=0.
\]
\end{lemma}

\begin{proof}
We establish the limit as~$x\to +\infty$, the other one being similar. To this end, let~$\varepsilon>0$ and observe that
\[
q(x+\varepsilon) -q(x) =\int_x^{x+\varepsilon} \dot{q} (\tau)\, d\tau = \varepsilon \dot{q}(x) + \int_x^{x+\varepsilon} \big(\dot{q} (\tau) -\dot{q}(x)\big)\, d\tau.
\]
This gives that
\[
\dot{q}(x) = \frac{q(x+\varepsilon) -q(x)}{\varepsilon} +\frac{1}{\varepsilon} \int_x^{x+\varepsilon} \big(\dot{q} (x) -\dot{q}(\tau)\big)\, d\tau.
\]

Now, recalling that~$q\in C^{\eta}(\R\setminus [a-1/2, b+1/2), \R^n)$ with~$\eta>1$,
for~$x$ sufficiently large we have that
\[
\begin{split}
|\dot{q}(x)| &\le \frac{\big|q(x+\varepsilon) -q(x)\big|}{\varepsilon} +\frac{c_1}{\varepsilon} \int_x^{x+\varepsilon} |x-\tau|^{\eta-1} \,d\tau= \frac{\big|q(x+\varepsilon) -q(x)\big|}{\varepsilon} +\frac{c_1}{\eta} \, \varepsilon^{\eta-1},
\end{split}
\]
for some positive constant~$c_1$. 

Hence, by~\eqref{lim1},
\[
\lim_{x\to +\infty} |\dot{q}(x)| \le \frac{c_1}{\eta}\, \varepsilon^{\eta-1}.
\]
The desired result now follows by sending~$\varepsilon\searrow 0$.
\end{proof}

We now deal with the case in which~$a=-b$ and~$q_0$ is even. Thanks to the next result, we will be able to prove that, under these assumptions, the functional~$I$ admits an even minimizer.
This observation will be useful to prove the last claim in Theorem~\ref{main1}.

\begin{lemma}\label{lemma-even}
Let~$q$, $q_*\in\Gamma$ satisfy~$I(q)$, $I(q_*)<+\infty$. Let~$V\in C(\R^n, \R)$. 
Let also
\begin{equation}\label{Mm}
M(x):= \max\{ q(x), q_*(x)\}\quad\mbox{ and }\quad m(x):=\min\{q(x), q_*(x)\}.
\end{equation}

Then,
\[
[M]^2_s +[m]^2_s \le [q]^2_s+
[q_*]^2_s \] and \[ \int_\R V(M(x)) \,dx + \int_\R V(m(x)) \,dx = \int_\R V(q(x)) \,dx+
\int_\R V(q_*(x)) \,dx.
\]
In particular,
\[
I(M) + I(m)\le I(q)+I(q_*).
\]
\end{lemma}

In our setting, we will use this result when~$a=-b$, $q_0$ is an even function
and~$q_*(x):= q(-x)$, so that~$q_*$ is a competitor for the minimization process over~$\Gamma$,
\begin{equation*} 
[q_*]_s^2 = \iint_{\R^{2}} \frac{|q(-x) -q(-y)|^2}{|x-y|^{1+2s}} \, dx\, dy = \iint_{\R^{2}} \frac{|q(x) -q(y)|^2}{|x-y|^{1+2s}} \, dx\, dy= [q]_s^2
\end{equation*}
and
\begin{equation*}
\int_\R V(q_*(x)) \,dx =\int_\R V(q(x)) \,dx,
\end{equation*}giving that~$I(q)=I(q_*)$.
The result in Lemma~\ref{lemma-even} is however of general flavor.

\begin{proof}[Proof of Lemma~\ref{lemma-even}]
Recalling the definitions in~\eqref{Mm}, we have that, for any~$x$, $y\in\R$,
\[
|m(x)-m(y)|^2 + |M(x)-M(y)|^2 \le |q(x)-q(y)|^2 +|q_*(x) -q_*(y)|^2
\]and therefore
\[
\begin{split}
[m]^2_s + [M]^2_s &=\iint_{\R^2} \frac{|m(x)-m(y)|^2}{|x-y|^{1+2s}} \,dx\, dy + \iint_{\R^2} \frac{|M(x)-M(y)|^2}{|x-y|^{1+2s}} \,dx\, dy \\
&\le\iint_{\R^2} \frac{|q(x)-q(y)|^2}{|x-y|^{1+2s}} \,dx\, dy + \iint_{\R^2} \frac{|q_*(x)-q_*(y)|^2}{|x-y|^{1+2s}} \,dx\, dy\\
&\le[q]^2_s + [q_*]^2_s.
\end{split}
\]
Moreover, by~\eqref{Mm},
\[
\int_\R V(m(x)) \,dx = \int_{\{q(x)\le q_*(x)\}} V(q(x)) \,dx + \int_{\{q(x)> q_*(x)\}} V(q_*(x)) \,dx 
\]
and, similarly,
\[
\int_\R V(M(x)) \,dx = \int_{\{q(x)\le q_*(x)\}} V(q_*(x)) \,dx + \int_{\{q(x)> q_*(x)\}} V(q(x)) \,dx.
\]
Hence,
\[
\int_\R V(m(x)) \,dx + \int_\R V(M(x)) \,dx = \int_\R V(q(x)) \,dx + \int_\R V(q_*(x)) \,dx .
\]
These observations lead to the desired result.\end{proof}

\section{Some regularity results}\label{sec-regularity}

The goal of this section is to prove some auxiliary regularity results which are partial steps towards the proof of the
regularity statements in
Theorems~\ref{main1} and~\ref{THAGG}.
Some of the results presented are suitable adaptation of results and methods in the existing literature
to the case under consideration. 

For this purpose, we introduce the notion of weak solution to our problem. In what follows, unless otherwise specified,
we take~$a<b$ and we suppose that~$q_0$
satisfies the regularity assumption in~\eqref{ALPHAQ0}.

\begin{definition}
We say that~$q\in \mathcal D^{s, 2}(\R, \R^n)$ is a weak solution to problem~\eqref{prob1} if
\[
\begin{cases}
\begin{split}
&c_s\displaystyle\iint_{\R^2} \frac{(q_i(x) -q_i(y))(\varphi(x)-\varphi(y))}{|x-y|^{1+2s}} \,dx\, dy = \int_{\R} \partial_i V(q(x))\, \varphi(x) \, dx\\
&\qquad\qquad\qquad\qquad\mbox{ for any } \varphi\in C_0^\infty(\R\setminus [a, b]) \mbox{ and for any } i\in\{1, \dots, n\},\\
&q=q_0 \quad \mbox{ in } [a, b].
\end{split}
\end{cases}
\]
\end{definition}

Also, we use the notation\footnote{We point out that the notation in~\eqref{IAB} also comprises the cases in which~$a=-\infty$ and~$b=+\infty$. For instance, when~$a=-\infty$ the set in~\eqref{IAB} boils down to
$$ I_{a, b} = \left(b, b+\frac12\right),$$
while when~$b=+\infty$ to
$$ I_{a, b} =  \left(a-\frac12, a\right).$$
When~$a=-\infty$ and~$b=+\infty$, then~$I_{a, b}$ is the empty set.}
\begin{equation}\label{IAB}
I_{a, b} := \left(a-\frac12, a\right)\cup \left(b, b+\frac12\right).
\end{equation}
We aim to prove the following theorem:

\begin{theorem}\label{reg2}
Let~$q$ be a bounded weak solution of~\eqref{prob1}, with~$V\in C^1(\R^n, \R)$.

Then, $q\in C^{\beta}(\R, \R^n)$, with~$\beta$ as in~\eqref{beta}, and 
\begin{equation}\label{41IN}
[q]_{C^{\beta}(\R, \R^n)}\le C\, \Big(\|\nabla V(q)\|_{L^\infty(I_{a, b}, \R)} +\|q\|_{L^\infty(\R, \R^n)}\Big),
\end{equation}
for some~$C>0$ depending only on $n$, $s$, $\alpha$, $a$, $b$ and~$q_0$.
\end{theorem}

It is worth noting that the argument which will lead to Theorem~\ref{reg2} is different from the
one exploited in~\cite{Rab1, Rab2}. Indeed, in the classical case the transition from a weak solution to a classical solution is achieved through an ODE argument (see~\cite[Proposition~3.18]{Rab1}). Differently, the argument that we develop here relies upon some regularity results for nonlocal equations stated in~\cite{MR3168912, AR2020}.

We provide the following preliminary regularity result:
\begin{lemma}[Interior regularity]
\label{reg1}
Let~$a\le b$ and~$x_0\in(-\infty,a-1]\cup[b+1,+\infty)$. Let~$q\in \mathcal D^{s, 2}(\R, \R^n)$ be a bounded weak solution to the system
\begin{equation}\label{eqn}
(-\Delta)^s q = \nabla V(q) \quad\mbox{ in } (x_0-1, x_0+1),
\end{equation}
with~$V\in C^1(\R^n, \R)$. 

Then, $q\in C^{\alpha^\star}([x_0-1/2, x_0+1/2], \R^n)$ for any~$\alpha^\star \in(0, \min\{2s, 1\})$ and
\begin{equation}\label{est1}
[q]_{C^{\alpha^\star}([x_0-1/2, x_0+1/2], \R^n)} \le C \Big(\|q\|_{L^\infty(\R, \R^n)} + \|\nabla V(q)\|_{L^\infty((x_0-1, x_0+1))}\Big),
\end{equation}
for some constant~$C>0$ depending only on $n$, $s$ and~$\alpha^\star$.
\end{lemma}

\begin{proof}
We consider the case~$x_0\ge b+1$, the other one being similar.
Moreover, up to a translation, we assume that~$x_0=0$, so that the
proof can be centered around~$0$.

The desired result follows by~\cite[Proposition~5]{MR3161511} taking~$f:=\partial_i V(q)\in L^\infty([-1, 1], \R)$.
\end{proof}

By using Lemma~\ref{reg1} along with a covering argument, one can prove that~$q\in C^{\alpha\star}([a-1/2, b+1/2], \R^n)$,  as stated in the next lemma.

\begin{proposition}[Uniform interior regularity]\label{reg5}
Let~$a\le b$ and~$q\in \mathcal D^{s, 2}(\R, \R^n)$ be a bounded weak solution to the system
\begin{equation*}
(-\Delta)^s q = \nabla V(q) \quad\mbox{ in } \R\setminus [a, b],
\end{equation*}
with~$V\in C^1(\R^n, \R)$. 

Then, $q\in C^{\alpha^{\star}}(\R \setminus [a-1/2, b+1/2], \R^n)$ for any~$\alpha^{\star} \in(0, \min\{2s, 1\})$ and
\begin{equation}\label{SHA}
[q]_{C^{\alpha^{\star}}(\R \setminus [a-1/2, b+1/2], \R^n)} \le C \|q\|_{L^\infty(\R, \R^n)},
\end{equation}
for some constant~$C>0$ depending only on~$n$, $s$, $\alpha^{\star}$ and~$V$.
\end{proposition}

\begin{proof}
We employ Lemma~\ref{reg1} with~$x_0 := b+1$ (this will take care of the desired result in the region~$(b+1/2,+\infty)$ and a similar argument
with~$x_0 :=a-1$ would provide the result in the region~$(-\infty,a-1/2)$). In particular, we have that~$q\in C^{\alpha^\star}(\R \setminus [b+1/2, b+3/2], \R^n)$ and
$$[q]_{C^{\alpha^\star}([b+1/2, b+3/2], \R^n)} \le C \Big(\|q\|_{L^\infty(\R, \R^n)} + \|\nabla V(q)\|_{L^\infty((b, b+2))}\Big),
$$for some constant~$C>0$ depending only on $n$, $s$ and~$\alpha^\star$.

Let now~$x$, $y\ge b +3/2$ and, without loss of generality, assume that~$x<y$. We set~$m:=(x+y)/2$. We notice that two cases can occur:
\begin{itemize}
\item if~$|x-m|\le 1/2$ (and then also~$|y-m|\le 1/2$),
then~$m\ge b+1$ and therefore we can use Lemma~\ref{reg1} with~$x_0:=m$. Hence,  we obtain that~$q\in C^{\alpha^{\star}}([m-1/2, m+1/2], \R^n)$ and 
\[
[q]_{C^{\alpha^{\star}}([m-1/2, m+1/2], \R^n)}\le c_1 \Big(\|q\|_{L^\infty(\R, \R^n)} +\|\nabla V(q)\|_{L^\infty((m-1, m+1))}\Big),
\]
for a positive constant~$c_1$ depending on $n$, $s$ and~$\alpha^{\star}$.
\item if~$|x-m|>1/2$ (and then also~$|y-m|>1/2$), it follows that~$|x-y|>1$. Hence,
\[
|q(x)-q(y)|\le 2 \| q\|_{L^\infty(\R, \R^n)} \frac{|x-y|^{\alpha^{\star}}}{|x-y|^{\alpha^{\star}}} \le 2 \| q\|_{L^\infty(\R, \R^n)} |x-y|^{\alpha^{\star}}.
\]
\end{itemize}
Consequently, putting together the previous estimates, we have that~$q\in C^{\alpha^{\star}}(\R \setminus [a-1/2, b+1/2], \R^n)$ and~\eqref{SHA} holds true.
\end{proof}

In order to complete the proof of Theorem~\ref{reg2}, we have to take care of the boundary regularity. For this, we present the
following statement result which essentially patches together the 
regularity results obtained in~\cite{MR3168912} and~\cite{AR2020}. For the sake of completeness, we state the result in any dimension.

\begin{proposition}[Regularity up to the boundary]\label{reg3}
Let~$s\in (0, 1)$ and let~$\Omega\subset\R^N$ be a bounded domain satisfying the exterior ball condition. 

Suppose that~$f\in L^\infty(\Omega)$ and~$g\in C^{\alpha^\star}(\overline{\R^N\setminus\Omega})$ for some~$\alpha^\star\in (0, s)\cup(s,2s)$.

In addition, assume that~$\Omega$ is of class~$C^1$ when~$\alpha^\star \in(0, s)$ and of class~$C^{1, \gamma}$ for some~$\gamma>0$ when~$\alpha^\star\in(s,2s)$.  

Let~$\beta^\star:=\min\{\alpha^\star, s\}$.

Let~$u$ be a solution of
\begin{equation}\label{uprob}
\left\{\begin{aligned}
(-\Delta)^s u &= f &&\mbox{ in } \Omega,\\
u&=g &&\mbox{ in } \R^N\setminus\Omega.
\end{aligned}
\right.
\end{equation}

Then, $u\in C^{\beta^\star}(\overline{\Omega})$ and
\begin{equation*}
\|u\|_{C^{\beta^\star}(\overline{\Omega})} \le C\Big(\|f\|_{L^\infty(\Omega)} +\|g\|_{C^{\alpha\star}(\overline{\R^N\setminus\Omega})}\Big),
\end{equation*}
for a positive constant~$C$ depending only on~$N$, $s$, $\alpha^\star$ and~$\Omega$.
\end{proposition}

\begin{proof}
Let us consider the problem
\begin{equation}\label{pro67342}
\left\{\begin{aligned}
(-\Delta)^s w &= 0 &&\mbox{ in } \Omega,\\
w&= g &&\mbox{ in } \R^N\setminus \Omega.
\end{aligned}
\right.
\end{equation}
We point out that a solution of this problem can be obtained, e.g., by a standard minimization of the Gagliardo seminorm
with prescribed datum outside~$\Omega$.

Also, since~$g\in C^{\alpha^\star}(\overline{\R^N\setminus\Omega})$, we have that, for all~$x\in \R^N\setminus\Omega$ and~$z\in\partial\Omega$,
\begin{equation}\label{g1}
|g(x) - g(z)|\le [g]_{C^{\alpha^\star}(\overline{\R^N\setminus\Omega})} |x-z|^{\alpha^\star}\le \|g\|_{C^{\alpha^\star}(\overline{\R^N\setminus\Omega})}|x-z|^{\alpha^\star}
\end{equation}
and, in addition, for all~$x\in \R^N\setminus \Omega$,
\begin{equation}\label{g2}
|g(x)| \le \|g\|_{L^\infty(\R^N\setminus\Omega)} (1+|x|^{\alpha^\star})\le \|g\|_{C^{\alpha^\star}(\overline{\R^N\setminus\Omega})}(1+|x|^{\alpha^\star}).
\end{equation}
This entails that we are in the right setting to apply~\cite{AR2020}. More precisely, 
when~$\alpha^\star\in(0,s)$ we use~\cite[Theorem~1.1]{AR2020} and when~$\alpha^\star\in(s,2s)$ we exploit~\cite[Proposition~1.3]{AR2020}.
In this way, we infer that~$w\in C^{\beta^\star}(\overline{\Omega})$ and
\begin{equation}\label{8439ythgfkuest9843ythgiroew}
\|w\|_{C^{\beta^\star}(\overline{\Omega})} \le c_1 \|g\|_{C^{\alpha^\star}(\overline{\R^N\setminus\Omega})},
\end{equation}
for some~$c_1>0$ depending only on~$N$, $s$, $\alpha^\star$ and~$\Omega$. 

We now define~$v:= u-w$.
By construction, $v$ solves
\[
\left\{\begin{aligned}
(-\Delta)^s v &= f &&\mbox{ in } \Omega,\\
v&=0 &&\mbox{ in } \R^N\setminus\Omega.
\end{aligned}
\right.
\]
Since~$f\in L^\infty(\Omega)$, by~\cite[Proposition~1.1]{MR3168912} we have that~$v\in C^{s}(\R^N)$ and 
\begin{equation}\label{8439ythgfkuest9843ythgiroew2}
\|v\|_{C^{s}(\R^N)} \le c_2 \|f\|_{L^\infty(\Omega)},
\end{equation}
for some~$c_2>0$ depending only on~$s$ and~$\Omega$. 

Thus, combining the estimates in~\eqref{8439ythgfkuest9843ythgiroew} and~\eqref{8439ythgfkuest9843ythgiroew2},
we obtain the desired result.
\end{proof}

\begin{remark}
We point out that 
Proposition~\ref{reg3} does not comprise the case~$\alpha^\star=s$. 
This comes from the fact that
when~$\alpha^\star=s$, 
in light of~\cite[Proposition~1.2]{AR2020}, there exist a~$C^\infty$ domain~$\Omega$ and a function~$g$ satisfying~\eqref{g1} and~\eqref{g2} such that the solution~$w$
of~\eqref{pro67342}
does not belong to~$ C^{s}(\overline{\Omega})$.
\end{remark}

We are now in the position of completing the proof of Theorem~\ref{reg2}.

\begin{proof}[Proof of Theorem~\ref{reg2}]
We exploit Proposition~\ref{reg5} with~$\alpha^{\star}:=s$ and we see that 
\begin{equation}\begin{split}\label{L441}
&q\in C^s(\R \setminus [a-1/2, b+1/2], \R^n)
\\{\mbox{and }}\quad &
[q]_{C^s(\R \setminus [a-1/2, b+1/2], \R^n)} \le c_1 \|q\|_{L^\infty(\R, \R^n)}
\end{split}\end{equation}
for some constant~$c_1>0$ depending only on~$n$, $s$ and~$V$.

Moreover, we point out that we are in the setting of Proposition~\ref{reg3}
with~$N=1$, $\Omega:=I_{a, b}$ (recall that~$a<b$ and therefore
the regularity assumptions on~$\Omega$ are satisfied),
$f:= \partial_i V(q)\in L^\infty(I_{a, b}, \R)$ and~$g:=q_i\in C^{\min\{\alpha, s\}}(\R\setminus I_{a, b})$ (thanks to~\eqref{ALPHAQ0}
and~\eqref{L441}) for any~$i\in\{1,\dots, n\}$. 

Then, by Proposition~\ref{reg3} we infer that, if~$\alpha\neq s$,
\begin{equation}\label{CBETAI}\begin{split}
&q\in C^{\min\{\alpha, s\}}(\overline{I_{a, b}}, \R^n)\\
{\mbox{and }}\quad &
\|q\|_{C^{\min\{\alpha, s\}}(\overline{I_{a, b}})} \le c_2 \Big(\|\nabla V(q)\|_{L^\infty(I_{a, b}, \R^n)} +\|q\|_{L^\infty(\R, \R^n)}\Big),\end{split}
\end{equation}
for some~$c_2>0$ depending on~$n$, $s$, $\alpha,$ $a$, $b$ and~$q_0$.

Thus, when~$\alpha\neq s$, \eqref{L441} and~\eqref{CBETAI} entail that~$q\in C^{\min\{\alpha, s\}}(\R, \R^n)$ and
\[
[q]_{C^{\min\{\alpha, s\}}(\R, \R^n)}\le c \Big(\|\nabla V(q)\|_{L^\infty(I_{a, b}, \R)} +\|q\|_{L^\infty(\R, \R^n)}\Big),
\]
for some~$c>0$ depending only on~$n$, $s$, $\alpha$, $a$, $b$ and~$q_0$.

If instead~$\alpha=s$, we use Proposition~\ref{reg3} with~$\alpha^\star:=\overline\beta\in(0,s)$ and, noticing that~$\min\{\overline\beta, s\}=\overline\beta$, we obtain that
\begin{equation*}\begin{split}
&q\in C^{\overline\beta}(\overline{I_{a, b}}, \R^n)\\
{\mbox{and }}\quad &
\|q\|_{C^{\overline\beta}(\overline{I_{a, b}})} \le c_3 \Big(\|\nabla V(q)\|_{L^\infty(I_{a, b}, \R^n)} +\|q\|_{L^\infty(\R, \R^n)}\Big),\end{split}
\end{equation*}
for some~$c_3>0$ depending on~$n$, $s$, $a$, $b$ and~$q_0$.
This and~\eqref{L441} give that~$q\in C^{\overline\beta}(\R, \R^n)$ and
\[
[q]_{C^{\overline\beta}(\R, \R^n)}\le c \Big(\|\nabla V(q)\|_{L^\infty(I_{a, b}, \R)} +\|q\|_{L^\infty(\R, \R^n)}\Big),
\]
for some~$c>0$ depending only on~$n$, $s$, $a$, $b$ and~$q_0$.

Recalling the definition of~$\beta$ in~\eqref{beta} completes the proof
of Theorem~\ref{reg2}.
\end{proof}

\begin{remark}\label{remimporego0686}
We point out that, even though Lemma~\ref{reg1} and Proposition~\ref{reg5} apply also when~$a=b$, Theorem~\ref{reg2} holds only when~$a<b$.

This is due to the fact that if~$a=b$, then 
\[
I_{a,a}= \left(a-\frac12, a\right)\cup \left(a, a+\frac12\right)
\]
does not verify the assumptions in Proposition~\ref{reg3} (due to the lack of domain regularity at the point~$a$).
\end{remark}

The next result states that, if~$V$ is assumed to be slightly more regular than~$C^1$, then one can gain more regularity on the solution~$q$ to~\eqref{prob1}. This fact 
will be useful in the proof of~\eqref{dot_lim}
and is a consequence of the next result and a covering argument.

\begin{lemma}\label{reg4}
Let~$x_0\in\R$ and~$\beta$ be as in~\eqref{beta}. Let~$q\in C^{\beta}(\R, \R^n)$ be a bounded weak solution to the system
\begin{equation}\label{EQNC1}
(-\Delta)^s q = \nabla V(q) \quad\mbox{ in } (x_0-1, x_0+1),
\end{equation}
where~$V\in C^{1+\gamma}(\R^n, \R)$, with~$\gamma\in (\max\{0,1-2s\}, 1)$. 

Then, there exist an interval~$I\subset (x_0-1, x_0+1)$ and~$\eta>1$ such that~$q\in C^\eta (I, \R^n)$.

In addition,
\[
\|q\|_{C^\eta (I, \R^n)}\le c,
\]
for a positive constant~$c$ depending only on~$n$, $s$, $\beta$, $V$ and~$\|q\|_{C^\beta(\R, \R^n)}$.
\end{lemma}

\begin{proof}
Up to a translation, we assume that~$x_0=0$.
We observe that, since~$q\in C^{\beta}((-1, 1), \R^n)$ and~$V\in C^{1+\gamma}(\R^n, \R)$, 
for any~$i\in\{1,\dots, n\}$ and~$x$, $y\in (-1, 1)$, we have that
\[
\left|\partial_i V(q(x)) - \partial_i V(q(y))\right| \le [\partial_i V]_{C^{\gamma}(\R,\R)} |q(x)-q(y)|^\gamma
\le [\partial_i V]_{C^{\gamma}(\R,\R)} \, [q]_{C^{\beta}(\R, \R^n)} |x-y|^{\gamma\beta}.
\]
On this account, it follows that the map~$(-1, 1)\ni x\mapsto
\nabla V(q(x))$ belongs to~$ C^{\gamma\beta}((-1, 1), \R^n)$.

Consequently, by~\cite[Proposition~2.2]{MR3168912} we have\footnote{Here, for simplicity, we are implicitly supposing that~$s\ne\frac12$; when~$s=\frac12$, the regularity would improve by~$1-\epsilon$
instead of~$2s=1$, at any step of the iteration, but the gist of the argument would remain basically unaltered.} that~$q\in C^{2s +\gamma\beta}([-1/2, 1/2], \R^n)$ and 
\[
\| q\|_{C^{2s +\gamma\beta}([-1/2, 1/2], \R^n)}\le c_1 \Big(\|q\|_{C^{\beta}(\R, \R^n)} +\|\nabla V(q)\|_{C^{\gamma\beta}((-1,1))}\Big),
\]
for some~$c_1>0$ depending on~$s$, $\beta$ and~$\gamma$.

Now, if~$2s+\gamma\beta>1$, then the proof of Lemma~\ref{reg4} is concluded.
If not, then we iterate the
previous argument: the gist is that, setting 
\begin{equation}\label{betasequence}
\beta_k:= 
\begin{cases}
2s+\gamma\beta &\mbox{ if } k=1,\\
2s+ \gamma\beta_{k-1} &\mbox{ if } k\ge 2,
\end{cases}
\end{equation}
we can repeat the earlier argument until
\[
q\in C^{2s+\gamma \beta_k} ((-2^{-k}, 2^{-k}), \R^n)
\]
with~$2s+\gamma \beta_k>1$. 

To make this work, we need to show that
\begin{equation}\label{S4B90}
{\mbox{there exists~$k\in\N$ such that~$2s+\gamma \beta_k>1$.}}\end{equation}
Actually, since~$\gamma>\max\{0,1-2s\}$ by assumption, to show~\eqref{S4B90} it suffices to check that there exists~$k\in\N$ such that~$\beta_k>1$. 

For this, we assume, by contradiction, that~$\beta_{k}\le 1$ for each~$k$.
Then, recalling that~$\gamma\in (\max\{0,1-2s\}, 1)$,
\[
\beta_k -\beta_{k-1} = 2s -\beta_{k-1}(1-\gamma)\ge 2s-1 +\gamma >0.
\]
This shows that~$\beta_k$ is an increasing sequence.

As a result, taking the limit in~\eqref{betasequence}, we have that
\[ \lim_{k\to+\infty}\beta_k=
\frac{2s}{1-\gamma}>1.
\]
This is against our contradictory assumption and therefore the proof of~\eqref{S4B90} is complete.
\end{proof}

In view of Lemma~\ref{reg4},
one can perform an argument similar to the one in Proposition~\ref{reg5}
and obtain an improved regularity of the solution of problem~\eqref{prob1}:

\begin{corollary}[Differentiable regularity theory]\label{C1}
Let~$\beta$ as in~\eqref{beta} and~$q\in C^{\beta}(\R, \R^n)$ be a bounded weak solution of
\begin{equation}\label{EQNC1}
(-\Delta)^s q = \nabla V(q) \quad\mbox{ in } \R\setminus [a, b],
\end{equation}
with~$V\in C^{1+\gamma}(\R^n, \R)$ and~$\gamma\in (\max\{0,1-2s\}, 1)$. 

Then, 
\[
q\in C^\eta (\R \setminus [a-1/2, b+1/2], \R^n)
\]
and
\[
\|q\|_{C^\eta (\R \setminus [a-1/2, b+1/2], \R^n)}\le c,
\]
for a positive constant depending only on~$ n$, $s$, $\beta$, $V$ and~$\|q\|_{C^\beta(\R, \R^n)}$.
\end{corollary}

\section{Proofs of Theorems~\ref{main1} and~\ref{THAGG}}\label{section-proof}

In this section we prove Theorems~\ref{main1} and~\ref{THAGG}. The strategy is to employ the results collected in Sections~\ref{OJNDDohnerfweP} and~\ref{sec-regularity} and perform an appropriate minimization argument over the set~$\Gamma$.

We will present a unified proof for Theorems~\ref{main1} and~\ref{THAGG}, with the only difference that the regularity claims
in~\eqref{QCBETA2} and~\eqref{BOH2} have to be dealt separately (indeed, the lack of domain regularity when~$a=b$
prevents us from exploiting Theorem~\ref{reg2} in this case, recall Remark~\ref{remimporego0686}).

\begin{proof}[Proof of Theorems~\ref{main1} and~\ref{THAGG}]
Let~$\Gamma$ be as in~\eqref{gamma} and consider the variational functional~$I$ given in~\eqref{functional}.
We take a minimizing sequence~$q^{j} \in\Gamma$ for the functional~$I$, i.e. 
\[
\lim_{j\to +\infty} I(q^{j}) = \inf_{q\in\Gamma} I(q).
\] 
Up to suitably extending~$q_0$ outside~$[a,b]$, we can suppose without loss of generality that~$q_0$
has finite energy,
hence, by~\eqref{V1} and~\eqref{functional}, we infer that
\begin{equation}\label{LTO}
+\infty> I(q_0) \ge I(q^{j})\ge \frac12 [q^{j}]^2_s.
\end{equation}
This shows that~$[q^{j}]_s$ is bounded independently of~$j\in\N$.  

We now remark that the minimizing sequence~$q^j$ can be taken in such a way that~$q^j\in L^\infty(\R, \R^n)$. Indeed, the argument which leads to~\eqref{LTO} remains valid by replacing~$q^{j}$ with the sequence~$\mathcal T_R\left(q^{j}\right)$, where~$\mathcal T_R$ is defined as in Lemma~\ref{lemma_taglio}. Indeed, $\mathcal T_R\left(q^{j}\right)$ remains a minimizing sequence itself and 
\[
I\left(\mathcal T_R\left(q^{j}\right)\right) \le I(q^{j})
\]
by the assumption in~\eqref{V2} and Lemma~\ref{lemma_taglio}. 

Hence, without loss of generality, we rename~$q^{j}$ as the cutoff sequence~$\mathcal T_R\left(q^{j}\right)$. With this, we notice that~$q^{j}\in L^\infty(\R, \R^n)$ by construction
(and, in fact, $|q_j|\le R$).

Now, by employing a diagonal argument and~\cite[Theorem~7.1]{guidagalattica} we infer that there exists some~$q\in\mathcal{D}^{s, 2}(\R, \R^n)$ such that~$q^j$ converges to~$q$ a.e. in~$\R$ as~$j\to+\infty$.

Moreover, by construction, we have that~$q\in\Gamma$ and also~$q\in L^\infty(\R, \R^n)$ and
\begin{equation}\label{QERRE}
\|q\|_{L^\infty(\R, \R^n)}\le R.
\end{equation}
In addition, by~\eqref{V1}, \eqref{functional} and Fatou's Lemma,
\begin{equation}\label{MNHH}
I(q)\ge \liminf_{j\to +\infty} I(q^{j}) \ge \frac12 [q]^2_s -\int_{\R} V(q(x))\, dx = I(q).
\end{equation}Accordingly, we have that~$q$ minimizes~$I$ over~$\Gamma$, namely~\eqref{minimo} holds. Also, by the previous inequality we deduce that~$I(q)<+\infty$.

We now focus on the regularity claims in~\eqref{QCBETA2} and~\eqref{BOH2}.
We first consider the case~$a<b$. In this case, thanks to Theorem~\ref{reg2}, we have that~$q\in C^{\beta}(\R, \R^n)$, where~$\beta$ as in~\eqref{beta}. Moreover, by~\eqref{41IN} and~\eqref{QERRE}, there exists a positive constant~$C$, depending only on~$n$, $s$, $\alpha$, $a$, $b$, $V$ and~$q_0$, such that
\[
[q]_{C^\beta(\R, \R^n)} \le C,
\]
which shows~\eqref{QCBETA2}.

Besides, if~$a=b$ and~$s\in (1/2, 1)$, we can employ the inequality in equation~(8.8) of~\cite[Theorem~8.2]{guidagalattica}, which entails that
\begin{equation*}
|q(x)-q(y)|\le c_3 [q]_s \, |x-y|^{s-1/2},
\end{equation*}
for a positive constant~$c_3$. Combining this with~\eqref{LTO} and~\eqref{QERRE}, we obtain that
$$
\|q\|_{C^{s-1/2}(\R, \R^n)} = \|q\|_{L^\infty(\R, \R^n)} + \sup_{\substack{x, y\in\R \\ x\neq y}} \frac{|q(x)-q(y)|}{|x-y|^{s-1/2}}
\le R + c_3 [q]_s
\le C,
$$ for a positive constant depending only on~$n$, $s$, $V$ and~$q_0$. This proves~\eqref{BOH2}.

Accordingly, in both cases, we have that~$q$ is uniformly continuous and therefore we can employ Lemma~\ref{lemma-limite} to deduce that~\eqref{limite} holds true. 

In addition, if~$V\in C^{1+\gamma}(\R, \R^n)$ for some~$\gamma\in(\max\{0,1-2s\},1)$,
then by Corollary~\ref{C1} and Lemma~\ref{qpunto} we conclude that also~\eqref{dot_lim} is verified.

Let us now assume that~$a=-b$ and that~$q_0$ is even.
We prove that~$I$ admits an even minimizer.  To this end, let~$q_*(x):= q(-x)$ and~$M(x)$ and~$m(x)$ be as in~\eqref{Mm}. 
By Lemma~\ref{lemma-even},
\begin{equation}\label{due}
I(M) + I(m)\le 2 I(q).
\end{equation}
Hence, either~$I(M)\le I(q)$ or~$I(m)\le I(q)$.  Without loss of generality, we assume that 
\begin{equation}\label{uno}
I(M)\le I(q).
\end{equation}

Now, since~$q$ minimizes~$I$ over~$\Gamma$, we see that~\eqref{uno} holds true with the equality sign, namely~$I(M)= I(q)$. Accordingly, by~\eqref{due}, 
\[
I(m)\le I(q)
\]
and once again we have that~$I(m) = I(q)$. Hence, we infer that
\[
I(M) = I(m)= I(q).
\]
Notice that~$M(-x)=\max\{q(-x),q_*(-x)\}=\max\{q_*(x), q(x)\}=M(x)$, and therefore
there exists an even minimizer.

If instead~$a=b$, one can take~$q_*(x):= q(2a-x)$ and~$M(x)$ and~$m(x)$ as in~\eqref{Mm} and repeat the argument above to show that~$q$
is symmetric with respect to~$a$.

This completes the proof of Theorems~\ref{main1} and~\ref{THAGG}.
\end{proof}

\section{The spatially dependent case of Theorems~\ref{maintheorem2} and~\ref{maintheorem3}: preliminary results}\label{ojsnkdD}

In this section, we derive some auxiliary results which will be useful for carrying out the proof of Theorems~\ref{maintheorem2} and~\ref{maintheorem3}. 

In the first place, we use the structural hypotheses~\eqref{LITTLEO} and~\eqref{AR} to assess the behavior of the potential~$W$ when~$q$ approaches the origin.
\begin{lemma}\label{lemmaW0}
Let~$W\in C^1(\R\times\R^n, \R)$ satisfy~\eqref{LITTLEO} and~\eqref{AR}. 

Then, for any~$\varepsilon>0$ there exists~$\delta=\delta(\varepsilon)>0$ such that,
for any~$x\in\R$ and any~$q\in\R^n\setminus\{0\}$ with~$|q|\le\delta$, 
\begin{equation*}
0<W(x, q)< \varepsilon |q|^2.
\end{equation*}
\end{lemma}

\begin{proof}
In light of the hypothesis in~\eqref{LITTLEO}, we have that for any~$\varepsilon>0$ there exists~$\delta=\delta(\varepsilon)>0$ such that, for any~$x\in\R$ and any~$q\in\R^n$ with~$|q|\le\delta$,
\[
|\nabla_q W(x, q)\cdot q|\le \varepsilon |q|^2.
\]
Hence, combining this with~\eqref{AR}, we infer that, for any~$x\in\R$ and any~$q\in\R^n\setminus\{0\}$ with~$|q|\le\delta$,
\[
0<W(x, q)\le \frac{1}{\mu} \nabla_q W(x, q)\cdot q \le \frac{\varepsilon |q|^2}{\mu}
< \varepsilon |q|^2,
\]as desired.
\end{proof}

The following results provides some appropriate upper bounds for~$W$ and its gradient and will be used when~$s\in (0, 1/2]$.

\begin{lemma}
Let~$W\in C^1(\R\times\R^n, \R)$ satisfy~\eqref{LITTLEO} and~\eqref{INPIUNUOVA}.

Then, there exist~$a_1$, $a_2>0$ such that, for any $x\in\R$ and any~$q\in\R^n$,
\begin{equation}\label{INPIU}
|\nabla_q W(x, q)|\le a_1+a_2 |q|^{p-1} .\end{equation}
\end{lemma}

\begin{proof} In light of~\eqref{LITTLEO} we know that there exists~$r_0>0$ such that if~$|q|\le r_0$ then
\begin{equation}\label{INPIU2} \sup_{x\in\R}|\nabla_q W(x,q)|\le|q|.\end{equation}
Hence, taking~$a_0$ and~$p$ as in~\eqref{INPIUNUOVA}, we define
$$ a_1:=1\qquad{\mbox{ and }}\qquad a_2:=a_0(r_0^{2-p}+1).$$
In this setting, if~$|q|\le r_0$ we have that~\eqref{INPIU} follows from~\eqref{INPIU2}; if instead~$|q|>r_0$ we use~\eqref{INPIUNUOVA} and
see that
\begin{equation*}\begin{split}& |\nabla_q W(x, q)|=\left(\sum_{j=1}^n |\partial_{q_j} W(x, q)|^2\right)^{1/2}
\le a_0\,(1+ |q|^{p-2})\left(\sum_{j=1}^n |q_j|^2\right)^{1/2}\\&\qquad=a_0\,(1+ |q|^{p-2})|q|=
a_0\,(|q|^{2-p}+1)|q|^{p-1}\le a_0\,(r_0^{2-p}+1)|q|^{p-1}=a_2 |q|^{p-1}.
\qedhere\end{split}\end{equation*}
\end{proof}

\begin{lemma}\label{lemmaW1}
Let~$W\in C^1(\R\times\R^n, \R)$ satisfy~\eqref{LITTLEO} and~\eqref{INPIU}.

Then, for any~$\varepsilon>0$ there exists~$\sigma=\sigma(\varepsilon)>0$ such that, for any~$x\in\R$ and any~$q\in\R^n$,
\begin{equation*}\begin{split}
&|\nabla_q W(x, q)|\le \varepsilon |q| + \sigma(\varepsilon)\, |q|^{p-1}\\
{\mbox{and }} \qquad &
|W(x, q)|\le \frac{\varepsilon}{2} |q|^2 +\frac{\sigma(\varepsilon)}{p} |q|^p.\end{split}
\end{equation*}
\end{lemma}

The proof of Lemma~\ref{lemmaW1}
is the multi-dimensional version of
that in~\cite[Lemma~3]{MR2879266}.

We notice that hypothesis~\eqref{AR} on its own can be used to deduce some suitable lower bounds for the potential~$W$, as stated in the next result:

\begin{lemma}
Let~$W\in C^1(\R\times\R^n, \R)$ satisfy~\eqref{AR} and
\begin{equation}\label{omega12}
\omega_1(x):= \inf_{|\xi|=1} W(x, \xi).
\end{equation}

Then, $\omega_1(x)>0$ for every~$x\in\R$.

Moreover, for any~$x\in\R$ and any~$q\in\R^n$ with~$ |q|\ge 1$,
\begin{equation}\label{WGEQ}
W(x, q)\ge \omega_1(x) |q|^\mu .
\end{equation}
\end{lemma}

\begin{proof}
We observe that the infimum in~\eqref{omega12} is attained, due to the continuity of~$W$, and it is strictly positive,
thanks to~\eqref{AR}.

Now, for any~$q\in\R^n\setminus\{0\}$, we set
\begin{equation}\label{HATQ}
\widehat q:= \frac{q}{|q|}.
\end{equation}

Consider the function~$f:\R\times (0, +\infty)\to\R$ defined as
\begin{equation*}
f(x, \rho):=\log \left(\frac{W(x, \rho\widehat q)}{W(x, \widehat q)}\right).
\end{equation*}
We notice that~$f(x, 1)=0$ and that, by~\eqref{AR}, for any~$\rho> 1$,
\begin{equation*}
f(x, \rho) = f(x, \rho) -f(x, 1) =\int_1^{\rho} \frac{\nabla_q W(x, r\widehat q)\cdot r\widehat q}{W(x, r\widehat q)} \, \frac{dr}{r}\ge \mu\int_1^{\rho} \frac{dr}{r} = \log (\rho^{\mu}).
\end{equation*}
This entails that, for any~$\rho\ge1$,
\[
W(x, \rho\widehat q)\ge W(x, \widehat q) \, \rho^{\mu}.
\]

Now, for any~$q\in\R^n$ with~$|q|\ge1$, we can take~$\rho:= |q|$ and, recalling the definition in~\eqref{HATQ}, we get
\begin{equation*}
W(x, q)\ge W\left(x, \frac{q}{|q|}\right) \, |q|^{\mu} \ge \omega_1(x) \, |q|^{\mu} ,
\end{equation*}
being~$\omega_1(x)$ as in~\eqref{omega12}. This proves the inequality in~\eqref{WGEQ}.
\end{proof}

We recall that our goal is to find homoclinic solutions to problem~\eqref{eqn2}. To this end, we observe that if~$q$ is a uniformly continuous function and~$q\in L^r(\R, \R^n)$, then~$q$ tends to zero at infinity. More precisely, the following result holds true:

\begin{lemma}\label{lemmainfty}
Let~$r\in [1, +\infty)$. Assume that~$q\in L^r(\R, \R^n)$ and~$q$ is uniformly continuous.

Then,
\begin{equation*}
\lim_{x\to \pm\infty} q(x)=0.
\end{equation*}
\end{lemma}

\begin{proof}
We establish the limit as~$x\to +\infty$, the other one being similar.
Suppose by contradiction that there exist~$\varepsilon>0$ and a sequence~$x^j$ such that~$x^{j}\to +\infty$ as~$j\to +\infty$ and
\[
|q(x^{j})| >\varepsilon.
\]
Without loss of generality, we can suppose that~$x^{j+1}>x^j+1$.

Since~$q$ is uniformly continuous, we can find~$\delta\in(0,1/2)$ small enough such that, for all~$j\in\N$ and all~$x\in [x^j-\delta, x^j +\delta]$,
\[
|q(x)| >\frac{\varepsilon}{2} .
\]
Hence, we have that
\[
\int_\R |q(x)|^r \,dx \ge \sum_{j=1}^{+\infty} \int_{x^j-\delta}^{x^j+\delta} |q(x)|^r \,dx > \sum_{j=1}^{+\infty} \frac{\delta\,\varepsilon^r}{2^{r-1}} = +\infty,
\]
which contradicts the assumption that~$q\in L^r(\R, \R^n)$. This concludes the proof.
\end{proof}

\section{Functional setting}\label{SIPDJOLNDFUOJFOJLN}
In this section we introduce the main functional analytic setting needed in the study of problem~\eqref{eqn2}.

Let~$s\in (0, 1)$ and assume that~$L\in C(\R, \R^{n\times n})$ is symmetric and satisfies~\eqref{LMAT}. For any measurable function~$q:\R\to\R^n$, we define
\begin{equation}\label{NORMH}
\|q\|_{ \widetilde H^s}:= \left([q]^2_s +\int_\R L(x) q(x)\cdot q(x) \,dx\right)^{1/2}\end{equation}
and we point out that
\begin{equation}\label{AGGNORMPROOF}
{\mbox{$\|\cdot\|_{ \widetilde H^s}$ is a norm.}}
\end{equation}
We postpone the proof of this claim in Appendix~\ref{AGGNORMPROOFSEC}.

Now we define the space~$ \widetilde H^s$ as the 
set of measurable functions~$q:\R\to\R^n$ such that~$\|q\|_{ \widetilde H^s}<+\infty$.

We remark that~$ \widetilde H^s$ is a Hilbert space with scalar product
\begin{equation*}
\langle q, \overline q\rangle_{ \widetilde H^s}= \left( c_s\iint_{\R^2} \frac{(q(x)-q(y))\cdot(\overline q(x)-\overline q(y))}{|x-y|^{1+2s}} \, dx\, dy
+\int_\R L(x) q(x)\cdot \overline q(x) \,dx\right)^{1/2}.\end{equation*}
For completeness, we establish this claim in Appendix~\ref{AGGNORMPROOFSEC22}.

In fact, we will prove also a stronger result, namely:

\begin{proposition}\label{DENNUO:AS0}
The space~$ \widetilde H^s$ coincides with the completion of~$C^\infty_0(\R,\R^n)$ with respect to the norm in~\eqref{AGGNORMPROOF}.
\end{proposition}

The proof of this result is technically delicate and deferred to Appendix~\ref{DENNUO:AS0ap}
for the facility of the reader.\smallskip

Furthermore, we notice that, in view of~\eqref{LMAT},
\begin{equation}\label{EMB1}\begin{split}
\|q\|^2_{ \widetilde H^s}&= [q]^2_s +\int_\R L(x) q(x)\cdot q(x) \,dx\\& \ge \min\{\alpha, 1\} \left([q]^2_s +\|q\|^2_{L^2(\R, \R^n)}\right)
\\&= \min\{\alpha, 1\} \|q\|^2_{H^s(\R, \R^n)}.\end{split}
\end{equation}

We now present a preliminary embedding result.
\begin{lemma}\label{S1S2}
Let~$0< s_1 < s_2 < 1$. Then, for any measurable function~$q:\R\to\R^n$,
\[
\|q\|^2_{\widetilde H^{s_1}}\le \left(\frac{1}{\alpha}+1\right) \|q\|^2_{\widetilde H^{s_2}},
\]
being~$\alpha$ given in~\eqref{LMAT}.

In particular, the space~$\widetilde H^{s_2}$ is continuously embedded in~$\widetilde H^{s_1}$.
\end{lemma}

\begin{proof}
To check this, it is convenient to write the Gagliardo seminorm in terms of the Fourier transform (see~\cite{guidagalattica}) as
\[
[q]_{s_1}=\left(\;\int_{\R} |2\pi \xi|^{2s_1}\,|\widehat q(\xi)|^2\,d\xi\right)^{\frac12}.
\]
{F}rom this, we have that
\begin{eqnarray*}
[q]_{s_1}^2&\le& \int_{B_{1/(2\pi)}} |2\pi \xi|^{2s_1}\,|\widehat q(\xi)|^2\,d\xi+\int_{\R\setminus B_{1/(2\pi)}} |2\pi \xi|^{2s_2}\,|\widehat q(\xi)|^2\,d\xi\\
&\le&  \int_{\R} |\widehat q(\xi)|^2\,d\xi+\int_{\R} |2\pi \xi|^{2s_2}\,|\widehat q(\xi)|^2\,d\xi
\\&=& \|q\|^2_{L^2(\R, \R^n)}+[q]^2_{s_2}.
\end{eqnarray*}
This, \eqref{NORMH} and~\eqref{LMAT} give that
\[
\begin{split}
\|q\|^2_{\widetilde H^{s_1}}&= [q]^2_{s_1} +\int_\R L(x) q(x)\cdot q(x) \,dx\\
& \le [q]^2_{s_2} +\|q\|^2_{L^2(\R, \R^n)} +\int_\R L(x) q(x)\cdot q(x) \,dx\\
&\le [q]^2_{s_2} +\left(\frac{1}{\alpha}+1\right) \int_\R L(x) q(x)\cdot q(x) \,dx\\
&\le \left(\frac{1}{\alpha}+1\right) \|q\|^2_{\widetilde H^{s_2}},
\end{split}
\]
as desired.
\end{proof}

It is worth noting that the space~$ \widetilde H^s$ is continuously embedded in suitable Lebesgue spaces,  depending on where the parameter~$s$ falls within the fractional range~$(0, 1)$. Precisely, we have the next result:

\begin{lemma}\label{lemmaembeddcontinua}
Let~$\alpha$ be as in~\eqref{LMAT}, $s\in(0,1)$ and~$q\in \widetilde H^s$. 

Then, the following statements hold true:
\begin{enumerate}
\item[i)] if~$s\in (0, 1/2)$, there exists a positive constant~$\overline c$ depending only on~$n$, $s$ and~$\alpha$ such that
\begin{equation}\label{FIN1}
\|q\|_{L^{r}(\R, \R^n)}\le \overline c \,\|q\|_{ \widetilde H^s} \quad\mbox{ for any } r\in [2, 2^*_s],
\end{equation}
namely the space~$ \widetilde H^s$ is continuously embedded in~$L^r(\R, \R^n)$ for any~$r\in [2, 2^*_s]$;
\item[ii)] if~$s=1/2$, there exists a positive constant~$\widehat c$ depending only on~$n$, $s$ and~$\alpha$ such that
\begin{equation}\label{FIN3}
\|q\|_{L^r(\R, \R^n)}\le \widehat c \,\|q\|_{ \widetilde H^s} \quad\mbox{ for any } r\in [2, +\infty),
\end{equation}
i.e. the space~$ \widetilde H^s$ is continuously embedded in~$L^r(\R, \R^n)$ for any~$r\in [2, +\infty)$;
\item[iii)] if~$s\in (1/2,1)$, there exists a positive constant~$\widetilde c$ depending only on~$n$, $s$ and~$\alpha$ such that
\begin{equation}\label{FIN2}
\|q\|_{L^r(\R, \R^n)}\le \widetilde c \,\|q\|_{ \widetilde H^s} \quad\mbox{ for any } r\in [2, +\infty],
\end{equation}
i.e. the space~$ \widetilde H^s$ is continuously embedded in~$L^r(\R, \R^n)$ for any~$r\in [2, +\infty]$.
\end{enumerate}
\end{lemma}

\begin{proof}
Notice that inequality~\eqref{EMB1} entails that the space~$ \widetilde H^s$ is continuously embedded~$H^s(\R, \R^n)$.

Let us consider the case~$s\in (0, 1/2)$. By~\cite[Theorem~6.7]{guidagalattica} and~\eqref{EMB1}, there exists a constant~$c_1$, depending on~$n$ and~$s$, such that
\[
\|q\|_{L^r(\R, \R^n)}\le c_1 \|q\|_{H^s(\R, \R^n)} \le \frac{c_1}{\min\{\sqrt{\alpha}, 1\}} \|q\|_{ \widetilde H^s},
\]
for any~$r\in [2, 2^*_s]$. Thus, \eqref{FIN1} holds true with~$\overline c:= c_1/\min\{\sqrt{\alpha}, 1\}$. 

Similarly, when~$s=1/2$, \eqref{FIN3} follows from~\eqref{EMB1} and~\cite[Theorem~6.10]{guidagalattica}.

We address the case~$s\in (1/2, 1)$. By~\cite[Theorem~8.2]{guidagalattica} we have that~$q\in L^\infty(\R, \R^n)$
and there exists a constant~$c_2$ depending only~$n$ and~$s$ such that
\[
\|q\|_{L^\infty(\R, \R^n)} \le c_2 \|q\|_{H^s(\R, \R^n)}.
\]
{F}rom this and the fact that~$q\in L^2(\R, \R^n)$, by interpolation we have that~$q\in L^r(\R, \R^n)$ for any~$r\in [2, +\infty]$ and
\begin{equation*}
\|q\|_{L^r(\R, \R^n)} \le c_3 \|q\|_{H^s(\R, \R^n)} \quad\mbox{ for any } r\in [2, +\infty],
\end{equation*}
with~$c_3>0$ depending only on~$n$ and~$s$.  

Hence, combining the latter inequality with~\eqref{EMB1}, we obtain that
\[
\|q\|_{L^r(\R, \R^n)} \le c_3 \|q\|_{H^s(\R, \R^n)}\le \frac{c_3}{\min\{\sqrt{\alpha}, 1\}} \|q\|_{ \widetilde H^s} \quad\mbox{ for any } r\in [2, +\infty],
\]
namely,~\eqref{FIN2} holds true
with~$\widetilde c:= c_3/\min\{\sqrt{\alpha}, 1\}$. This concludes the proof.
\end{proof}

In addition, we provide a convergence result for bounded sequences in~$ \widetilde H^s$.

\begin{proposition}\label{proposition_conv}
Let~$q^j$ be a bounded sequence in~$ \widetilde H^s$.

Then, the following statements hold true:
\begin{itemize}
\item[i)] if~$s\in (0, 1/2)$, there exists~$q\in L^r (\R, \R^n)$ such that, up to subsequences, 
\begin{equation}\label{SOTTOUNMEZZO}
q^j\to q \quad\mbox{ in } L^r_{\rm{loc}}(\R, \R^n)\quad \mbox{ for any } r\in [2, 2^*_s);
\end{equation}
\item[ii)] if~$s = 1/2$, there exists~$q\in L^r  (\R, \R^n)$ such that, up to subsequences, 
\begin{equation}\label{UGUALEUNMEZZO}
q^j\to q \quad\mbox{ in } L^r_{\rm{loc}}(\R, \R^n)\quad \mbox{ for any } r\in [2, +\infty);
\end{equation}
\item[iii)] if~$s\in (1/2, 1)$, there exists~$q\in L^r (\R, \R^n)$ such that, up to subsequences, 
\begin{equation}\label{SOPRAUNMEZZO}
q^j\to q \quad\mbox{ in } L^r_{\rm{loc}}(\R, \R^n)\quad \mbox{ for any } r\in [2, +\infty].
\end{equation}
\end{itemize}
\end{proposition}

\begin{proof}
Let~$q^j$ be a bounded sequence in~$ \widetilde H^s$. Hence, we see that~$q^j\in H^s(\R, \R^n)$ and, by inequality~\eqref{EMB1},
\[
\|q^j\|^2_{H^s(\R, \R^n)} \le\frac{1}{\min\{\alpha, 1\}} \|q^j\|^2_{ \widetilde H^s},
\]
which implies that~$q^j$ is bounded in~$H^s(\R, \R^n)$ as well. Now, if~$s\in (0, 1/2)$,  then~\eqref{SOTTOUNMEZZO} follows from~\cite[Corollary~7.2]{guidagalattica}.

Let us prove the statement in~$ii)$.
To this end, let~$q^j$ be a bounded sequence in~$ \widetilde H^{\frac12}$ and let~$\varepsilon\in(0,1/2)$. By applying Lemma~\ref{S1S2} with 
\[
s_1:= \frac12 -\varepsilon\qquad\mbox{and}\qquad s_2:=\frac12,
\]
we infer that~$q^j$ is bounded in~$\widetilde H^{\frac12-\varepsilon}$ as well. 
Hence, we can exploit the statement in~$i)$, thus obtaining that there exists~$q\in L^r_{\rm{loc}}(\R, \R^n)$ such that
\[
q^j\to q \quad\mbox{ in } L^r_{\rm{loc}}(\R)\quad \mbox{ for any } r\in \left[1, 2^*_{\frac12-\varepsilon}\right).
\]
Now, since
\[
2^*_{\frac12-\varepsilon} =\frac{2}{1-2\left(\frac12-\varepsilon\right)}= \frac{1}{\varepsilon},
\]
by sending~$\varepsilon\searrow 0$ we deduce~\eqref{UGUALEUNMEZZO}.

Finally, if~$s\in (1/2, 1)$, we have that~\eqref{SOPRAUNMEZZO} is a consequence of~\cite[Theorem~8.2]{guidagalattica} and Ascoli-Arzel\`a Theorem.

We remark that~$q\in L^r(\R,\R^n)$, since, by Lemma~\ref{lemmaembeddcontinua} and Fatou's Lemma, 
$$ \|q\|^r_{L^r(\R,\R^n)}\le \liminf_{j\to+\infty}\|q^j\|^r_{L^r(\R,\R^n)}\le C
\liminf_{j\to+\infty}\|q^j\|^r_{\widetilde{H}^s}\le C,$$
up to renaming~$C>0$.
\end{proof}

We now introduce the variational formulation of problem~\eqref{eqn2}. 
The energy functional~$I: \widetilde H^s\to\R$ associated with problem~\eqref{eqn2} is given by
\begin{equation}\label{FUN2}
I(q)=\frac{c_s}2 \iint_{\R^2} \frac{|q(x)-q(y)|^2}{|x-y|^{1+2s}} \,dx\, dy +\frac12\int_\R L(x) q(x)\cdot q(x) \,dx -\int_\R W(x, q(x)) \,dx.
\end{equation}
Moreover, it is easy to check that~$I\in C^1( \widetilde H^s, \R)$ and, for any~$q$, $\varphi\in \widetilde H^s$,
\begin{equation}\label{DIFF2}
\begin{split}
\langle I'(q),  \varphi\rangle &=c_s \iint_{\R^2} \frac{(q(x)-q(y))\cdot(\varphi(x)-\varphi(y))}{|x-y|^{1+2s}} \,dx\, dy +\int_\R L(x) q(x)\cdot \varphi(x) \,dx\\
&\qquad\qquad -\int_\R \nabla_q W(x, q(x))\cdot\varphi(x) \,dx.
\end{split}
\end{equation}

We now show that, in this framework, the geometry of the Mountain Pass Theorem is respected. 

It is worth noting that the proof of the next geometric property is influenced by the location of the exponent~$s$ within the fractional range~$(0, 1)$. 
Indeed, as stated in Lemma~\ref{lemmaembeddcontinua}, if~$s\in (1/2, 1)$, we can take advantage of the embedding into the space~$L^\infty(\R, \R^n)$, unlike the complementary case~$s\in (0, 1/2]$. 
For this reason, we distinguish two cases, according to whether~$s\in (0, 1/2]$ or~$s\in (1/2, 1)$.
Clearly, this has also influenced our choice in the assumptions on the potential~$W$, as it is evident from the next two results. 

\begin{proposition}\label{geom3}
Let~$s\in (0, 1/2]$. Let~$W\in C^1(\R\times\R^n, \R)$ satisfy~\eqref{LITTLEO} and~\eqref{INPIUNUOVA}.

Then, there exist~$\rho_1$, $\beta_1>0$ such that~$I(q)\ge \beta_1$
for any~$q\in  \widetilde H^s$ with~$\|q\|_{ \widetilde H^s}=\rho_1$.
\end{proposition}

\begin{proof}
Let~$q\in \widetilde H^s$. By~\eqref{FUN2}, \eqref{FIN1}, \eqref{FIN3} and Lemma~\ref{lemmaW1}
(recall that here~$p\in (2, 2^*_s)$), we infer that for any~$\varepsilon>0$,
\[
\begin{split}
I(q) & \ge \frac12 \|q\|^2_{ \widetilde H^s} -\frac{\varepsilon}{2} \int_\R |q|^2 \,dx -\frac{\sigma(\varepsilon)}{p}\int_\R |q|^p \,dx\\
& = \frac12 \|q\|^2_{ \widetilde H^s} -\frac{\varepsilon}{2} \|q\|^2_{L^2(\R, \R^n)} -\frac{\sigma(\varepsilon)}{p}\|q\|^p_{L^p(\R,\R^n)}\\
&\ge\frac12\big(1 -\varepsilon\, {\overline{c}}^2\big)\,\|q\|^2_{ \widetilde H^s} -\frac{\sigma(\varepsilon)\,{\overline{c}}^p}{p} \|q\|^p_{ \widetilde H^s}.
\end{split}
\]
Therefore, taking~$\varepsilon:=1/(2{\overline{c}}^{2})$, we deduce that
\[
I(q) \ge \frac14 \|q\|^2_{ \widetilde H^s} \left( 1-\frac{4\sigma}{p}
\|q\|^{p-2}_{ \widetilde H^s}\right),
\]
where now~$\sigma=\sigma(1/(2{\overline{c}}^{2}))$.

Now we choose~$\rho_1>0$ sufficiently small such that
\[
1-\frac{4\sigma}{p}\, \rho_1^{p-2}>0.
\]
In this way, if~$\|q\|_{ \widetilde H^s} =\rho_1$, we have that
\[ I(q)\ge \frac14 \rho_1^2 \left(1-\frac{4\sigma}{p}\, \rho_1^{p-2}\right)=:\beta_1 >0.
\]
Hence, the desired result is proved.
\end{proof}

\begin{proposition}\label{geom1}
Let~$s\in (1/2, 1)$. Let~$W\in C^1(\R\times\R^n, \R)$ satisfy~\eqref{LITTLEO} and~\eqref{AR}.

Then, there exist~$\rho_2$, $\beta_2>0$ such that~$I(q)\ge \beta_2$
for any~$q\in  \widetilde H^s$ with~$\|q\|_{ \widetilde H^s}=\rho_2$.
\end{proposition}

\begin{proof}
Let~$\alpha$ be as given by~\eqref{LMAT}
and take
\begin{equation}\label{EPS1}
\varepsilon:=\frac{\min\{\alpha, 1\}}{4}.
\end{equation}
By Lemma~\ref{lemmaW0}, there exists~$\delta>0$ such that, for any~$x\in\R$ and any~$q\in\R^n\setminus\{0\}$ with~$|q|\le\delta$,
\begin{equation}\label{OPICCOLO}
0<W(x, q)< \varepsilon |q|^2.
\end{equation}

Now, let~$\widetilde c$ be as in~\eqref{FIN2} and set~$
\rho_2:=\delta/\widetilde c$.
If~$q\in \widetilde H^s$ with~$\|q\|_{ \widetilde H^s} =\rho_2$, from the embedding result in~\eqref{FIN2} we deduce that, for a.e.~$x\in\R$,
\[ |q(x)|\le
\|q\|_{L^\infty(\R, \R^n)}\le \widetilde c\, \rho_2 = \delta.
\]
Accordingly, by~\eqref{OPICCOLO} and~\eqref{EMB1} we infer that
\begin{equation*}
\int_\R W(x, q(x)) \,dx < \varepsilon \int_\R |q(x)|^2 \,dx\le \frac{\varepsilon}{\min\{\alpha, 1\}} \,\|q\|_{ \widetilde H^s}
=\frac{\varepsilon}{\min\{\alpha, 1\}}\,\rho_2^2.
\end{equation*}
Recalling~\eqref{FUN2}, this and~\eqref{EPS1} entail that,
for all~$q\in \widetilde H^s$ with~$\|q\|_{ \widetilde H^s} =\rho_2$,
\begin{equation*}
I(q) =\frac12 \, \rho_2^2 -\int_\R W(x, q(x)) \,dx > \left(\frac12 - \frac{\varepsilon}{\min\{\alpha, 1\}}\right) \, \rho_2^2 =\frac14 \, \rho_2^2 =:\beta_2>0,
\end{equation*}
as desired. 
\end{proof}

Let~$\rho_1$ and~$\rho_2$ be as in Propositions~\ref{geom3} and~\ref{geom1}, respectively, and set
\begin{equation}\label{SPIJLndeiwohgeriujkghvnei89}
\rho:=\begin{cases}\displaystyle \rho_1 &{\mbox{ if }} s\in\left(0,\frac12\right], \\
\displaystyle  \rho_2 &{\mbox{ if }} s\in\left(\frac12,1\right).\end{cases}
\end{equation}
We now prove another geometric property of the functional~$I$. We notice that, contrary to what occurred in the previous result, in the present framework there is no need to distinguish between the cases~$s\in (0, 1/2]$ and~$s\in (1/2, 1)$.

\begin{proposition}\label{geom2}
Let~$s\in (0, 1)$ and~$\rho$ be as in~\eqref{SPIJLndeiwohgeriujkghvnei89}. Let~$W\in C^1(\R\times\R^n, \R)$ satisfy~\eqref{AR}.

Then, there exists~$q_0\in \widetilde H^s$ such that~$\|q_0\|_{ \widetilde H^s}>\rho$ and~$I(q_0)<0$.
\end{proposition}

\begin{proof}
Let~$t\ge1$ and
\begin{equation}\label{OVERT0}
{\mbox{$q_\diamondsuit\in \widetilde H^s$ such that~$|q_\diamondsuit|=1$ in~$(-1,1)$.}}\end{equation}
By~\eqref{FUN2}, \eqref{NORMH} and~\eqref{WGEQ}, we infer that
\[
\begin{split}
I(t q_\diamondsuit) &= \frac{c_s t^2}{2} \iint_{\R^2}\frac{|q_\diamondsuit(x)-q_\diamondsuit(y)|^2}{|x-y|^{1+2s}} \,dx\, dy +\frac{t^2}{2} \int_\R L(x) q_\diamondsuit(x)\cdot q_\diamondsuit(x) \,dx -\int_\R W(x, tq_\diamondsuit(x)) \,dx\\
&\le \frac{t^2}{2}\|q_\diamondsuit\|_{ \widetilde H^s}^2
- t^\mu \int_{-1}^1\omega_1(x) |q_\diamondsuit(x)|^\mu \,dx\\
&= \frac{t^2}{2}\|q_\diamondsuit\|_{ \widetilde H^s}^2
- t^\mu \int_{-1}^1\omega_1(x)\,dx,
\end{split}
\]
with~$\omega_1(x)$ as in~\eqref{omega12}. 

Let now
\begin{equation}\label{OVERT}
\overline t:= \max\left\lbrace \rho, \, \left(\frac2{\|q_\diamondsuit\|_{ \widetilde H^s}^2}\int_{-1}^1 \omega_1(x)\,dx\right)^{\frac1{2-\mu}}\right\rbrace+1.
\end{equation}
Thus, taking~$q_0:= \frac{\overline t q_\diamondsuit}{\|q_\diamondsuit\|_{ \widetilde H^s}   }$, we have that
\[
\|q_0\|_{ \widetilde H^s}>\rho \qquad\mbox{and} \qquad I(q_0) < 0,
\]
as desired.
\end{proof}

Now we obtain a uniform bound for sequences approaching a given energy level.
This will be a useful ingredient in combination with the elliptic regularity theory,
in order to allow for a suitable bootstrap argument and provide uniform estimates on the
solutions of the main results.

\begin{proposition}\label{PROPBOUNDED} Let~$s\in(0,1)$.
Let~$W\in C^1(\R\times\R^n, \R)$ satisfy~\eqref{AR}. Let~$c\in\R$ and let~$q^j$ be a sequence in~$ \widetilde H^s$ satisfying
\begin{equation}\label{C01}
\lim_{j\to +\infty} I(q^j) = c
\end{equation}
and
\begin{equation}\label{C02}
\lim_{j\to +\infty} \sup_{{\varphi\in \widetilde H^s}\atop{ \|\varphi\|_{ \widetilde H^s}=1}} |\langle I'(q^j), \varphi\rangle| =0.
\end{equation}

Then, $c\ge0$ and~$q^j$ is bounded in~$ \widetilde H^s$.

In particular, there exists a positive constant~$c_1$, depending on~$c$ and~$\mu$, such that
\begin{equation}\label{LIMSUP}
\limsup_{j\to +\infty} \|q^j\|_{ \widetilde H^s}\le c_1.
\end{equation}
Explicitly, one can take
\begin{equation}\label{STIMA1}
c_1:=  \left( 
\frac{2\mu \,c}{\mu-2} \right)^{1/2}.
\end{equation}
\end{proposition}

\begin{proof}
Let~$\varepsilon>0$. By~\eqref{C01} and~\eqref{C02} there exists~$\overline j=\overline j(\varepsilon)\in\N$ such\footnote{Note that we are assuming here that~$\|q^j\|_{ \widetilde H^s}\neq0$, otherwise we are done.}that, for any~$j\ge\overline j$,
\[
|I(q^j)-c|\le \varepsilon \qquad\mbox{and}\qquad \left|\left\langle I'(q^j), \frac{q^j}{\|q^j\|_{ \widetilde H^s}} \right\rangle \right|\le\varepsilon.
\]
In light of this, recalling~\eqref{FUN2} and~\eqref{DIFF2}, we find that
\begin{equation}\label{II}
\mu I(q^j) - \langle I'(q^j), q^j\rangle 
\le \mu(c+\varepsilon)+\varepsilon\|q^j\|_{ \widetilde H^s}.
\end{equation}

On the other hand, exploiting the assumption in~\eqref{AR}, we see that
\begin{equation*}
\begin{split}
\mu I(q^j) - \langle I'(q^j), q^j\rangle &=\left(\frac{\mu}{2} -1\right) \|q^j\|^2_{ \widetilde H^s} +\int_\R \Big(\nabla_q W(x, q(x))\cdot q(x) -\mu W(x, q(x)) \Big) \,dx\\
&\ge \left(\frac{\mu}{2} -1\right) \|q^j\|^2_{ \widetilde H^s}
\end{split}
\end{equation*}
Hence, combining this with~\eqref{II}, we deduce that, for any~$j\ge\overline j$,
\begin{equation}\label{BDD1}
\|q^j\|^2_{ \widetilde H^s}\le  \frac{2\mu(c+\varepsilon)}{\mu-2}+\frac{2\varepsilon}{\mu-2}\|q^j\|_{ \widetilde H^s}
.\end{equation}
This proves that~$q^j$ is bounded in~$ \widetilde H^s$.

We now want to be more accurate and find an explicit upper bound for~$\|q^j\|_{ \widetilde H^s}$. For this, we send~$\varepsilon\searrow 0$ in~\eqref{BDD1} and we find that 
\[
\limsup_{j\to +\infty}\|q^j\|_{ \widetilde H^s}\le \left( 
\frac{2\mu \,c}{\mu-2} \right)^{1/2}.
\]
This gives that~$c\ge0$ and establishes the bound in~\eqref{LIMSUP},
as desired.
\end{proof}

We now show that if~$q^j$ is a bounded sequence in~$ \widetilde H^s$ verifying~\eqref{C02}, then it converges to a critical point of the functional in~\eqref{FUN2}, namely the following result holds true:

\begin{proposition}\label{PROPOINT1}
Let~$s\in (0, 1)$ and let~$W\in C^1(\R\times\R^n, \R)$. 
If~$s\in(0,1/2]$ suppose in addition that~$W$ that satisfies~\eqref{INPIUNUOVA}.

Let~$q^j$ be a bounded sequence in~$ \widetilde H^s$.

Then, there exists~$q\in \widetilde H^s$ such that, for any~$\varphi\in C_0^\infty(\R, \R^n)$,
\begin{equation}\label{P1}
\lim_{j\to+\infty}
\langle I'(q^j), \varphi\rangle = \langle I'(q), \varphi\rangle .
\end{equation}

In addition, if the sequence~$q^j$ verifies~\eqref{C02}, then~$q$ is a critical point of the functional~$I$.
\end{proposition}

\begin{proof}
We know that~$ \widetilde H^s$ is a Hilbert space, thanks to Lemma~\ref{LemmaHilbert}.
Therefore, since~$q^j$ is a bounded sequence in~$ \widetilde H^s$, there exists~$q\in \widetilde H^s$ such that~$q^j$ weakly converges to~$q$
in~$\widetilde H^s$ as~$j\to+\infty$.
Namely, for any~$\varphi\in C_0^\infty(\R, \R^n)$,
\begin{equation}\label{WEAK1}
\begin{split}\lim_{j\to+\infty}
&c_s \iint_{\R^2} \frac{(q^j(x)-q^j(y))\cdot(\varphi(x)-\varphi(y))}{|x-y|^{1+2s}} \,dx\, dy +\int_\R L(x) q^j(x)\cdot\varphi(x) \,dx\\
&\quad =c_s \iint_{\R^2} \frac{(q(x)-q(y))\cdot(\varphi(x)-\varphi(y))}{|x-y|^{1+2s}} \,dx\, dy +\int_\R L(x) q(x)\cdot\varphi(x) \,dx.
\end{split}
\end{equation}

Furthermore, Proposition~\ref{proposition_conv} gives that there exists~$\overline{q}\in L^{2} (\R,\R^n)$ such that, up to a subsequence, 
\begin{equation}\label{conv2}
q^j \to \overline q \quad\mbox{in~$ L^{1}_{\rm{loc}}(\R,\R^n)$ and a.e. in } \R.
\end{equation}
We point out that
\begin{equation}\label{f847t5467tyreoghrwoli987654}
q=\overline q.
\end{equation} 
Not to interrupt the flow of the argument, we postpone the proof of~\eqref{f847t5467tyreoghrwoli987654}
to Appendix~\ref{f847t5467tyreoghrwoli987654SEC}.

As a consequence of~\eqref{conv2} and~\eqref{f847t5467tyreoghrwoli987654}, we have that
\begin{equation}\label{conv2BIS}
q^j \to q \quad\mbox{in~$ L^{1}_{\rm{loc}}(\R,\R^n)$ and a.e. in } \R.
\end{equation}

Now we claim that, for any~$\varphi\in C_0^\infty(\R, \R^n)$,
\begin{equation}\label{WTOPROVE}\lim_{j\to+\infty}
\int_\R \nabla_q W(x, q^j(x))\cdot\varphi(x) \,dx = \int_\R \nabla_q W(x, q(x))\cdot\varphi(x) \,dx.
\end{equation}
For this, let~$\varphi\in C_0^\infty(\R, \R^n)$ and
set~$K:=\mbox{supp}(\varphi)$. In this way, we have that
\begin{equation}\label{nour8576dgd35Awqru}
\int_\R \nabla_q W(x, q^j(x))\cdot\varphi(x) \,dx = \int_K \nabla_q W(x, q^j(x))\cdot\varphi(x) \,dx.
\end{equation}
Moreover, since~$W\in C^1(\R\times\R^n, \R)$, by~\eqref{conv2BIS},
\begin{equation}\label{CONVK}\lim_{j\to+\infty}
\nabla_q W(x, q^j(x)) = \nabla_q W(x, q(x)) \quad\mbox{ for a.e. }x\in K.
\end{equation}

We deal now with the cases~$s\in (0,1/2]$ and~$s\in (1/2, 1)$ separately.
Let us consider first the case~$s\in (1/2, 1)$. By~\eqref{SOPRAUNMEZZO} and~\eqref{f847t5467tyreoghrwoli987654},
we get that~$q^j\to q$ in~$L^\infty(K)$ as~$j\to +\infty$. Thus, setting 
\[
M:= \sup_{j\in\N}\|q^j-q\|_{L^\infty(K, \R^n)}, 
\]
we obtain that
\[
\|q^j\|_{L^\infty(K, \R^n)}\le \|q^j-q\|_{L^\infty(K, \R^n)} +\|q\|_{L^\infty(K, \R^n)}\le M +\|q\|_{L^\infty(K, \R^n)}. 
\]
Now, let~$R:=M+\|q\|_{L^\infty(K, \R^n)}$. Hence, for a.e.~$x\in K$,
\[
\left| \nabla_q W(x, q^j(x)) \cdot\varphi(x)\right| \le \|\varphi\|_{L^\infty(K, \R^n)} \sup_{(y, q)\in K\times \overline{B_R}} \left| \nabla_q W(y, q) \right|.
\]
Since the right hand side is bounded uniformly in~$j\in\N$, by~\eqref{nour8576dgd35Awqru},
\eqref{CONVK} and the Dominated Convergence Theorem, we infer that~\eqref{WTOPROVE} holds true. 

If~$s\in (0, 1/2]$, by~\eqref{SOTTOUNMEZZO}, \eqref{UGUALEUNMEZZO} and~\eqref{f847t5467tyreoghrwoli987654},
we have that~$q^j\to q$ in~$L^p(K,\R^n)$ for any~$p\in [1, 2^*_s)$. In addition, there exists~$h\in L^p (K,\R^n)$ such that 
\[
|q^j(x)|\le |h(x)| \quad\mbox{ a.e. in~$K$, for any~$j\in\N$}.
\]
Hence, by~\eqref{INPIU},
\begin{eqnarray*}&&
|\nabla_q W(x, q^j(x))|\le a_1+a_2 |q^j(x)|^{p-1} \le a_1+a_2 |h(x)|^{p-1}
\end{eqnarray*}
{F}rom this, \eqref{nour8576dgd35Awqru}, \eqref{CONVK} and the Dominated Convergence Theorem, we obtain that~\eqref{WTOPROVE} holds true
in this case as well.

Thus, combining~\eqref{DIFF2},
\eqref{WEAK1} and~\eqref{WTOPROVE}, we have that~\eqref{P1} is satisfied.

In addition, if also~\eqref{C02} holds, we have that
\[ \lim_{j\to +\infty}
\langle I'(q^j), \varphi\rangle = 0 .
\]
This, together with~\eqref{P1}, gives that~$q^j$ converges to a critical point to the functional~$I$, as desired.
\end{proof}

As a useful step towards the proofs of Theorems~\ref{maintheorem2} and~\ref{maintheorem3}, we provide the following result (see~\cite[Proposition~5.14]{Rab2} for the classical counterpart):

\begin{lemma}\label{lemma_Rab}
Let~$K>0$ and let~$q^j$ be a bounded sequence in~$ \widetilde H^s$ such that
\begin{equation}\label{L2K}
\mbox{$q^j\to 0\;$ in $L^2([-K, K], \R^n)\;$ as~$j\to+\infty$.}
\end{equation}

Then,
\[
\limsup_{j\to +\infty} \|q^j\|^2_{L^2(\R, \R^n)} \le \frac{\gamma}{\beta(K)},
\]
where
\begin{equation}\label{BKDEFN} \gamma:= \sup_{j\in\N} \|q^j\|^2_{ \widetilde H^s}\qquad {\mbox{and}}\qquad
\beta(K):= \inf_{\substack{|\xi|=1\\ |x|> K}} L(x) \xi\cdot\xi.
\end{equation}
\end{lemma}

\begin{proof}
By~\eqref{BKDEFN},
\[
\begin{split}
\|q^j\|^2_{L^2(\R, \R^n)} & =\int_{-K}^K |q^j (x)|^2 \,dx +\int_{\{|x|>K\}} |q^j (x)|^2 \,dx\\
&\le \int_{-K}^K |q^j (x)|^2 \,dx + \frac{1}{\beta(K)} \int_{\{|x|>K\}} L(x) q^j(x)\cdot q^j (x) \,dx\\
& \le \int_{-K}^K |q^j (x)|^2 \,dx + \frac{\gamma}{\beta(K)}.
\end{split}
\]
Thus, sending~$j\to +\infty$ and recalling the assumption in~\eqref{L2K}, we obtain the desired result.
\end{proof}

\section{Uniform bounds for a system of nonlocal equations with unbounded coefficients}\label{KSPldMFSw}

In this section, we provide a bound in the $L^\infty$-class for solutions of~\eqref{eqn2}.
This is an important step in our construction, and also a rather delicate issue.
Indeed, on the one hand, this type of results can be seen as an extension, for instance,
of a result in~\cite[Theorem 2.3]{MR3635980} to the case of systems of equations. 
On the other hand, this extension is nontrivial and technically involved,
due to the system structure, which appears to be new in the literature.

Also, even when~$n=1$, this type of extension requires some care since,
differently from the existing literature, presents a linear term of order zero which gets multiplied by an unbounded coefficient (in the case~$n=1$ however the proof would be simpler,
since this term would present ``the right sign'' for the estimate).

Our result, which can be applied to problem~\eqref{eqn2}, goes as follows.

\begin{theorem}[Boundedness of solutions]\label{THDMPV}
Let~$L\in C(\R, \R^{n\times n})$ satisfy~$(1.18)$ and~$(1.19)$. Moreover, assume that~$L(x)$ has nonnegative entries. 

Let
\begin{equation}\label{TI}
t:=\begin{cases}
\dfrac{2}{1-2s} &\mbox{ if } s\in (0, 1/2)\\
\mbox{any value within } (2, +\infty) &\mbox{ if } s=1/2.
\end{cases}
\end{equation}
Let~$F\in C(\R\times\R^n, \R^n)$ be such that, for any~$j\in\{1, \dots, n\}$,
\begin{equation}\label{HPF}
|F_j(x,q)|\le \sum_{\kappa=1}^K h_\kappa(x) \, |q_j(x)| |q(x)|^{\gamma_\kappa}
\end{equation}
where, for all~$\kappa\in\{1,\dots,K\}$,
\begin{equation}\label{DIST}
\begin{split}
&\gamma_\kappa\in [0,t-2)\qquad
{\mbox{and}}\qquad h_\kappa\in L^{m_\kappa}(\R,[0,+\infty)), \\
&{\mbox{ with }} \; m_\kappa \in \left( \underline{m}_\kappa,\,+\infty\right] \qquad
{\mbox{ and }}\qquad
\underline{m}_\kappa:= \displaystyle\frac{t}{t-2-\gamma_\kappa}.
\end{split}
\end{equation}
Let~$q\in\widetilde H^s(\R, \R^n)$ be a weak solution of
\[
(-\Delta)^s q(x) + L(x)q(x)= F(x, q(x))  \qquad {\mbox{for all }} x\in\R.
\]

Then, for any~$j\in\{1,\dots, n\}$,
\begin{equation}\label{vjlinfty}
\|q_j\|_{L^\infty(\R)}\le C,
\end{equation}
where~$C$ denotes a positive constant depending on~$n$, $s$, $\gamma_\kappa$, $m_\kappa$, $\|q\|_{L^{t}(\R)}$ and~$\|h_\kappa\|_{L^{m_\kappa}(\R)}$.
\end{theorem}

\begin{proof}
We will prove a stronger statement, namely that if, for all~$j\in\{1,\dots,n\}$,
\[
(-\Delta)^s q_j (x)  + \sum_{i=1}^n L_{ji}(x) q_i(x) \le F_j(x,  q(x))
\]
in the weak sense, then~$q_j$ is bounded from above (the bound from below can be obtained similarly under the corresponding growth assumptions).

Throughout the proof 
\begin{equation}\label{CCCACC}
\begin{split}
&{\mbox{we will denote by~$C>0$ a quantity that}}\\&{\mbox{may depend on~$n$, $s$, $\alpha$, $\gamma_\kappa$, $m_\kappa$, $\|q\|_{L^{t}(\R)}$ and~$\|h_\kappa\|_{L^{m_\kappa}(\R)}$,}}
\end{split}
\end{equation}
and which may change from line to line.

We also suppose, without loss of generality, that~$q_j$ does not vanish identically
for all~$j\in\{1,\dots,n\}$.

Also, we rewrite the condition on~$m_i$ in~\eqref{DIST} as
\begin{equation}\label{DIST2} 
\frac{1}{m_\kappa}<\frac{t-2-\gamma_\kappa}{t}.
\end{equation}
Now, let~$\delta\in (0, 1)$ to be chosen later (see~\eqref{deltachoice}), and, for any~$j\in\{1,\dots, n\}$, define
\begin{equation}\label{def aggiunta}
\phi_j:=\frac{\delta q_j}{\|q\|_{L^{t} (\R, \R^n)}}
\end{equation}
and set~$\phi(x):=(\phi_1(x),\dots, \phi_n(x))$.

Thus, using the notation in~\eqref{CCCACC},  we can write
\begin{equation}\label{notation}
\phi_j=C\delta q_j.
\end{equation}
Also, by~\eqref{def aggiunta},
\begin{equation}\label{0990990}
\|\phi\|_{L^t(\R, \R^n)}^t= \sum_{j=1}^n \|\phi_j\|_{L^t(\R)}^t = \frac{\delta^t}{\|q\|_{L^{t} (\R,\R^n)}^t} \sum_{j=1}^n \|q_j\|_{L^t(\R)}^t = \delta^t.
\end{equation}
Moreover, we have that~$\phi$ weakly solves
\begin{equation}\label{eqPhi}
(-\Delta)^s \phi (x)  +L(x) \phi(x) = G(x)
\end{equation}
where~$G=(G_1,\dots,G_n)$ and, for all~$j\in\{1,\dots,n\}$,
\begin{equation}\label{gjdefn}
G_j(x)\le g_j(x):=\frac{\delta}{\|q\|_{L^{t} (\R, \R^n)}} \, F_j(x, q(x)).
\end{equation}
We observe that, in light of~\eqref{HPF}, for all~$x\in\R$,
\begin{equation}\label{PG1}
|G_j(x)|\le|g_j(x)|\le \frac{\delta}{\|q\|_{L^{t} (\R, \R^n)}} |F_j(x, q(x))| \le C\delta \sum_{\kappa=1}^K h_\kappa(x) \,|q_j(x)| |q(x)|^{\gamma_\kappa}.
\end{equation}

Now, for every~$k\in\mathbb{N}$, let us define~$A_k:=1-2^{-k}$ and
\begin{equation}\label{wjkdefn}
w_{j, k}(x):=(\phi_j(x)-A_k)^+,\quad\hbox{ for all }x\in\R.
\end{equation}
Let also~$w_{k}:= (w_{1, k},\dots, w_{n, k+1})$.
By Proposition~\ref{proppartepositiva} we infer that
\begin{equation}\label{wjknellospazio}
w_{k}\in \widetilde H^s (\R, \R^n).
\end{equation}
Moreover, by construction, 
\begin{equation}\label{ww}
w_{j, k+1}\leq w_{j, k} \quad {\mbox{ a.e. in }}\R.
\end{equation}

In addition, we have that
\begin{equation}\label{subsets}
\{w_{j, k+1}>0\}\subseteq \{w_{j, k}>2^{-(k+1)}\}\quad\mbox{ for any~$k\in\mathbb{N}$}.
\end{equation}
To check this, let~$x\in\{w_{j, k+1}>0\}$. Thus,
\[
0< \phi_j(x)-A_{k+1} = \phi_j(x)- (1-2^{-(k+1)}).
\]
{F}rom this, we infer that
\[
w_{j, k}(x) =  \phi_j(x)- (1-2^{-k})> 2^{-(k+1)},
\]
and this proves~\eqref{subsets}.

Also, we point out that
\begin{equation}\label{boundPhi}
0<\phi_j(x)<2^{k+1}w_{j, k}(x)\quad\hbox{ for any }x\in \{w_{j, k+1}>0\}.
\end{equation}
Indeed, taking~$x\in\{w_{j, k+1}>0\}$ and using~\eqref{ww}, we get that
\[
\phi_j(x)- (1-2^{-(k+1)}) =w_{j, k+1}(x) \le w_{j, k}(x).
\]
This and~\eqref{subsets} entail that
\[
\begin{split}
\phi_j(x) &\le w_{j, k}(x) + 1-2^{-(k+1)} = w_{j, k}(x)  + 2^{-(k+1)}(2^{k+1}-1)\\
& < w_{j, k}(x)  + w_{j, k}(x)(2^{k+1}-1)\\
&= 2^{k+1} w_{j, k}(x),
\end{split}
\]
which completes the proof of~\eqref{boundPhi}.

Now, we recall~\eqref{TI} and we set, for any~$k\in\N$,
\begin{equation}\label{Ukkk}
U_{k}:=\|w_{k}\|_{L^{t} (\R, \R^n)}^{t}.
\end{equation}
By~\eqref{NORMH}, we see that
\begin{equation}\label{csdqDKSo0o0CL}
\begin{split}
&\|w_{k+1}\|^2_{\widetilde H^s} \\
=\;& [w_{k+1}]_s^2 + \int_{\R} L(x)w_{k+1}(x)\cdot w_{k+1}(x)\, dx\\
=\;&c_s\iint_{\R^2} \sum_{j=1}^n \frac{|w_{j, k+1}(x)-w_{j, k+1}(y)|^2}{|x-y|^{1+2s}} \,dx \,dy
+ \int_{\R} \sum_{i,j=1}^n L_{ji}(x)w_{i, k+1}(x) w_{j, k+1}(x)\, dx.
\end{split}
\end{equation}
Moreover, exploiting formula~(8.10) in~\cite{DFV}, we obtain that
\begin{equation}\label{wqpgftuiqweg}
\begin{split}
[w_{k+1}]_s^2&=c_s\iint_{\R^2}\sum_{j=1}^n \frac{|w_{j, k+1}(x)-w_{j, k+1}(y)|^2}{|x-y|^{1+2s}}\,dx\,dy \\
&\le c_s\iint_{\R^2} \sum_{j=1}^n \frac{(\phi_j(x)-\phi_j(y))(w_{j,k+1}(x)-w_{j,k+1}(y))}{|x-y|^{1+2s}}\,dx\,dy.
\end{split}
\end{equation}
Also, since~$L(x)$ consists of positive entries,
\begin{equation*}
\begin{split}
&\sum_{i,j=1}^n \int_{\R} L_{ji}(x)w_{i, k+1}(x) w_{j, k+1}(x) \,dx\\
&=\sum_{i,j=1}^n \;\int_{  \{w_{j, k+1}>0\}\cap\{w_{i, k+1}>0\}  } L_{ji}(x)(\phi_i(x)-A_{k+1}) w_{j, k+1}(x) \,dx\\
&= \sum_{i,j=1}^n  \; \int_{  \{w_{j, k+1}>0\}\cap\{w_{i, k+1}>0\}  } L_{ji}(x) \phi_i(x) w_{j, k+1}(x)\, dx\\
&\qquad- A_{k+1}  \sum_{i,j=1}^n \; \int_{  \{w_{j, k+1}>0\}\cap\{w_{i, k+1}>0\}  } L_{ji}(x) w_{j, k+1}(x) \,dx\\
&\le \sum_{i,j=1}^n \; \int_{  \{w_{j, k+1}>0\}\cap\{w_{i, k+1}>0\}  } L_{ji}(x) \phi_i(x) w_{j, k+1}(x) \,dx.
\end{split}
\end{equation*}
Hence, from this, \eqref{csdqDKSo0o0CL} and~\eqref{wqpgftuiqweg}, we infer that
\begin{equation}\label{adhjo0o0bo0o0}
\begin{split}
\|w_{k+1}\|^2_{\widetilde H^s} 
&\le c_s \iint_{\R^2}\sum_{j=1}^n\frac{(\phi_j(x)-\phi_j(y))(w_{j,k+1}(x)-w_{j,k+1}(y))}{|x-y|^{1+2s}}\,dx\,dy\\
&\qquad+ \sum_{i,j=1}^n \; \int_{  \{w_{j, k+1}>0\}\cap\{w_{i, k+1}>0\}  } L_{ji}(x) \phi_i(x) w_{j, k+1}(x) \,dx.
\end{split}
\end{equation}

We now use~\eqref{notation} and~\eqref{boundPhi} to see that,
for all~$ x\in \{w_{j, k+1}>0\}$,
\[
\delta |q_j(x)|< C^k w_{j, k}(x). 
\]
Thus, by \eqref{PG1}, for all~$x\in \{w_{j, k+1}>0\}$,
\begin{equation}\label{vqeipjo0o0vn}
|g_j(x)|\le C\delta \sum_{\kappa=1}^K h_\kappa(x) \,|q_j(x)| |q(x)|^{\gamma_\kappa} \le C^k \sum_{\kappa=1}^K h_\kappa(x)\, w_{j, k}(x) |q(x)|^{\gamma_\kappa}.
\end{equation}

Now, thanks to~\eqref{wjknellospazio}, we can use~$w_{k+1}$ as a test function
in~\eqref{eqPhi}. In this way,
employing also~\eqref{vqeipjo0o0vn}, we infer that
\begin{equation*}
\begin{split}
&c_s\iint_{\R^2}\sum_{j=1}^n\frac{(\phi_j(x)-\phi_j(y))(w_{j,k+1}(x)-w_{j,k+1}(y))}{|x-y|^{1+2s}}\,dx\,dy\\
&\qquad + \sum_{i,j=1}^n \; \int_{  \{w_{j, k+1}>0\}\cap\{w_{i, k+1}>0\}  } L_{ji}(x) \phi_i(x) w_{j, k+1}(x) \,dx\\
&\le \sum_{j=1}^n \; \int_{  \{w_{j, k+1}>0\}  } g_j(x)\,w_{j, k+1}(x)\,dx\\
&\le C^k \, \sum_{{j=1,\dots,n}\atop{\kappa=1,\dots,K}} \; \int_{  \{w_{j, k+1}>0\}  }h_\kappa(x)\, w_{j, k}(x) \, w_{j, k+1}(x)\,|q(x)|^{\gamma_\kappa}\, dx.
\end{split}
\end{equation*}
Hence, combining this and~\eqref{adhjo0o0bo0o0}, we obtain that
\begin{equation}\label{hicbw}
\|w_{k+1}\|^2_{\widetilde H^s}\le C^k \, \sum_{{j=1,\dots,n}\atop{\kappa=1,\dots,K}} \; \int_{  \{w_{j, k+1}>0\}  }h_\kappa(x)\, w_{j, k}(x) \, w_{j, k+1}(x)\,|q(x)|^{\gamma_\kappa}\, dx.
\end{equation}

In order to estimate the term in the right-hand side of~\eqref{hicbw}, we introduce a new set of parameters
\begin{equation}\label{XII}
\xi_\kappa := \frac{t-2-\gamma_\kappa}{t}-\frac{1}{m_\kappa} =1-\frac{2}{t}-\frac{\gamma_\kappa}{t}-\frac{1}{m_\kappa}.
\end{equation}
In light of~\eqref{DIST2}, we have that
\begin{equation}\label{adbqvlo0o0djb}
\xi_\kappa\in (0,1).
\end{equation}
Therefore, using~\eqref{ww}
and the H\"older inequality with exponents
\[
m_\kappa, \quad \frac{t}{2},\quad\frac{t}{\gamma_\kappa},\quad \frac{1}{\xi_\kappa},
\]
and recalling~\eqref{Ukkk},
we find that, for any~$\kappa\in\{1,\cdots,K\}$ and~$j\in\{1,\dots, n\}$,
\begin{equation}\label{P0}
\begin{split}
&\int_{  \{w_{j, k+1}>0\}  }h_\kappa(x)\, w_{j, k}(x) \, w_{j, k+1}(x)\,|q(x)|^{\gamma_\kappa}\, dx\\
\le\;&\int_{\{w_{j,k+1}>0\}} h_\kappa(x)\,w_{j, k}^2(x)\,|q(x)|^{\gamma_\kappa} \,dx\\
\le\;& \left(\;\int_{\{w_{j, k+1}>0\}} h_\kappa^{m_\kappa}(x)\,dx\right)^{\frac{1}{m_\kappa}}\,\left(\;
\int_{\{w_{j, k+1}>0\}}w_{j, k}^{t}(x)\,dx\right)^{\frac{2}{t}} \\ 
&\qquad\qquad\times
\left(\;\int_{\{w_{j, k+1}>0\}} |q(x)|^{t} \,dx\right)^{\frac{\gamma_\kappa}{t}} \left(\;\int_{\{w_{j, k+1}>0\}} 1\,dx\right)^{\xi_\kappa}
\\=\;&\|h_\kappa\|_{L^{m_\kappa}(\R)}\; \|w_{j,k}\|_{L^{t}(\R)}^{2}
\;\|q\|_{L^{t}(\R,\R^n)}^{\gamma_\kappa}\;\big| \{w_{j, k+1}>0\} \big|^{\xi_\kappa}
\\
\leq\;& \|h_\kappa\|_{L^{m_\kappa}(\R)}\; U_{k}^{\frac{2}{t}} \;\|q\|_{L^{t}(\R,\R^n)}^{\gamma_\kappa}\;\big| \{w_{j, k+1}>0\} \big|^{\xi_\kappa}.
\end{split}
\end{equation}

On the other hand, by~\eqref{Ukkk} and~\eqref{subsets}, we deduce that,
for all~$j\in\{1,\dots,n\}$,
\begin{equation*}
\begin{split}
&U_{k}=\|w_{k}\|_{L^{t}(\R, \R^n)}^{t} = \sum_{\ell=1}^n
\int_\R  |w_{\ell, k}(x)|^t\, dx=\sum_{\ell=1}^n
\int_{\{w_{\ell, k}>0\}}  w_{\ell, k}^t(x)\, dx
\\&\qquad\ge \int_{\{w_{j, k}>0\}}  w_{j, k}^t(x)\, dx
\ge \int_{\{w_{j, k}>2^{-(k+1)}\}}{w_{j, k}^{t}(x)\,dx}\\
&\qquad
\ge 2^{-t(k+1)}\big|\{w_{j, k}>2^{-(k+1)}\}\big|
\ge 2^{-t(k+1)}\big|\{w_{j, k+1}>0\}\big|,
\end{split}
\end{equation*}
and thus, for all~$j\in\{1,\dots,n\}$,
\[
\big|\{w_{j, k+1}>0\}\big|^{\xi_\kappa}\leq \big(2^{t(k+1)}U_{k}\big)^{\xi_\kappa}
= 2^{\xi_\kappa t(k+1)}\, U_{k}^{\xi_\kappa}.
\]
Using this into~\eqref{P0} we thereby find that
\begin{equation}\label{TO2Y}
\int_{  \{w_{j, k+1}>0\}  }h_\kappa(x)\, w_{j, k}(x) \, w_{j, k+1}(x)\,|q(x)|^{\gamma_\kappa}\, dx\le C \, 2^{\xi_\kappa t(k+1)}\, U_{k}^{\tau_\kappa},
\end{equation}
where
\[
\tau_\kappa:=\frac{2}{t}+\xi_\kappa.
\]
We remark that~\eqref{adbqvlo0o0djb} entails that
\begin{equation}\label{TO9}
\tau_\kappa >\frac{2}{t}.
\end{equation}

Now, combining~\eqref{hicbw} and~\eqref{TO2Y}, and recalling~\eqref{CCCACC} and~\eqref{XII}, we
conclude that
\begin{eqnarray*}&&
\|w_{k+1}\|^2_{\widetilde H^s} \le C^k  \sum_{{j=1,\dots,n}\atop{\kappa=1,\dots,K}}  2^{\xi_\kappa t(k+1)}U_{k}^{\tau_\kappa}
\le C^k \, 2^{(t-2)(k+1)}\sum_{{j=1,\dots,n}\atop{\kappa=1,\dots,K}}  U_{k}^{\tau_\kappa}\\
&&\qquad \qquad=n \,C^k \, 2^{(t-2)(k+1)}\sum_{\kappa=1}^{K}U_{k}^{\tau_\kappa}
\le (C+1)^{k} \sum_{\kappa=1}^K U_{k}^{\tau_\kappa}.
\end{eqnarray*}
That is, we can write
\begin{equation*}
\|w_{k+1}\|^2_{\widetilde H^s} \le \widetilde{C}^{k} \sum_{\kappa=1}^K U_{k}^{\tau_\kappa},
\end{equation*}
with
\begin{equation}\label{C>1}
\widetilde{C}>1.
\end{equation}
As a consequence of this, by~\eqref{TI}, Lemma~\ref{lemmaembeddcontinua}
and~\eqref{Ukkk},
\begin{equation}\label{TO5}
U_{k+1}^{\frac{2}{t}} \le \widetilde{C}^{k} \sum_{i=\kappa}^K U_{k}^{\tau_\kappa}.
\end{equation}

Now, by~\eqref{0990990},
\begin{equation}\label{step 0}
\begin{split}
&U_{0}= \|w_0\|_{L^t(\R, \R^n)}^t =\sum_{j=1}^n \|w_{j, 0}\|_{L^t(\R)}^t
=\sum_{j=1}^n\|\phi_j^+\|_{L^{t}(\R)}^{t}\\
&\qquad\le \sum_{j=1}^n \|\phi_j\|_{L^{t}(\R)}^{t}
= \|\phi\|_{L^t(\R, \R^n)}^t=\delta^{t}.
\end{split}
\end{equation}
As a matter of fact, we have that~$U_{k}\le U_{0}\le 1$ for all~$k\in\N$.

As a result, setting
\begin{equation}\label{tau-bis}
\tau:=\min\{\tau_1,\dots,\tau_K\},
\end{equation}
we deduce from~\eqref{TO5} that, for all~$ k\in\N$,
\[
U_{k+1}^{\frac{2}{t}} \le \widetilde{C}^{k} U_{k}^{\tau}.
\]
This gives that, for all~$ k\in\N$,
\begin{equation}\label{forallk}
U_{k+1} \le \widetilde{C}^{k} \, U_{k}^{\vartheta},
\end{equation}
being~$\vartheta:= t \tau/2$. 
We point out that, thanks to~\eqref{TO9} and~\eqref{tau-bis},
\begin{equation}\label{vartheta>1}
\vartheta>1.
\end{equation}

We now take
\begin{equation}\label{deltachoice}
\delta:= \widetilde{C}^{-\frac1{t(\vartheta-1)^2}}\qquad
{\mbox{and}}\qquad
\mu:=\widetilde{C}^{-\frac{1}{\vartheta-1}}.
\end{equation}
Notice that, thanks~\eqref{C>1} and~\eqref{vartheta>1},
we have that~$\mu\in (0, 1)$. 

We claim that, for all~$k\in\N$,
\begin{equation}\label{claimfin}
U_{k}\le \delta^{t} \mu^k.
\end{equation}
To check this, we argue by induction over~$k\in\N$.
If~$k=0$, then~\eqref{claimfin} follows by~\eqref{step 0}.

We now assume that~\eqref{claimfin} holds true for~$k$ and we prove it for~$k+1$.
For this, we use~\eqref{forallk} and the inductive assumption, we recall
the definitions in~\eqref{deltachoice} and we find that
\[
U_{k+1}\le  \widetilde{C}^{k} \, U_{k}^{\vartheta}\le
\widetilde{C}^k (\delta^{t} \mu^k)^{\vartheta} = \widetilde{C}^k (\delta^{t} \mu^k)^{1+(\vartheta-1)} = \delta^{t} \mu^k (\widetilde{C}\, \mu^{\vartheta-1})^k \, \delta^{t(\vartheta-1)} = \delta^{t} \mu^{k+1},
\]
which completes the proof of~\eqref{claimfin}.

Now, since~$\mu\in (0, 1)$, by~\eqref{claimfin} it follows that
\begin{equation}\label{agvdsbhjk}
\lim_{k\to+\infty} U_{k}=0.
\end{equation}
Since~$0\leq w_{j, k}\leq |\phi_j|\in L^{t}(\R)$ for any~$k\in\mathbb{N}$ and 
\[
\lim_{k\rightarrow+\infty}{w_{j, k}}=(\phi_j-1)^+ \quad\mbox{ a.e. in } \R,
\]
by~\eqref{agvdsbhjk} and the Dominated Convergence Theorem, we get that
\[
0=
\lim_{k\rightarrow+\infty}U_{k}= \lim_{k\rightarrow+\infty} \|w_k\|^t_{L^t(\R, \R^n)} = \lim_{k\rightarrow+\infty} \sum_{j=1}^n \|w_{j, k}\|^t_{L^t(\R)}= \sum_{j=1}^n \|(\phi_j-1)^+\|_{L^{t}(\R)}^{t}.
\]
Therefore, for any~$j\in\{1,\dots, n\}$, it holds that~$\phi_j\leq 1$ a.e. in~$\R$.

Thus, recalling the definition of~$\phi_j$ in~\eqref{def aggiunta}, we gather that,
for all~$j\in\{,1\dots,n\}$,
\[
q_j(x)\le \frac{\|q\|_{L^{t}(\R, \R^n)}}{\delta}<+\infty,
\]
as desired.
\end{proof}

\section{Setting up a barrier}\label{PKSJLDNPHFOIKBFUIJBN}

The goal of this section is to construct a suitable barrier which will be useful to study the behaviour at infinity of the solution~$q$ to problem~\eqref{eqn2} when~$s\in (0, 1/2]$. 
Our strategy relies upon the results presented in~\cite[Theorem~2]{MR3081641}.

Throughout this section, we suppose that~$s\in(0,1/2]$
and that~$L$ satisfies the assumption in~\eqref{AGGIMPPO}.

We apply~\cite[Theorem~2]{MR3081641} with~$\mathcal W:\R\to [0, +\infty)$ given by
\[
\mathcal W(x):=\frac{(1-x^2)^2}4.
\]
It is easy to check that it meets all the assumptions required by~\cite[Theorem~2]{MR3081641}, namely~$\mathcal W$ is a smooth double-well potential with wells only at~$\pm 1$ and it satisfies~$\mathcal W''(\pm 1)=2>0$.

Hence, there exists a strictly increasing function~$\widetilde{\beta}\in C^2(\R)$ which solves
\begin{equation}\label{problematilde}
(-\Delta)^s \widetilde{\beta} = \widetilde{\beta} -\widetilde{\beta}^3\quad\mbox{ in } \R.
\end{equation}
In addition, $\widetilde\beta$ is unique (up to translations) in the class of monotone solutions of~\eqref{problematilde} and there exists~$C\ge 1$ such that,
for any large~$x\in\R$,
\[
|\widetilde\beta(x) -\mbox{sign}(x)|\le \frac{C}{|x|^{2s}}.
\]

Let~$\beta:= \widetilde{\beta}'$ and notice that~$\beta$ is a solution of
\begin{equation}\label{betaa}
(-\Delta)^s \beta = a (x)\beta \quad\mbox{ in }\R,
\end{equation}
where
\[
a(x):=1-3\widetilde{\beta}^2(x).
\]
We have that~$\beta>0$ and, by~\cite[Theorem~2]{MR3081641}, for any large~$x\in\R$,
\begin{equation}\label{betaatinfinity}
\beta(x)\le\frac{C}{|x|^{1+2s}}.
\end{equation}
Also, we point out that
\begin{equation}\label{limitea}
a(x)\in (-2, 1] \quad\mbox{ and }\quad \lim_{x\to\pm\infty} a(x) = -2.
\end{equation}

Now, let~$j\in\{1, \dots, n\}$ and~$\eta>0$. We define
\begin{eqnarray}\label{vjmu}
v_{j, \eta}(x)&:=& A \beta(x) -q_j(x) +\eta(1+x_+^s+x_-^s)
\\
{\mbox{and }}
\label{wjmu}\quad
w_{j, \eta}(x)&:=& A \beta(x) +q_j(x) +\eta(1+x_+^s+x_-^s),
\end{eqnarray}
with~$A>0$ (later on, $A$ will be chosen sufficiently large).

\begin{lemma}\label{teoremabarriera}
Let~$W\in C^1(\R\times\R^n, \R)$ satisfy~\eqref{LITTLEO} and~\eqref{INPIUNUOVA} and
let~$q\in L^\infty(\R, \R^n)$ be a solution of~\eqref{eqn2}.

Let~$v_{j, \eta}$ and~$w_{j, \eta}$ be
as in~\eqref{vjmu} and~\eqref{wjmu}, respectively. 

Then, there exists~$A_0>0$ such that if~$A\ge A_0$, for any $j\in\{1,\dots, n\}$, any~$\eta>0$ and any~$x\in\R$, we have that
\begin{equation}\label{claimvjmu}
v_{j, \eta}(x)\ge 0 
\qquad {\mbox{and}}\qquad
w_{j, \eta}(x)\ge 0.
\end{equation}
\end{lemma}

\begin{proof}
We prove the first claim in~\eqref{claimvjmu}. For this, we suppose by contradiction that
there exist~$j\in\{1,\dots, n\}$
and~$\eta>0$ such that
\begin{equation}\label{infvjmu}
\inf_{x\in \R} v_{j, \eta} (x) <0.
\end{equation}
Let~$x_m \in\R$
be a minimizing sequences for the function~$v_{j, \eta}$. That is
\begin{equation}\label{89BIS}
\lim_{m\to +\infty} v_{j, \eta}(x_m) = \inf_{x\in \R} v_{j, \eta}(x)<0
.\end{equation}

We notice that there exists~$M_{j,\eta}>0$ such that
\begin{equation}\label{5476bvthffbkjewgkjera9876543}
x_m\in[-M_{j,\eta}, M_{j,\eta}].
\end{equation}
Indeed, if~$x_m\to \pm\infty$ as~$m\to +\infty$, then by~\eqref{vjmu} and~\eqref{betaatinfinity}, and recalling that~$q\in L^\infty(\R, \R^n)$, it would follow that
\[
\lim_{m\to +\infty} v_{j, \eta}(x_m) =\lim_{m\to +\infty}\Big(
 A \beta(x_m) -q_j(x_m) +\eta\big(1+(x_m)_+^s+(x_m)_-^s\big)\Big)= +\infty,
\]
which gives a contradiction with~\eqref{89BIS}. Hence, \eqref{5476bvthffbkjewgkjera9876543} is established.

Furthermore, we point out that~$(-\Delta)^s q_j \in L^\infty([-M_{j,\eta}-1, M_{j,\eta}+1])$
by~\eqref{eqn2}, and therefore, by fractional elliptic regularity theory, we have that~$q_j$
is continuous in~$[-M_{j,\eta}, M_{j,\eta}]$. 
Consequently, in light of~\eqref{5476bvthffbkjewgkjera9876543}, 
we infer that there exists~$\overline{x}_{j, \eta}\in \R$ such that
\begin{equation}\label{minimonegativo}
v_{j, \eta}(\overline{x}_{j, \eta})=\inf_{x\in \R} v_{j, \eta}(x) < 0.\end{equation}
This and~\eqref{vjmu} entail that
\begin{equation}\label{qjpositiva}
q_j(\overline{x}_{j, \eta}) > A \beta(\overline{x}_{j, \eta}) +\eta (1+|\overline{x}_{j, \eta}|^s) >0.
\end{equation}

In view of~\eqref{AGGIMPPO}, we pick~$\overline R\ge D$ such that,
if~$|x|\ge \overline R$, we have that~$d_j(x)> 3 +a_0\big(1+\|q\|^{p-2}_{L^\infty (\R, \R^n)}\big)$.
Recall that~$a_0>0$ and~$p\in(2,2^*_s)$ here are given by~\eqref{INPIUNUOVA}.

Moreover, due to~\eqref{limitea},
there exists~$\widetilde R>0$ such that, if~$|x|\ge \widetilde R$, then $a(x)\in(-2,-1]$.

We also recall that, for all~$x>0$,
\begin{equation}\label{qwertyuioasdfghjkzxcvbnmYTFDSHGFD}
(-\Delta)^s x^s_-= -\frac{C_s}{|x|^s},\end{equation}
for some~$C_s>0$, depending only on~$s$, see e.g.~\cite[Theorem~3.4.1]{MR3469920}.

We now choose
\begin{equation}\label{choiurefs1232099078675654} R_0:=\max\left\{ \overline R, \widetilde R, 1+C_s^{\frac1{2s}}\right\}\end{equation} and
we see that, if~$|x|\ge R_0$, 
\begin{equation}\label{Asufflarge0}
d_j(x)> 3 +a_0\Big(1+\|q\|^{p-2}_{L^\infty (\R, \R^n)}\Big)\qquad{\mbox{and}}\qquad
a(x)\in(-2,-1].
\end{equation}

We now define
\begin{equation}\label{choiurefs12320990786756542}\varsigma_{R_0}:=\min_{x\in[-R_0,R_0]}\beta(x)>0
\qquad{\mbox{and}}\qquad A_{R_0}:=\frac{ \|q\|_{L^\infty (\R, \R^n)}}{\varsigma_{R_0}}.\end{equation}
By~\eqref{qjpositiva} we know that~$\|q\|_{L^\infty(\R, \R^n)} > A\beta(\overline{x}_{j, \eta})$ and consequently,
if~$A> A_{R_0}$, for all~$x\in[-R_0,R_0]$,
\[
\beta(\overline{x}_{j, \eta})< \frac{\|q\|_{L^\infty(\R, \R^n)}}{A}
< \frac{\|q\|_{L^\infty(\R, \R^n)}}{A_{R_0}}=
\varsigma_{R_0}\le\beta(x)
.\]
For this reason, if~$A> A_{R_0}$,
we have that~$|\overline{x}_{j, \eta}|> R_0$.

This and~\eqref{Asufflarge0} give that, if~$A> A_{R_0}$,
\begin{equation}\label{f8o4erytghldsdrehbbgfmnbvcLLLL} 
d_j(\overline{x}_{j, \eta})> 3 +a_0\Big(1+\|q\|^{p-2}_{L^\infty (\R, \R^n)}\Big)\qquad{\mbox{and}}\qquad
a(\overline{x}_{j, \eta})\in(-2,-1].\end{equation}

Now, suppose that~$\overline{x}_{j, \eta}>R_0$ (the case~$\overline{x}_{j, \eta}<-R_0$ is analogous
and can be dealt with replacing~$(-\Delta)^s x_-^s$ with~$(-\Delta)^s x_+^s$ in
the forthcoming formula~\eqref{74865497iuewgfgfffjewdjskafkjgfs}).
In this case, by~\eqref{vjmu}, \eqref{betaa} and~\eqref{eqn2}, we infer that
\begin{equation}\label{74865497iuewgfgfffjewdjskafkjgfs}
\begin{split}
0&\ge (-\Delta)^s v_{j, \eta} (\overline{x}_{j, \eta}) \\&= A \, a(\overline{x}_{j, \eta}) \beta(\overline{x}_{j, \eta}) +\sum_{i=1}^n L_{i, j}(\overline{x}_{j, \eta}) q_i (\overline{x}_{j, \eta}) -\partial_{q_j} W(\overline{x}_{j, \eta}, q(\overline{x}_{j, \eta}))+\eta(-\Delta)^s x_-^s(\overline{x}_{j,\eta}).
\end{split}\end{equation}

Now we point out that, in light of the assumption in~\eqref{AGGIMPPO},
\begin{equation}\label{BISINPIUNUOVA} \sum_{i=1}^n L_{i, j}(\overline{x}_{j, \eta}) q_i (\overline{x}_{j, \eta}) =
\sum_{i=1}^n \delta_{ij}d_i(\overline{x}_{j, \eta}) q_i (\overline{x}_{j, \eta}) =
d_j(\overline{x}_{j, \eta}) q_j (\overline{x}_{j, \eta}).
\end{equation}
Plugging this information and~\eqref{qwertyuioasdfghjkzxcvbnmYTFDSHGFD}
into~\eqref{74865497iuewgfgfffjewdjskafkjgfs}, we conclude that
$$ \frac{\eta\,C_s}{\overline{x}_{j, \eta}^s}\ge
A \, a(\overline{x}_{j, \eta}) \beta(\overline{x}_{j, \eta}) +d_j(\overline{x}_{j, \eta}) q_j (\overline{x}_{j, \eta}) -\partial_{q_j} W(\overline{x}_{j, \eta}, q(\overline{x}_{j, \eta})).$$
Using~\eqref{INPIUNUOVA}, we thus get that
\begin{eqnarray*}
\frac{\eta\,C_s}{\overline{x}_{j, \eta}^s}&\ge&
A \, a(\overline{x}_{j, \eta}) \beta(\overline{x}_{j, \eta})
+d_j(\overline{x}_{j, \eta}) q_j (\overline{x}_{j, \eta})
- a_0\, q_j(\overline{x}_{j, \eta})\Big(1+ |q(\overline{x}_{j, \eta})|^{p-2}\Big)\\
&\ge&A \, a(\overline{x}_{j, \eta}) \beta(\overline{x}_{j, \eta})
+d_j(\overline{x}_{j, \eta}) q_j (\overline{x}_{j, \eta})
- a_0\, q_j(\overline{x}_{j, \eta})\Big(1+  \|q\|^{p-2}_{L^\infty(\R, \R^n)}\Big).
\end{eqnarray*}

Thus, recalling
the definition of~$v_{j,\eta}$ in~\eqref{vjmu}, and also~\eqref{minimonegativo} and the second claim in~\eqref{f8o4erytghldsdrehbbgfmnbvcLLLL},
we obtain that
\begin{eqnarray*}&&
\frac{\eta\,C_s}{\overline{x}_{j, \eta}^s}\\
&\ge& a(\overline{x}_{j, \eta}) \Big(v_{j, \eta}(\overline{x}_{j, \eta})+q_j(\overline{x}_{j, \eta}) -\eta(1+\overline{x}_{j, \eta}^s)
\Big)
+d_j(\overline{x}_{j, \eta}) q_j (\overline{x}_{j, \eta})
- a_0\, q_j(\overline{x}_{j, \eta})\Big(1+  \|q\|^{p-2}_{L^\infty(\R, \R^n)}\Big)\\
&\ge&a(\overline{x}_{j, \eta}) q_j(\overline{x}_{j, \eta}) +\eta(1+\overline{x}_{j, \eta}^s)
+d_j(\overline{x}_{j, \eta}) q_j (\overline{x}_{j, \eta})
- a_0\, q_j(\overline{x}_{j, \eta})\Big(1+  \|q\|^{p-2}_{L^\infty(\R, \R^n)}\Big)\\&=&
q_j(\overline{x}_{j, \eta})\left( a(\overline{x}_{j, \eta}) 
+d_j(\overline{x}_{j, \eta}) 
- a_0\Big(1+  \|q\|^{p-2}_{L^\infty(\R, \R^n)}\Big)\right) +\eta(1+\overline{x}_{j, \eta}^s)
.
\end{eqnarray*}
Hence, exploiting also the first claim in~\eqref{f8o4erytghldsdrehbbgfmnbvcLLLL} we deduce that
\begin{eqnarray*}
\frac{\eta\,C_s}{\overline{x}_{j, \eta}^s}&\ge& 
q_j(\overline{x}_{j, \eta})\left( a(\overline{x}_{j, \eta}) 
+3+a_0\Big(1+  \|q\|^{p-2}_{L^\infty(\R, \R^n)}\Big)
- a_0\Big(1+  \|q\|^{p-2}_{L^\infty(\R, \R^n)}\Big)\right) +\eta(1+\overline{x}_{j, \eta}^s)\\&\ge& q_j(\overline{x}_{j, \eta}) +\eta(1+\overline{x}_{j, \eta}^s)\\&>&\eta \,\overline{x}_{j, \eta}^s.
\end{eqnarray*}

All in all, recalling the choice of~$R_0$ in~\eqref{choiurefs1232099078675654}, we conclude that
$$ C_s^{\frac1{2s}}>\overline{x}_{j, \eta}> R_0\ge 1+C_s^{\frac1{2s}},$$
which gives the desired contradiction and establishes the first claim
in~\eqref{claimvjmu}.

We now focus on the second claim in~\eqref{claimvjmu}. The proof follows the same line as the one of the first claim, but we provide
the details here for the facility of the reader.

We argue by contradiction and suppose that there exist~$j\in\{1,\dots,n\}$ and~$\eta>0$
such that
$$    \inf_{x\in\R} w_{j, \eta} (x) <0.$$
We consider a minimizing sequence~$y_m$ and, arguing as for~\eqref{5476bvthffbkjewgkjera9876543}, we point out that~$y_m$ is a bounded sequence
and therefore there exists~$\overline{y}_{j,\eta}\in\R$ such that
$$ w_{j, \eta}(\overline{y}_{j,\eta})=
\inf_{x\in\R} w_{j, \eta} (x) <0.$$
In light of~\eqref{wjmu}, this says that
\begin{equation*}
-q_j(\overline{y}_{j, \eta}) > A\beta(\overline{y}_{j, \eta}) +\eta  (1+\overline{y}_{j, \eta}^s) >0.
\end{equation*}

We pick now~$R_0$ as in~\eqref{choiurefs1232099078675654} and~$A_{R_0}$ as in~\eqref{choiurefs12320990786756542}
and we see that, if~$A>A_{R_0}$ then~$|\overline{y}_{j, \eta}|>R_0$.
This and~\eqref{Asufflarge0} give that, if~$A> A_{R_0}$,
\begin{equation*} 
d_j(\overline{y}_{j, \eta})> 3 +a_0\Big(1+\|q\|^{p-2}_{L^\infty (\R, \R^n)}\Big)\qquad{\mbox{and}}\qquad
a(\overline{y}_{j, \eta})\in(-2,-1].\end{equation*}

Now, 
we suppose without loss of generality that~$\overline{y}_{j, \eta}>R_0$ and,
by~\eqref{wjmu}, \eqref{betaa} and~\eqref{eqn2}, we infer that
\begin{eqnarray*}
0&\ge& (-\Delta)^s w_{j, \eta} (\overline{y}_{j, \eta}) \\&=& A \, a(\overline{y}_{j, \eta}) \beta(\overline{y}_{j, \eta}) - \sum_{i=1}^{n}
L_{i, j}(\overline{y}_{j, \eta}) q_i (\overline{y}_{j, \eta}) +\partial_j W(\overline{y}_{j, \eta}, q(\overline{y}_{j, \eta}))+\eta(-\Delta)^s x_-^s(\overline{y}_{j,\eta}).
\end{eqnarray*}
Thus, recalling~\eqref{INPIUNUOVA}, \eqref{qwertyuioasdfghjkzxcvbnmYTFDSHGFD} and~\eqref{BISINPIUNUOVA},
\begin{eqnarray*}&&
\frac{\eta \,C_s}{ \overline{y}_{j,\eta}^s}\\&\ge& 
A \, a(\overline{y}_{j, \eta}) \beta(\overline{y}_{j, \eta}) - d_j(\overline{y}_{j, \eta})
q_j (\overline{y}_{j, \eta}) +a_0\,q_j (\overline{y}_{j, \eta}) \Big(1+  \|q\|^{p-2}_{L^\infty(\R, \R^n)}\Big)\\
&\ge& a(\overline{y}_{j, \eta}) 
\Big(w_{j, \eta}(\overline{y}_{j, \eta})-q_j(\overline{y}_{j, \eta}) -\eta(1+\overline{y}_{j, \eta}^s)
\Big)
- d_j(\overline{y}_{j, \eta})
q_j (\overline{y}_{j, \eta}) +a_0\,q_j (\overline{y}_{j, \eta}) \Big(1+  \|q\|^{p-2}_{L^\infty(\R, \R^n)}\Big)\\&\ge&
- a(\overline{y}_{j, \eta}) q_j(\overline{y}_{j, \eta}) +\eta(1+\overline{y}_{j, \eta}^s)
- d_j(\overline{y}_{j, \eta})
q_j (\overline{y}_{j, \eta}) +a_0\,q_j (\overline{y}_{j, \eta}) \Big(1+  \|q\|^{p-2}_{L^\infty(\R, \R^n)}\Big)\\&=&
- q_j(\overline{y}_{j, \eta})\left(a(\overline{y}_{j, \eta})+ d_j(\overline{y}_{j, \eta})-a_0 \Big(1+  \|q\|^{p-2}_{L^\infty(\R, \R^n)}\Big)\right)
+\eta(1+\overline{y}_{j, \eta}^s)\\&\ge&
- q_j(\overline{y}_{j, \eta})\left( a(\overline{y}_{j, \eta})+ 3+a_0 \Big(1+  \|q\|^{p-2}_{L^\infty(\R, \R^n)}\Big)-a_0 \Big(1+  \|q\|^{p-2}_{L^\infty(\R, \R^n)}\Big)\right)
+\eta(1+\overline{y}_{j, \eta}^s)\\&\ge&
- q_j(\overline{y}_{j, \eta})+\eta(1+\overline{y}_{j, \eta}^s)\\&>&\eta \overline{y}_{j, \eta}^s
.\end{eqnarray*}
As a consequence,
$$  C_s^{\frac1{2s}}>\overline{y}_{j, \eta}> R_0\ge 1+C_s^{\frac1{2s}},$$
which gives the desired contradiction and establishes the second claim
in~\eqref{claimvjmu} as well.
\end{proof}

By Lemma~\ref{teoremabarriera} we deduce the following result:

\begin{corollary}\label{COR1}
Let~$W\in C^1(\R\times\R^n, \R)$ satisfy~\eqref{LITTLEO} and~\eqref{INPIUNUOVA} and let~$q\in L^\infty(\R, \R^n)$ be a solution of~\eqref{eqn2}.  

Then,
\[
\lim_{x\to \pm\infty} q(x) =0.
\]
\end{corollary}

\begin{proof}
Lemma~\ref{teoremabarriera} gives that, if~$A$ is sufficiently large, for any $j\in\{1,\dots, n\}$, any~$\eta>0$ and any~$x\in\R$,
$$ |q_j(x)|\le A \beta(x) +\eta(1+x_+^s+x_-^s).$$
Taking the limit as~$\eta\searrow0$,
$$ |q_j(x)|\le A \beta(x).$$
As a consequence,
recalling~\eqref{betaatinfinity}, we obtain the desired limit.
\end{proof}

\section{Proofs of Theorems~\ref{maintheorem2} and~\ref{maintheorem3}}\label{SPKJODLN}

In this section we suppose that the assumptions of Theorems~\ref{maintheorem2} (when~$s\in(1/2,1)$)
and~\ref{maintheorem3} (when~$s\in(0,1/2]$) are satisfied.

We introduce the following notation. We recall the space~$\widetilde{H}^s$ introduced at the beginning
of Section~\ref{SIPDJOLNDFUOJFOJLN} and the functional~$I$ in~\eqref{FUN2} and we
let~$q_0$ be given by Proposition~\ref{geom2}. We set
\begin{equation}\label{CALH}
\mathcal H:=\Big\{ h\in C([0, 1],  \widetilde H^s) \;{\mbox{ s.t. }}\; h(0)=0 \mbox{ and } h(1)= q_0 \Big\}
\end{equation}
and
\begin{equation}\label{INFMAX}
c:= \inf_{h\in\mathcal H} \max_{\eta\in [0, 1]} I(h(\eta)).
\end{equation}

We notice that the function~$[0, 1]\ni\eta\mapsto \eta q_0$ belongs to~$\mathcal H$. Hence, by~\eqref{INFMAX}, we see that
\begin{equation}\label{CSUP}
c \le \max_{\eta\in [0, 1]} I(\eta q_0).
\end{equation}
We now prove that~$c$ is bounded by a constant depending only on the structural parameters of the problem.

\begin{proposition}\label{ICANDOIT}
Let~$q_0$ be given by Proposition~\ref{geom2}.
Then, there exists a positive constant~$c_1$ depending only on~$n$, $s$, $W$ and~$L$ such that
\begin{equation*}
\max_{\eta\in [0, 1]}  \left(\frac{c_s\eta^2}{2} \iint_{\R^2} \frac{|q_0(x)-q_0(y)|^2}{|x-y|^{1+2s}} \,dx\, dy +\frac{\eta^2}{2}\int_\R L(x) q_0(x)\cdot q_0(x) \,dx -\int_\R W(x, \eta q_0(x)) \,dx\right) \le c_1.
\end{equation*}
In particular,
\begin{equation}\label{CC1}
c\le c_1.
\end{equation}
\end{proposition}

\begin{proof}
We consider the function~$g:[0,1]\to\R$ defined as
\[
g(\eta):= \frac{\eta^2}{2} \iint_{\R^2} \frac{|q_0(x)-q_0(y)|^2}{|x-y|^{1+2s}} \,dx\, dy +\frac{\eta^2}{2}\int_\R L(x) q_0(x)\cdot q_0(x) \,dx -\int_\R W(x, \eta q_0(x)) \,dx.
\]
By~\eqref{FUN2}, \eqref{AR} and the fact that~$I(q_0)<0$
(recall Proposition~\ref{geom2}), we have that
\[
\begin{split}
g(\eta)&= \eta^2 I(q_0) +\eta^2 \int_\R W(x, q_0(x)) \,dx -\int_\R W(x, \eta q_0(x)) \,dx\\
&\le \eta^2 \int_\R W(x, q_0(x)) \,dx -\int_\R W(x, \eta q_0(x)) \,dx\\
&\le \eta^2 \int_\R W(x, q_0(x)) \,dx.
\end{split}
\]
{F}rom this, we conclude that
\begin{equation}\label{ADZ}
\max_{\eta\in [0, 1]} g(\eta)\le \int_\R W(x, q_0(x)) \,dx.
\end{equation}

We recall that~$q_0:=\overline t q_\diamondsuit$, being~$q_\diamondsuit$ and~$\overline t$ be as in~\eqref{OVERT0}
and~\eqref{OVERT}, respectively. Then, the first claim in Proposition~\ref{ICANDOIT} follows from~\eqref{ADZ}.

Moreover, this also entails the claim in~\eqref{CC1} thanks to~\eqref{CSUP}. 
\end{proof}

With this, we are in the position of completing the proof of Theorems~\ref{maintheorem2} and~\ref{maintheorem3}.

\begin{proof}[Proof of Theorems~\ref{maintheorem2} and~\ref{maintheorem3}]
We will first show that we are in the position of employing~\cite[Theorem 4.3]{MAW}. 
To this end, we observe that,
for any~$h\in\mathcal H$, the function~$[0, 1]\ni \eta\mapsto \|h(\eta)\|_{\widetilde H^s}$ is continuous.
Moreover, it satisfies~$ \|h(0)\|_{\widetilde H^s} = 0$ and~$ \|h(1)\|_{\widetilde H^s} = \|q_0\|_{\widetilde H^s} >\rho$, with~$\rho$ as in Proposition~\ref{geom2}. {F}rom this, we have that
there exists~$\overline{\eta}_h\in (0, 1)$ such that~$\|h(\overline{\eta}_h)\|_{\widetilde H^s} =\rho$.
Thus, by~\eqref{INFMAX} and~\eqref{SPIJLndeiwohgeriujkghvnei89} we get that
\begin{equation}\label{C>BETA00}
c\ge \inf_{{q\in \widetilde H^s}\atop{\|q\|_{ \widetilde H^s} =\rho}} I(q).
\end{equation}

We consider~$\beta_1$ and~$\beta_2$ as given by Propositions~\ref{geom3} and~\ref{geom1} and set
\begin{equation*}
\beta:=\begin{cases}\displaystyle \beta_1 &{\mbox{ if }} s\in\left(0,\frac12\right], \\
\displaystyle  \beta_2 &{\mbox{ if }} s\in\left(\frac12,1\right).\end{cases}
\end{equation*}
With this notation, we deduce from~\eqref{C>BETA00} and Propositions~\ref{geom3} and~\ref{geom1} that
\begin{equation}\label{C>BETA} c\ge \beta>0.\end{equation}
Accordingly, \cite[Theorem~4.3]{MAW} provides the existence of a sequence~$q^j\subset \widetilde H^s$ satisfying~\eqref{C01} (with~$c$ as in~\eqref{INFMAX}) and~\eqref{C02}.

Moreover, since~\eqref{C01} and~\eqref{C02} hold true,  Proposition~\ref{PROPBOUNDED} entails that~$q^j$ is a bounded sequence in~$ \widetilde H^s$. In addition, thanks to Proposition~\ref{PROPOINT1}, $q^j$ converges to a critical point~$q$ of~$I$ as~$j\to+\infty$.

We claim that
\begin{equation}\label{NOTZERO}
q\not\equiv 0.
\end{equation}
To this aim, we notice that since~$q^j$ is a bounded sequence in~$ \widetilde H^s$, by Proposition~\ref{proposition_conv} it converges to some~$\overline q$ in~$L^2_{\rm{loc}}(\R,\R^n)$ as~$j\to+\infty$.
We point out that~$q=\overline q$ (the proof of this fact follows the lines of the proof of formula~\eqref{f847t5467tyreoghrwoli987654}
contained in Appendix~\ref{f847t5467tyreoghrwoli987654SEC}),
and therefore~$q^j$ converges to~$q$ in~$L^2_{\rm{loc}}(\R,\R^n)$ as~$j\to+\infty$.

In light of this, to prove~\eqref{NOTZERO}, it suffices to show that
\begin{equation}\label{KAPPA}
\mbox{there exists~$K>0$ such that } q^j\nrightarrow 0 \mbox{ in } L^2 ([-K, K], \R^n) \mbox{ as } j\to+\infty. 
\end{equation}
Assume by contradiction that for all~$K>0$
\begin{equation*}
{\mbox{$q^j\to 0\;$ in~$L^2 ([-K, K], \R^n)\;$ as~$j\to+\infty$.}}\end{equation*}
This gives that we are in the position 
of using Lemma~\ref{lemma_Rab}, thus obtaining that
\begin{equation}\label{lkjhgfdsmnbvcx09876543}
\limsup_{j\to +\infty} \|q^j\|^2_{L^2(\R, \R^n)} \le \frac{\gamma}{\beta(K)},
\end{equation}
with~$\gamma$ and~$\beta(K)$ as in~\eqref{BKDEFN}.

Now, by~\eqref{FUN2}, \eqref{DIFF2}, \eqref{C01} and~\eqref{C02}, we get that
\begin{equation}\label{LIMTOC}
\begin{split}
c &= \lim_{j\to +\infty} \left(I(q^j)-\frac12 \langle I'(q^j), q^j\rangle\right)\\
&= \lim_{j\to +\infty} \int_\R \left(\frac12\nabla_q W(x, q^j (x))\cdot q^j (x)- W(x, q^j (x))\right) \,dx.
\end{split}
\end{equation}

Let us consider first the case~$s\in (1/2, 1)$, namely we establish Theorem~\ref{maintheorem2}.
We recall the~$L^\infty$-bound given by 
Lemma~\ref{lemmaembeddcontinua}
and we exploit the
assumption on~$W$ in~\eqref{WSEGNATO} with
\[
M:= \sup_{j\in\N} \|q^j\|_{L^\infty(\R, \R^n)}.
\]
Accordingly, there exists~$\kappa_M>0$ such that
\begin{equation*} 
\int_\R \left|\frac12 \nabla_q W(x, q^j (x))\cdot q^j (x) -W(x, q^j (x))\right| \,dx \le \kappa_M \int_\R |q^j(x)|^2 \,dx .
\end{equation*}
Owing to this, \eqref{lkjhgfdsmnbvcx09876543} and~\eqref{LIMTOC}, we obtain that
\begin{equation*}
c\le \kappa_M \limsup_{j\to +\infty} \|q^j\|^2_{L^2(\R, \R^n)}
\le \frac{\kappa_M \, \gamma}{\beta(K)}.
\end{equation*}

Now, by the assumption in~\eqref{AUTOVAL} we have that
\begin{equation}\label{421215236y5yfdsgdsger65476ujtr}
\lim_{K\to+\infty}\beta(K)= +\infty\end{equation}
and so we conclude that~$c\le 0$, which contradicts~\eqref{C>BETA}.
This proves~\eqref{KAPPA}, which in turn establishes~\eqref{NOTZERO} when~$s\in (1/2, 1)$.

Let us now consider the case~$s\in (0, 1/2]$. By Lemma~\ref{lemmaW1} (used here with~$\varepsilon:=1$), we see that
\begin{equation}\label{P22}
\begin{split}
& \int_\R \left|\frac12\nabla_q W(x, q^j (x))\cdot q^j (x)- W(x, q^j (x))\right| \,dx\\
&\le\int_\R \left(\frac12 |q^j (x)|^2 +\frac{\sigma(1)}{2} |q^j (x)|^p +\frac12 |q^j (x)|^2 +\frac{\sigma(1)}{p} |q^j (x)|^p \right) \,dx\\
&\le \|q^j\|^2_{L^2(\R, \R^n)} + \sigma(1) \|q^j\|^p_{L^p(\R, \R^n)}.
\end{split}
\end{equation}
We recall the definition in~\eqref{2STARESSE} and, for any~$t\in (p, +\infty)$, we set
\[
\theta:=\begin{cases}
\displaystyle\dfrac{p-2}{2^*_s -2} &\mbox{ if } s\in \left(0, \frac12\right),\\
\displaystyle\dfrac{p-2}{t -2} &\mbox{ if } s=\frac12.
\end{cases}
\]
Since~$p\in (2, 2^*_s)$, we infer that~$\theta\in (0, 1)$. Hence, by~\eqref{INTERP1} we have that, for any~$q\in\widetilde H^s$,
\[
\|q\|^p_{L^p(\R, \R^n)}\le 
\begin{cases}\displaystyle
\|q\|^{2(1-\theta)}_{L^2(\R, \R^n)} \|q\|^{2^*_s \theta}_{L^{2^*_s}(\R, \R^n)} &\mbox{ if } s\in \left(0, \frac12\right),\\
\\
\displaystyle
\|q\|^{2(1-\theta)}_{L^2(\R, \R^n)} \|q\|^{t \theta}_{L^t (\R, \R^n)} &\mbox{ if } s=\frac12.
\end{cases}
\]
Accordingly, by using~\eqref{FIN1} and~\eqref{FIN3},
and recalling that~$q^j$ is bounded in~$ \widetilde H^s$,
we have that there exists~$C>0$ depending only on~$n$, $s$, $\alpha$ and~$\theta$ such that, for any~$j\in\N$,
\[
\|q^j\|^p_{L^p(\R, \R^n)}\le C \|q^j\|^{2(1-\theta)}_{L^2(\R, \R^n)}.
\]

Plugging this information into~\eqref{P22}, we obtain that
\begin{equation*}
\int_\R \left|\frac12\nabla_q W(x, q^j (x))\cdot q^j (x)- W(x, q^j (x))\right| \,dx
\le \|q^j\|^2_{L^2(\R, \R^n)} + C\sigma(1)\|q^j\|^{2(1-\theta)}_{L^2(\R, \R^n)}.
\end{equation*}
This, together with~\eqref{LIMTOC} and~\eqref{lkjhgfdsmnbvcx09876543}, entails that
\begin{equation*}
c \le \limsup_{j\to +\infty} \|q^j\|^2_{L^2(\R, \R^n)} +C \sigma(1) \limsup_{j\to +\infty} \|q^j\|^{2(1-\theta)}_{L^2(\R, \R^n)}
\le \frac{C}{\beta(K)},
\end{equation*}
being~$C>0$ a constant depending only on~$n$, $s$, $\alpha$, $\theta$ and~$ \gamma$,
with~$\gamma$ as in~\eqref{BKDEFN}. 
Therefore, recalling~\eqref{421215236y5yfdsgdsger65476ujtr}, we obtain that~$c\le 0$,
which contradicts inequality~\eqref{C>BETA}.
Thus, \eqref{KAPPA} and~\eqref{NOTZERO} are established when~$s\in (0,1/2]$ as well.

Moreover, in light of~\eqref{LIMSUP} and~\eqref{STIMA1}, taking~$c$ as in~\eqref{INFMAX} and recalling
the estimate in~\eqref{CC1}, we infer that there exists~$c_2>0$ depending on~$n$, $s$, $W$ and~$L$ such that
\begin{equation}\label{STIMAINONDIP2}
\|q\|_{\widetilde H^s}\le c_2.
\end{equation}

The existence statements in Theorems~\ref{maintheorem2} and~\ref{maintheorem3} are thereby established.
We now focus on the regularity statements.

Suppose first that~$s\in(1/2,1)$. To prove~\eqref{QREGC}, we use~\cite[Theorem 8.2]{guidagalattica} to find that
\begin{equation}\label{CREG}
\|q\|_{C^{s-1/2}(\R, \R^n)}\le c_3 \left(\|q\|^2_{L^2(\R, \R^n)}+ [q]^2_s\right)^{\frac12},
\end{equation}
for some~$c_3>0$ depending only on~$n$ and~$ s$.  Now,
by~\eqref{LMAT} and~\eqref{NORMH}, we get
$$
\|q\|^2_{L^2(\R, \R^n)}+ [q]^2_s \le \frac{1}{\alpha}\int_\R L(x) q(x)\cdot q(x) \,dx + [q]^2_s
\le\max\left\{1, \frac{1}{\alpha}\right\} \|q\|^2_{\widetilde H^s}.
$$
Hence, combining this inequality with~\eqref{CREG} and~\eqref{STIMAINONDIP2}, we infer that 
\[
\|q\|_{C^{s-1/2}(\R, \R^n)}\le c_3 \max\left\{1, \frac{1}{\sqrt{\alpha}}\right\} \|q\|_{\widetilde H^s} \le C^\star,
\]
being~$C^\star:= c_3 \, c_2\, \max\{1, 1/\sqrt{\alpha}\}$. 
This proves~\eqref{QREGC}.

It is worth noting that the solution~$q$ exhibits (locally) a higher regularity, as stated in~\eqref{REGLOC}. Indeed, 
\begin{equation}\label{LINFLOC}
W\in L_{\rm{loc}}^\infty(\R\times\R^n, \R)\quad\mbox{ and } \quad L\in L^\infty_{\rm{loc}}(\R, \R^{n\times n}). 
\end{equation}
Moreover, by~\eqref{CREG} we have in particular that~$q\in L^\infty(\R, \R^n)$. Thus, we are in the position of applying~\cite[Theorem~2.4.3]{ROSOTON} which entails that~$q\in C^{2s}(B_1, \R^n)$ and
\begin{equation*}
\|q\|_{C^{2s}(B_{1/2}, \R^n)} \le C \Big(\|q\|_{L^\infty(\R, \R^n)} +\|Lq\|_{L^\infty(B_1, \R^n)} +\|\nabla_q W\|_{L^\infty(B_1, \R^n)}\Big)
\le C^{\star\star},
\end{equation*}
for a suitable~$C^{\star\star} >0$ depending only on~$n$, $s$, $W$ and~$L$. Hence, \eqref{REGLOC} holds true.  

Furthermore, by~\eqref{CREG}, in particular, $q$ is uniformly continuous and then, by Lemma~\ref{lemmainfty}, we deduce~\eqref{LIM}.
This finishes the proof of Theorem~\ref{maintheorem2}.

We now focus on completing the proof of Theorem~\ref{maintheorem3}
As an intermediate step towards the proof of~\eqref{cbakbd} and~\eqref{cbnklaj}, we show that
\begin{equation} \label{QQINF}
\|q\|_{L^\infty(\R, \R^n)}\le C,
\end{equation}
for some constant~$C>0$ depending on~$n$, $s$, $W$ and~$\alpha$.

For this, 
recalling~\eqref{FIN1}, \eqref{FIN3} and~\eqref{STIMAINONDIP2}, we see that
\begin{equation*}
\begin{split}
&
\mbox{either } \;
\|q\|_{L^{2^*_s}(\R, \R^n)}\le \overline c \|q\|_{\widetilde H^s} \le c_3  \ \ \mbox{ if } s\in (0, 1/2),\\
&\mbox{or } \;\|q\|_{L^{t}(\R, \R^n)}\le \widehat c  \|q\|_{\widetilde H^s}\le c_4\quad \, \mbox{ if } s=1/2
\end{split}
\end{equation*}
being~$c_3:=\overline c\, c_2$ and~$c_4:=\widehat c\, c_2$.

As a result, Theorem~\ref{THDMPV} applies in our setting with~$F(x, q(x)):= \nabla_q W(x, q(x))$, $K=1$, $h_1(x) = b_1$ and~$\gamma_1=p-1$, yielding~\eqref{QQINF}, as desired.

Now, if~$s\in (0, 1/2)$, by~\eqref{QQINF}, \eqref{LINFLOC} and~\cite[Theorem~2.4.3]{ROSOTON} we infer that~$q\in C^{2s}(B_1, \R^n)$ and
\[
\|q\|_{C^{2s}(B_{1/2}, \R^n)} \le C \Big(\|q\|_{L^\infty(\R, \R^n)} +\|Lq\|_{L^\infty(B_1, \R^n)} +\|\nabla_q W\|_{L^\infty(B_1, \R^n)}\Big)
\le C^{\star},
\]
for a suitable~$C^\star>0$ depending only on~$n$, $s$, $W$ and~$L$, namely~\eqref{cbakbd} holds true.

On the other hand, if~$s=1/2$, we can apply~\cite[Theorem~2.4.3]{ROSOTON} with any large~$p\in (1, +\infty)$. In this case,
we have that~$q\in C^{1-\varepsilon}(B_1, \R^n)$ for any~$\varepsilon\in(0,1)$ and 
\[
\|q\|_{C^{1-\varepsilon}(B_{1/2}, \R^n)}\le C^{\star\star},
\]
for a suitable~$C^{\star\star} >0$ depending only on~$n$, $s$, $W$ and~$L$.
This proves~\eqref{cbnklaj}.

Finally, thanks to~\eqref{QQINF}, we are in the position of applying Corollary~\ref{COR1}, this obtaining~\eqref{LIM}. 
This concludes the proof of Theorem~\ref{maintheorem3}.
\end{proof}

\begin{appendix}

\section{Proof of formula~\eqref{AGGNORMPROOF}}\label{AGGNORMPROOFSEC}

Here we check that~$\|\cdot\|_{ \widetilde H^s}$ is a norm, thus establishing the claim in formula~\eqref{AGGNORMPROOF}.

We first check the triangle inequality. To this aim, we claim that, for any~$q$, $\overline q:\R\to\R^n$,
\begin{equation}\label{kjhgfds12345678909876543}\left|
\int_\R L(x) q(x)\cdot\overline q(x)\,dx\right|\le 
\left( \int_\R L(x) q(x)\cdot q(x) \,dx\right)^{1/2}
\left( \int_\R L(x) \overline q(x)\cdot \overline q(x) \,dx\right)^{1/2}.
\end{equation}
Indeed, we suppose that~$q\not\equiv0$, otherwise we are done, and therefore~$\int_\R L(x) q(x)\cdot q(x) \,dx\neq0$.
Thus, we can define
$$ \beta:=\left( \frac{\displaystyle\int_\R L(x) \overline q(x)\cdot \overline q(x) \,dx}{\displaystyle \int_\R L(x) q(x)\cdot q(x) \,dx}
\right)^{1/4}.$$
In light of~\eqref{LMAT} and the symmetry property of the matrix~$L$, we have that
$$ 0\le L\left(\beta q\pm\frac{\overline q}{\beta}\right)\cdot\left(\beta q\pm\frac{\overline q}{\beta}\right)
= \beta^2 Lq\cdot q+\frac1{\beta^2}L\overline q\cdot\overline q \pm 2Lq\cdot\overline q.
$$
Accordingly,
\begin{eqnarray*}
2\left| \int_\R L(x)q(x)\cdot\overline q(x)\,dx\right| \le \beta^2 \int_\R L(x)q(x)\cdot q(x)\,dx+\frac1{\beta^2}\int_\R L(x)\overline q(x)\cdot\overline q(x)\,dx.
\end{eqnarray*}
Substituting the expression of~$\beta$, we obtain~\eqref{kjhgfds12345678909876543}. 

Exploiting~\eqref{kjhgfds12345678909876543}, we see that
\begin{equation*}\begin{split}&
\int_\R L(x) \big(q(x)+\overline q(x)\big)\cdot \big(q(x)+\overline q(x)\Big) \,dx
\\&\qquad=
\int_\R L(x) q(x)\cdot q(x)\,dx+\int_\R L(x) \overline q(x)\cdot \overline q(x)\,dx+2\int_\R L(x) q(x)\cdot\overline q(x)\,dx
\\&\qquad\le
\int_\R L(x) q(x)\cdot q(x)\,dx+\int_\R L(x) \overline q(x)\cdot \overline q(x)\,dx \\&\qquad\qquad\qquad+2
\left( \int_\R L(x) q(x)\cdot q(x) \,dx\right)^{1/2}
\left( \int_\R L(x) \overline q(x)\cdot \overline q(x) \,dx\right)^{1/2} \\&\qquad=
\left(\left( \int_\R L(x) q(x)\cdot q(x) \,dx\right)^{1/2} +
\left( \int_\R L(x) \overline q(x)\cdot \overline q(x) \,dx\right)^{1/2}
\right)^2
.\end{split}
\end{equation*}

As a consequence of this, we have that, for any~$q$, $\overline q:\R\to\R^n$,
\begin{eqnarray*}
&&\|q+\overline q\|_{ \widetilde H^s}=\left(
[q+\overline q]^2_s +\int_\R L(x) \big(q(x)+\overline q(x)\big)\cdot \big(q(x)+\overline q(x)\Big) \,dx\right)^{1/2}\\
&&\le
\left(
\big([q]_s+[\overline q]_s\big)^2 +\left(\left( \int_\R L(x) q(x)\cdot q(x) \,dx\right)^{1/2} +
\left( \int_\R L(x) \overline q(x)\cdot \overline q(x) \,dx\right)^{1/2}
\right)^2\right)^{1/2}\\&&\le
\left([q]_s^2+ \int_\R L(x) q(x)\cdot q(x) \,dx\right)^{1/2} 
+
\left([\overline q]_s^2+ \int_\R L(x)\overline q(x)\cdot \overline q(x) \,dx\right)^{1/2} 
\\&&=\|q\|_{ \widetilde H^s}+\|\overline q\|_{ \widetilde H^s},
\end{eqnarray*}
which completes the proof of the triangle inequality.

The absolute homogeneity follows from the fact that, for any~$t\in\R$ and~$q:\R\to\R^n$,
\begin{eqnarray*} &&
\|tq\|_{ \widetilde H^s}= \left([tq]^2_s +\int_\R L(x) (tq)(x)\cdot (tq)(x) \,dx\right)^{1/2}
=\left(t^2 [q]^2_s +t^2\int_\R L(x) q(x)\cdot q(x) \,dx\right)^{1/2}\\&&\qquad\qquad=|t|\left([q]^2_s +\int_\R L(x) q(x)\cdot q(x) \,dx\right)^{1/2}
=|t|\|q\|_{ \widetilde H^s}.
\end{eqnarray*}

It remains to check the positivity property.
For this, we observe that, in light of~\eqref{LMAT},
\begin{eqnarray*}
\|q\|_{ \widetilde H^s}= \left([q]^2_s +\int_\R L(x) q(x)\cdot q(x) \,dx\right)^{1/2}\ge \left([q]^2_s +\alpha\int_\R |q(x)|^2 \,dx\right)^{1/2}\ge0.
\end{eqnarray*}
This also implies that if~$\|q\|_{ \widetilde H^s}=0$, then~$[q]_s=\|q\|_{L^2(\R,\R^n)}=0$, and therefore~$q\equiv0$.

\section{Hilbert space properties for~$\widetilde{H}^s$}\label{AGGNORMPROOFSEC22}

Here we check that the space~$\widetilde{H}^s$, introduced
at the beginning of
Section~\ref{SIPDJOLNDFUOJFOJLN}, is a Hilbert space.

\begin{lemma}\label{LemmaHilbert}
$( \widetilde H^s, \|\cdot\|_{ \widetilde H^s})$ is a Hilbert space.
\end{lemma}
\begin{proof}
For any~$(q, u)\in  \widetilde H^s\times  \widetilde H^s$, we consider the map
\[
(q, u)\mapsto\langle q, u\rangle_{ \widetilde H^s} :=\iint_{\R^2} \frac{(q(x)-q(y))(u(x)-u(y))}{|x-y|^{1+2s}} \,dx\, dy +\int_{\R} L(x)q(x)\cdot u(x) \,dx.
\]
It is easy to check that~$\langle \cdot, \cdot\rangle_{ \widetilde H^s}$ is a scalar product and it induces the norm defined in~\eqref{NORMH}. We show that~$ \widetilde H^s$ is complete with respect to this norm.

To this end, let~$q^j$ denote a Cauchy sequence in~$ \widetilde H^s$ and take~$\epsilon>0$. Thus, there exists~$\overline{n}_{\epsilon}\in\N$ such that if~$j$, $m\ge\overline{n}_{\epsilon}$, then
\begin{equation}\label{CAUCHY1}
\epsilon\ge \|q^j -q^m\|^2_{ \widetilde H^s}\ge \min\{\alpha, 1\} \|q^j -q^m\|^2_{H^s(\R, \R^n)},
\end{equation}
where the last inequality is a consequence of~\eqref{EMB1}. 

Now, since~$H^s(\R, \R^n)$ is complete, there exists~$q\in H^s(\R, \R^n)$ such that~$q^j\to q$ in~$H^s(\R, \R^n)$ as~$j\to +\infty$. Moreover, we can extract a subsequence~$q^{j_k}$ such that~$q^{j_k}\to q$ a.e. in~$\R$ as~$k\to+\infty$.  

We claim that~$q\in \widetilde H^s$. Indeed, by the hypothesis in~\eqref{LMAT} and Fatou's Lemma,
\[
\begin{split}
&\iint_{\R^2} \frac{|q(x)-q(y)|^2}{|x-y|^{1+2s}} \,dx\, dy +\int_\R L(x) q(x)\cdot q(x) \,dx\\
&\qquad\le \liminf_{k\to +\infty} \left(\iint_{\R^2} \frac{|q^{j_k}(x)-q^{j_k}(y)|^2}{|x-y|^{1+2s}}\,dx\, dy+\int_{\R} L(x) q^{j_k}(x)\cdot q^{j_k}(x) \,dx \right)\\
&\qquad=\liminf_{k\to +\infty} \|q^{j_k}\|^2_{ \widetilde H^s}\\
&\qquad\le \liminf_{k\to +\infty} \Big(\|q^{j_k}-q^{\overline{n}_1}\|_{ \widetilde H^s} +\|q^{\overline{n}_1}\|_{ \widetilde H^s}\Big)^2\\
&\qquad\le \Big(1+ \|q^{\overline{n}_1}\|_{ \widetilde H^s}\Big)^2\\
&\qquad <+\infty,
\end{split}
\]
where we also used the first inequality in~\eqref{CAUCHY1} with~$\epsilon=1$. This proves that~$q\in \widetilde H^s$.

To conclude the proof of Lemma~\ref{LemmaHilbert} it remains to show that~$q^m\to q$ in~$ \widetilde H^s$
as~$m\to+\infty$. For this, let~$m\ge\overline{n}_{\epsilon}$. Thus, by~\eqref{CAUCHY1}, the assumption in~\eqref{LMAT} and Fatou's Lemma, we get that
\[
\begin{split}
\epsilon&\ge \liminf_{k\to +\infty} \left(\iint_{\R^2} \frac{|(q^m-q^{j_k})(x)-(q^m-q^{j_k})(y)|^2}{|x-y|^{1+2s}}\,dx\, dy+\int_{\R} L(x) (q^m -q^{j_k})(x)\cdot (q^m-q^{j_k})(x) \,dx \right)\\
&\ge \iint_{\R^2} \liminf_{k\to +\infty} \frac{|(q^m-q^{j_k})(x)-(q^m-q^{j_k})(y)|^2}{|x-y|^{1+2s}}\,dx\, dy\\
&\qquad\qquad + \int_{\R} \liminf_{k\to +\infty} \left(L(x) (q^m -q^{j_k})(x)\cdot (q^m-q^{j_k})(x)\right) \,dx \\
&= \iint_{\R^2} \frac{|(q^m-q)(x)-(q^m-q)(y)|^2}{|x-y|^{1+2s}}\,dx\, dy+\int_{\R} L(x) (q^m -q)(x)\cdot (q^m-q)(x) \,dx \\
&= \|q^m-q\|^2_{ \widetilde H^s},
\end{split}
\]
namely~$q^m\to q$ in~$ \widetilde H^s$ as~$m\to +\infty$, as desired.
\end{proof}

\section{Proof of Proposition~\ref{DENNUO:AS0}}\label{DENNUO:AS0ap}

To prove Proposition~\ref{DENNUO:AS0}, we observe that
the completion of~$C^\infty_0(\R,\R^n)$
with respect to the norm of~$\widetilde H^s$ is obviously contained in~$\widetilde H^s$
and therefore we only need to establish the opposite inclusion.
That is, given~$q\in \widetilde H^s$, our objective is to construct a sequence of functions~$q_\e\in C^\infty_0(\R,\R^n)$ such that
\begin{equation}\label{DENNUO:AS0.b}
\lim_{\e\searrow0}\|q-q_\e\|_{ \widetilde H^s}=0
\end{equation}
For this scope, we consider first a sequence of functions~$\tau_R\in C^\infty_0([-R-1,R+1],[0,1])$
such that~$\tau_R=1$ in~$[-R,R]$ and~$|\tau_R'|\le 2$ and we define, for each~$j\in\{1,\dots,n\}$,
\begin{equation}\label{DENNUO:AS0.bacr} q_{R,j}(x):= \tau_R(x)\, q_j(x).\end{equation}
We observe that
\begin{equation}\label{DENNUO:AS0.cd}
{\mbox{$q_{R,j}$ is supported in~$[-R-1,R+1]$.}}
\end{equation}
Besides, we know (see e.g.~\cite[Lemma~1.63]{MR4567945}) that
\begin{equation}\label{DENNUO:AS0.c}
\lim_{R\to+\infty} [q_j-q_{R,j}]_s=0.
\end{equation}

Moreover,
\begin{eqnarray*}
&&\lim_{R\to+\infty} \int_\R L(x)\big(q(x)-q_R(x)\big)\cdot\big(q(x)-q_R(x)\big)\,dx\\&&\qquad
=\lim_{R\to+\infty} \int_{\R\setminus[-R,R]} \big(1-\tau_R(x)\big)^2
L(x)q(x)\cdot q(x)\,dx\\&&\qquad\le\lim_{R\to+\infty} \int_{\R\setminus[-R,R]} 
L(x)q(x)\cdot q(x)\,dx\\&&\qquad=0.
\end{eqnarray*}
{F}rom this and~\eqref{DENNUO:AS0.c} it follows that
\begin{equation*}
\lim_{R\to+\infty} \|q-q_{R}\|_{ \widetilde H^s}=0.
\end{equation*}
In particular, given~$\e>0$, we take~$R_\e>0$ sufficiently large such that
\begin{equation}\label{DENNUO:AS0.d}
\|q-q_{R_\e}\|_{ \widetilde H^s}\le\e.
\end{equation}

Our objective is now to take a smooth approximation of~$q_{R_\e}$. For this, we consider
a mollification sequence. Namely, given~$\eta\in C^\infty_0([-1,1],[0,+\infty))$ with unit mass, we define, for~$\delta\in(0,1)$,
$$ \eta_\delta(x):=\frac1\delta \eta\left(\frac{x}\delta\right)\qquad{\mbox{and}}\qquad
q_{\delta,\e,j}:=q_{R_\e,j}*\eta_\delta.$$
We stress that
\begin{equation}\label{DENNUO:AS0.cdu}
{\mbox{$q_{\delta,\e,j}$ is of class~$C^\infty$ and supported in~$[-R_\e-2,R_\e+2]$.}}\end{equation}
due to~\eqref{DENNUO:AS0.cd}.

We know (see e.g. equation~(1.38) and the subsequent comment in~\cite{MR4567945}) that
\begin{equation}\label{DENNUO:AS0.cef}
\lim_{\delta\searrow0} [q_{R_\e,j}-q_{\delta,\e,j}]_s=0.
\end{equation}

In view of~\eqref{DENNUO:AS0.cd} and~\eqref{DENNUO:AS0.cdu}, we also remark that
\begin{equation}\label{DENNUO:AS0.bacr2}
\begin{split}&
\int_\R L(x)\big(q_{R_\e}(x)-q_{\delta,\e}(x)\big)\cdot\big(q_{R_\e}(x)-q_{\delta,\e}(x)\big)\,dx\\&\qquad=\int_{[-R_\e-2,R_\e+2]} L(x)\big(q_{R_\e}(x)-q_{\delta,\e}(x)\big)\cdot\big(q_{R_\e}(x)-q_{\delta,\e}(x)\big)\,dx\\
&\qquad\le
\sup_{[-R_\e-2,R_\e+2]}|L|\,\big\|q_{R_\e}(x)-q_{\delta,\e}(x)\big\|_{L^2(\R,\R^n)}^2.
\end{split}\end{equation}

Also, owing to~\eqref{DENNUO:AS0.bacr}, we have that~$\|q_{R_\e,j}(x)\|_{L^2(\R)}\le\|q_{j}(x)\|_{L^2(\R)}<+\infty$. Therefore (see e.g.~\cite[Theorem~9.6]{MR3381284}) we conclude that
$$ \lim_{\delta\searrow0}\big\|q_{R_\e,j}(x)-q_{\delta,\e,j}(x)\big\|_{L^2(\R)}=0.$$

Combining this information and~\eqref{DENNUO:AS0.bacr2}, we gather that
$$ \lim_{\delta\searrow0}
\int_\R L(x)\big(q_{R_\e}(x)-q_{\delta,\e}(x)\big)\cdot\big(q_{R_\e}(x)-q_{\delta,\e}(x)\big)\,dx=0,$$
which, together with~\eqref{DENNUO:AS0.cef}, gives that
\begin{equation*}
\lim_{\delta\searrow0} \|q_{R_\e}-q_{\delta,\e}\|_{\widetilde H^s}=0.
\end{equation*}

As a result, we can take~$\delta_\epsilon$ sufficiently small such that~$\|q_{R_\e}-q_{\delta_\epsilon,\e}\|_{\widetilde H^s}\le\epsilon$. Thus, using the short notation~$q_\epsilon:=q_{\delta_\epsilon,\e}$ (which, by~\eqref{DENNUO:AS0.cdu}, is
a smooth and compactly supported function) 
and recalling~\eqref{DENNUO:AS0.d}, we conclude that
$$ \|q-q_{\e}\|_{ \widetilde H^s}=\|q-q_{\delta_\epsilon,\e}\|_{ \widetilde H^s}\le
\|q-q_{R_\e}\|_{ \widetilde H^s}
+
\|q_{R,\e}-q_{\delta_\epsilon,\e}\|_{ \widetilde H^s}
\le2\e.$$
This establishes~\eqref{DENNUO:AS0.b} and completes the proof of Proposition~\ref{DENNUO:AS0}.

\section{Proof of formula~\eqref{f847t5467tyreoghrwoli987654}}\label{f847t5467tyreoghrwoli987654SEC}

In this appendix we provide the proof of formula~\eqref{f847t5467tyreoghrwoli987654}.

To this end, we observe that
$$ \sup_{j\in\N}[q^j]_s<+\infty\qquad
{\mbox{and}}\qquad  \sup_{j\in\N}\int_{\R}L(x)q^j(x)\cdot q^j(x)\,dx<+\infty.$$
Consequently (see~\eqref{LMAT} and e.g.~\cite[Proposition~3.6]{guidagalattica}), we have that
$$ \sup_{j\in\N}\|(-\Delta)^{\frac{s}2}q^j\|_{L^2(\R,\R^n)}<+\infty\qquad
{\mbox{and}}\qquad  \sup_{j\in\N}\|q^j\|_{L^2(\R,\R^n)}<+\infty.$$
Therefore, we can suppose that
there exist~$g$, $q^\star\in L^2(\R,\R^n)$ such that,
up to subsequences, $(-\Delta)^{\frac{s}2}q^j$ weakly converges
to~$g$ and~$q^j$ weakly converges to~$q^\star$
in~$L^2(\R,\R^n)$, as~$j\to+\infty$.

In particular, recalling also the convergence in~\eqref{conv2},
for every bounded function~$\psi:\R\to\R^n$ with compact support, we have that
$$ \int_\R \overline q(x)\cdot \psi(x)\,dx=\lim_{j\to+\infty}
\int_\R q^j(x)\cdot \psi(x)\,dx=\int_\R q^\star(x)\cdot \psi(x)\,dx.$$
This shows that
\begin{equation}\label{OSJLDMCDPMS-23e}
q^\star=\overline q.\end{equation}
Furthermore, for every~$\phi\in C^\infty_0(\R,\R^n)$,
\begin{equation}\label{cww214157fv346rt38rf24o}\begin{split}
&\int_\R g(x)\cdot \phi(x)\,dx=
\lim_{j\to+\infty}\int_\R(-\Delta)^{\frac{s}2}q^j(x)\cdot \phi(x)\,dx
\\&\qquad
=\lim_{j\to+\infty}\int_\R q^j(x)\cdot (-\Delta)^{\frac{s}2}\phi(x)\,dx
=\int_\R q^\star(x)\cdot (-\Delta)^{\frac{s}2}\phi(x)\,dx.
\end{split}\end{equation}
Since~$g$ and~$q^\star$ belong to~$L^2(\R,\R^n)$, by the density
of~$C^\infty_0(\R,\R^n)$ in~$H^s(\R,\R^n)$,
this holds true for all~$\phi\in H^s(\R,\R^n)$.

Now, exploiting the weak convergence in~\eqref{WEAK1}
(and again~\cite[Proposition~3.6]{guidagalattica}),
we see that, for every~$\varphi\in C^\infty_0(\R,\R^n)$,
\begin{equation*}
\begin{split}
&\int_{\R^2} q(x)\cdot (-\Delta)^{s}\varphi(x)
\,dx+\int_\R L(x) q(x)\cdot\varphi(x) \,dx\\
&\qquad=
\int_{\R^2} (-\Delta)^{\frac{s}2}q(x)\cdot (-\Delta)^{\frac{s}2}\varphi(x)
\,dx+\int_\R L(x) q(x)\cdot\varphi(x) \,dx\\
&\qquad=c_s \iint_{\R^2} \frac{(q(x)-q(y))\cdot(\varphi(x)-\varphi(y))}{|x-y|^{1+2s}} \,dx\, dy +\int_\R L(x) q(x)\cdot\varphi(x) \,dx\\
&\qquad=\lim_{j\to+\infty}
c_s \iint_{\R^2} \frac{(q^j(x)-q^j(y))\cdot(\varphi(x)-\varphi(y))}{|x-y|^{1+2s}} \,dx\, dy +\int_\R L(x) q^j(x)\cdot\varphi(x) \,dx\\
&\qquad =\lim_{j\to+\infty}
\int_{\R^2} (-\Delta)^{\frac{s}2}q^j(x)\cdot (-\Delta)^{\frac{s}2}\varphi(x)
\,dx+\int_\R L(x) q^j(x)\cdot\varphi(x) \,dx\\
&\qquad=
\int_{\R^2} g(x)\cdot (-\Delta)^{\frac{s}2}\varphi(x)
\,dx+\int_\R L(x) q^\star(x)\cdot\varphi(x) \,dx
.\end{split}
\end{equation*}
Hence, using~\eqref{cww214157fv346rt38rf24o} with~$\phi:=(-\Delta)^{\frac{s}2}\varphi$, we find that, for every~$\varphi\in C^\infty_0(\R,\R^n)$,
\begin{eqnarray*}&& \int_{\R^2} q(x)\cdot (-\Delta)^{s}\varphi(x)
\,dx+\int_\R L(x) q(x)\cdot\varphi(x) \,dx\\&&\qquad
=\int_{\R^2} q^\star(x)\cdot (-\Delta)^{s}\varphi(x)
\,dx+\int_\R L(x) q^\star(x)\cdot\varphi(x) \,dx.
\end{eqnarray*}

Setting~$Q:=q-\overline q$,
the above observation and~\eqref{OSJLDMCDPMS-23e} yield that, for any~$\varphi\in C^\infty_0(\R,\R^n)$,
\begin{equation}\label{f8o4erytilro4poiuytr0987654}
\int_\R Q(x)\cdot \Big((-\Delta)^s \varphi(x)+ L(x) \varphi(x)\Big)\,dx=0.\end{equation}

Now we take a standard mollification sequence~$\eta_\varepsilon\in C^\infty_0(\R,\R^n)$. Since~$Q\in L^2(\R,\R^n)$, we know (see e.g.~\cite[Theorem~9.6]{MR3381284}) that
\begin{equation}\label{LAONKDIPOHOIDGUPIGBIO1KEDF}
{\mbox{$Q*\eta_\varepsilon$ converges to~$Q$ in~$L^2(\R,\R^n)$ as~$\varepsilon\searrow0$.}}
\end{equation}
We take~$p\in\R$ to be a point in this set of convergence and set
\begin{equation}\label{LAONKDIPOHOIDGUPIGBIO1KEDFlowdfenv}
\eta_{\varepsilon,p}(x):=\eta_\varepsilon(p-x).\end{equation}

Also, thanks to Lax-Milgram Theorem (see e.g.~\cite[Corollary~5.8]{MR2759829}),
there exists~$\phi_{\varepsilon,p}\in H^s_\sharp$ such that, in the dual sense of~$\widetilde H^s$,
$$ (-\Delta)^s \phi_{\varepsilon,p}+ L \phi_{\varepsilon,p}= \eta_{\varepsilon,p}.$$
That is,
\begin{equation}\label{OJLmsdPo0iljwfeffghsdfFuqjwdmf}\begin{split}&
c_s \iint_{\R^2} \frac{(Q(x)-Q(y))\cdot(\phi_{\varepsilon,p}(x)-\phi_{\varepsilon,p}(y))}{|x-y|^{1+2s}} \,dx\, dy +\int_\R L(x) Q(x)\cdot\phi_{\varepsilon,p}(x) \,dx
\\&\qquad\qquad=\int_\R Q(x)\,\eta_{\varepsilon,p}(x)\,dx.\end{split}
\end{equation}

Now we employ Proposition~\ref{DENNUO:AS0} and we find
\begin{equation}\label{f8o4erytilro4poiuytr0987654.b}
\phi_{\delta,\varepsilon,p}\in C^\infty_0(\R,\R^n)\end{equation} such that
$$ \lim_{\delta\searrow0}\|\phi_{\varepsilon,p}-\phi_{\delta,\varepsilon,p}\|_{\widetilde H^s}=0.$$
In particular,
\begin{eqnarray*}&&
\lim_{\delta\searrow0}c_s \iint_{\R^2} \frac{(Q(x)-Q(y))\cdot(\phi_{\delta,\varepsilon,p}(x)-\phi_{\delta,\varepsilon,p}(y))}{|x-y|^{1+2s}} \,dx\, dy +\int_\R L(x) Q(x)\cdot\phi_{\delta,\varepsilon,p}(x) \,dx\\&&\qquad=
c_s \iint_{\R^2} \frac{(Q(x)-Q(y))\cdot(\phi_{\varepsilon,p}(x)-\phi_{\varepsilon,p}(y))}{|x-y|^{1+2s}} \,dx\, dy +\int_\R L(x) Q(x)\cdot\phi_{\varepsilon,p}(x) \,dx.
\end{eqnarray*}

This and~\eqref{OJLmsdPo0iljwfeffghsdfFuqjwdmf} return that
\begin{equation}\label{f8o4erytilro4poiuytr09876540ojdf}\begin{split}&
\lim_{\delta\searrow0}c_s \iint_{\R^2} \frac{(Q(x)-Q(y))\cdot(\phi_{\delta,\varepsilon,p}(x)-\phi_{\delta,\varepsilon,p}(y))}{|x-y|^{1+2s}} \,dx\, dy +\int_\R L(x) Q(x)\cdot\phi_{\delta,\varepsilon,p}(x) \,dx\\&\qquad\qquad=
\int_\R Q(x)\,\eta_{\varepsilon,p}(x)\,dx.\end{split}
\end{equation}

Furthermore, by virtue of~\eqref{f8o4erytilro4poiuytr0987654.b}, we can use~\eqref{f8o4erytilro4poiuytr0987654} with~$\varphi:=\phi_{\delta,\varepsilon,p}$ and we see that
\begin{eqnarray*}
&& c_s \iint_{\R^2} \frac{(Q(x)-Q(y))\cdot(\phi_{\delta,\varepsilon,p}(x)-\phi_{\delta,\varepsilon,p}(y))}{|x-y|^{1+2s}} \,dx\, dy +\int_\R L(x) Q(x)\cdot\phi_{\delta,\varepsilon,p}(x) \,dx\\&&\qquad
=
\int_\R Q(x)\cdot \Big((-\Delta)^s \phi_{\delta,\varepsilon,p}(x)+ L(x) \phi_{\delta,\varepsilon,p}(x)\Big)\,dx\\&&\qquad=0.\end{eqnarray*}
{F}rom this and~\eqref{f8o4erytilro4poiuytr09876540ojdf} it follows that
$$ \int_\R Q(x)\,\eta_{\varepsilon,p}(x)\,dx=0.$$

Hence, in light of~\eqref{LAONKDIPOHOIDGUPIGBIO1KEDF} and~\eqref{LAONKDIPOHOIDGUPIGBIO1KEDFlowdfenv},
\begin{eqnarray*}
0=\lim_{\varepsilon\searrow0}\int_\R Q(x)\,\eta_{\varepsilon,p}(x)\,dx
=\lim_{\varepsilon\searrow0}\int_\R Q(x)\,\eta_\varepsilon(p-x)\,dx
=\lim_{\varepsilon\searrow0} Q*\eta_\varepsilon(p)=Q(p).
\end{eqnarray*}
Therefore~$Q$ vanishes almost everywhere, which entails the desired result.

\section{Proof of~\eqref{wjknellospazio}}\label{appendixboh}

The aim of this appendix is to check that the claim in~\eqref{wjknellospazio}
holds true. For this, we establish the following statement, of which~\eqref{wjknellospazio} will be an obvious consequence.

\begin{proposition}\label{proppartepositiva}
Let~$c>0$ and~$q\in\widetilde{H}^s$. Let~$u:\R\to\R^n$ be defined by
\[
u(x):=\big((q_1 (x)-c)^+,\dots,  (q_n (x)-c)^+\big).
\]
Then, $u\in\widetilde{H}^s$.
\end{proposition}

\begin{proof}
We will prove that
\begin{equation}\label{normafinita}
\|u\|_{\widetilde{H}^s} <+\infty.
\end{equation}
Recalling~\eqref{NORMH}, we see that
\[
\begin{split}
\|u\|_{\widetilde{H}^s} &= [u]^2_s +\int_{\R} L(x) u(x)\cdot u(x) \,dx\\
&=c_s \iint_{\R^2} \sum_{j=1}^n \frac{|u_j(x)-u_j(y)|^2}{|x-y|^{1+2s}} \,dx\, dy + \int_{\R} \sum_{i,j=1}^n  L_{ji}(x) u_i(x) u_j(x)\,dx,
\end{split}
\]
where~$u_j=(q_j -c)^+$.

We now show that
\begin{eqnarray} \label{csnOIBo0o0U}
&& \iint_{\R^2}\sum_{j=1}^n \frac{|u_j(x)-u_j(y)|^2}{|x-y|^{1+2s}} \,dx\, dy <+\infty\\ \label{qcbijo0o0n}
\mbox{and } \quad && \int_{\R}\sum_{i,j=1}^n  L_{ji}(x) u_i(x) u_j(x) \,dx <+\infty.
\end{eqnarray}
Regarding~\eqref{csnOIBo0o0U}, we have that
\[ |u_j(x)-u_j(y)|=\big| (q_j(x) -c)^+ - (q_j(y) -c)^+\big|\le\big| (q_j(x) -c) - (q_j(y) -c)\big|=|q_j(x)-q_j(y)|,
\] from which~\eqref{csnOIBo0o0U} plainly follows.

Moreover, if~$u_j(x)>0$, then~$q_j(x)>c$ and therefore~$2q_j(x)-c>c>0$.
Consequently,
since~$L(x)$ has nonnegative entries, we see that
\begin{equation}\label{bcio0o0ljk}
\begin{split}
\sum_{i,j=1}^n \int_{\R} L_{ji}(x) u_j(x) u_i(x) \,dx &= \sum_{i,j=1}^n \,\int_{  \{u_j>0\}\cap\{u_i>0\}  } L_{ji}(x) (q_j(x)-c) (q_i(x)-c) \,dx\\
&= \sum_{i,j=1}^n\, \int_{  \{u_j>0\}\cap\{u_i>0\}  } L_{ji}(x) q_j(x) q_i(x) \,dx\\
&\qquad- c \sum_{i,j=1}^n \,\int_{  \{u_j>0\}\cap\{u_i>0\}  } L_{ji}(x) (2q_j(x)-c) \,dx\\
&\le \sum_{i,j=1}^n \int_{\R} L_{ji}(x) q_j(x) q_i(x) \,dx\\
&<+\infty,
\end{split}
\end{equation}
namely~\eqref{qcbijo0o0n} holds true.

{F}rom~\eqref{csnOIBo0o0U} and~\eqref{qcbijo0o0n} we obtain~\eqref{normafinita}, as desired.
\end{proof}

\section{An interpolation inequality}\label{interpappe}

We provide the following interpolation result.

\begin{proposition}
Let~$\theta\in (0, 1)$ and~$p$, $q$, $r\in [1, +\infty)$. If~$u\in L^p(\R, \R^n)$ and~$v\in L^q(\R, \R^n)$, then
\[
\left\Vert |u|^{\frac{p(1-\theta)}{r}} \, |v|^{\frac{\theta q}{r}} \right\Vert^r_{L^r(\R, \R^n)}\le \|u\|^{p(1-\theta)}_{L^p(\R, \R^n)} \, \|v\|^{q\theta}_{L^q(\R, \R^n)}.
\]

Furthermore, if~$1\le p < q<+\infty$, $r=p(1-\theta) +q\theta$ and~$u\in L^p(\R, \R^n)\cap L^q(\R, \R^n)$, then
\begin{equation}\label{INTERP1}
\|u\|^r_{L^r(\R, \R^n)} \le \|u\|^{p(1-\theta)}_{L^p(\R, \R^n)} \, \|u\|^{q\theta}_{L^q(\R, \R^n)}.
\end{equation}
\end{proposition}

\begin{proof}
By using the H\"older inequality with conjugate exponents~$1/(1-\theta)$ and~$1/\theta$, we obtain that
\[
\begin{split}
\left\Vert |u|^{\frac{p(1-\theta)}{r}} \, |v|^{\frac{\theta q}{r}} \right\Vert^r_{L^r(\R, \R^n)} =&\int_{\R} |u|^{p(1-\theta)} \, |v|^{q\theta}\, dx\\
&\le \left(\int_\R |u|^p \,dx \right)^{1-\theta} \left(\int_\R |v|^q \,dx\right)^\theta\\
&= \|u\|^{p(1-\theta)}_{L^p(\R, \R^n)} \, \|v\|^{q\theta}_{L^q(\R, \R^n)},
\end{split}
\]
as desired.

Moreover, if~$u\in L^p(\R, \R^n)\cap L^q(\R, \R^n)$, one is allowed to take~$u=v$. With this choice, if~$1\le p < q<+\infty$ and~$r=p(1-\theta) +q\theta$, we have that~\eqref{INTERP1} is satisfied.
\end{proof}

\end{appendix}

\section*{Acknowledgments}
All the authors are members of the Australian Mathematical Society (AustMS). CS
is member of INdAM-GNAMPA.
This work has been supported by the Australian Laureate Fellowship FL190100081 and
the Australian Future Fellowship FT230100333.
CS also acknowledges the support of the Juan de la Cierva Fellowship (grant number JDC2023-050365-I).

\vfill

\end{document}